\documentclass[11pt,letterpaper]{article}

\usepackage[utf8]{inputenc}

\usepackage{amsmath, amsthm}
\usepackage{MnSymbol} 
\usepackage{graphicx,color}
\usepackage[hyphens]{url}
\usepackage{dsfont}
\usepackage{booktabs}
\usepackage[normalem]{ulem}
\usepackage{mathtools}

\usepackage{hyperref}
\usepackage[nameinlink,capitalize]{cleveref}

\usepackage{multirow}
\usepackage{algorithm}
\usepackage[noend]{algpseudocode}
\usepackage{tikz} %

\usepackage{soul} %

\usepackage[giveninits=true, maxnames=10, style=alphabetic, defernumbers=true]{biblatex}
\AtBeginRefsection{\GenRefcontextData{sorting=ynt}}
\AtEveryCite{\localrefcontext[sorting=ynt]}

\usepackage{subcaption}

\usepackage{tabularx}
\newcolumntype{L}{>{\raggedright\arraybackslash}X}

\DeclareMathAlphabet\mathbb{U}{msb}{m}{n}

\numberwithin{equation}{section}
\newtheorem{theorem}{Theorem}[section]
\newtheorem{cor}[theorem]{Corollary}
\newtheorem{lemma}[theorem]{Lemma}
\newtheorem{remark}[theorem]{Remark}
\newtheorem{prop}[theorem]{Proposition}
\newtheorem{obs}[theorem]{Observation}

\newtheorem{defin}[theorem]{Definition}

\newcommand{\cD}{\mathcal{D}}

\newcommand{\cN}{\mathcal{N}}

\newcommand{\cP}{\mathcal{P}}

\newcommand{\cX}{\mathcal{X}}

\newcommand{\R}{\mathbb{R}}

\newcommand*{\E}{\mathbb{E}}

\newcommand*{\Otilde}{\widetilde{O}}

\newcommand*{\eps}{\varepsilon}

\DeclareMathOperator*{\argmin}{arg\,min}

\renewcommand{\le}{\leqslant}
\renewcommand{\ge}{\geqslant}
\renewcommand{\leq}{\leqslant}
\renewcommand{\geq}{\geqslant}

\providecommand{\abs}[1]{\lvert{#1}\rvert}

\providecommand{\norm}[1]{\lVert{#1}\rVert}

\newcommand{\sig}{\sigma}

\DeclareMathOperator{\KL}{\mathsf{KL}}

\newcommand{\Ren}{\mathsf{R}}

\usepackage{bbm}
\usepackage{blkarray}
\usepackage{cancel}
\usepackage{enumitem}
\usepackage{float}
\usepackage{mathrsfs}
\usepackage[framemethod=tikz]{mdframed}
\usepackage{microtype}
\usepackage{ragged2e}
\usepackage{sectsty}
\usepackage{thmtools}
\usepackage{thm-restate}
\usepackage[Symbolsmallscale]{upgreek}
\usepackage{array}

\usetikzlibrary{cd}

\hypersetup{citecolor=violet, colorlinks=true, linkcolor=blue}

\makeatletter
\patchcmd{\@counteralias}
{\@ifdefinable{c@#1}}
{\expandafter\@ifdefinable\csname c@#1\endcsname}
{}{}
\makeatother

\makeatletter
\newcommand{\opnorm}{\@ifstar\@opnorms\@opnorm}
\newcommand{\@opnorms}[1]{%
	\left|\mkern-1.5mu\left|\mkern-1.5mu\left|
	#1
	\right|\mkern-1.5mu\right|\mkern-1.5mu\right|
}
\newcommand{\@opnorm}[2][]{%
	\mathopen{#1|\mkern-1.5mu#1|\mkern-1.5mu#1|}
	#2
	\mathclose{#1|\mkern-1.5mu#1|\mkern-1.5mu#1|}
}
\makeatother

\newcommand{\PreserveBackslash}[1]{\let\temp=\\#1\let\\=\temp}
\newcolumntype{C}[1]{>{\PreserveBackslash\centering}p{#1}}

\newcommand\bs[1]{\boldsymbol{#1}}
\newcommand\mb[1]{\mathbf{#1}}

\newcommand\mc[1]{\mathcal{#1}}
\newcommand\mf[1]{\mathfrak{#1}}
\newcommand\ms[1]{\mathscr{#1}}

\newcommand\msf[1]{\mathsf{#1}}

\definecolor{MITBrown}{RGB}{164, 31, 50}

\DeclareMathOperator\ent{ent}

\DeclareMathOperator\law{law}

\newcommand{\D}{\mathrm{d}}

\newcommand{\Coup}{\ms C}

\newcommand\deq{\coloneqq}

\newcommand\mmid{\mathbin{\|}}

\newcommand{\esssup}[1]{\mathop{#1\text{-}\mathrm{ess\,sup}}}

\newcounter{dummy}
\makeatletter
\newcommand\myitem[1][]{\item[#1]\refstepcounter{dummy}\def\@currentlabel{#1}}
\makeatother

\tikzset{
    shadedNode/.style={rectangle, draw=none, fill=blue!20, inner sep=1mm}
}

\newcommand{\weak}{\mc E_{\rm weak}}
\newcommand{\strong}{\mc E_{\rm strong}}

\newcommand{\weakb}{\bar{\mc E}_{\rm weak}}
\newcommand{\strongb}{\bar{\mc E}_{\rm strong}}

\usepackage[margin=1in]{geometry}

\makeatletter
\def\blfootnote{\gdef\@thefnmark{}\@footnotetext}
\makeatother

\allowdisplaybreaks

\begin{document}

    \title{Shifted Composition III: \\ Local Error Framework for KL Divergence}

 	\author{
		Jason M.\ Altschuler\\
		UPenn\\
		\texttt{alts@upenn.edu}
		\and
		Sinho Chewi \\
		Yale \\
		\texttt{sinho.chewi@yale.edu}
	}  
	\date{}
	\maketitle

	\begin{abstract}
    Coupling arguments are a central tool for bounding the deviation between two stochastic processes, but traditionally have been limited to Wasserstein metrics.
    In this paper, we apply the shifted composition rule---an information-theoretic principle introduced in our earlier work~\cite{scr1}---in order to adapt coupling arguments to the Kullback--Leibler (KL) divergence. Our framework combine the strengths of two previously disparate approaches: local error analysis and Girsanov's theorem. Akin to the former, it yields tight bounds by incorporating the so-called weak error, and is user-friendly in that it only requires easily verified local assumptions; and akin to the latter, it yields KL divergence guarantees and applies beyond Wasserstein contractivity.
    \par We apply this framework to the problem of sampling from a target distribution $\pi$. Here, the two stochastic processes are the Langevin diffusion and an algorithmic discretization thereof. Our framework provides a unified analysis when $\pi$ is assumed to be strongly log-concave (SLC), weakly log-concave (WLC), or to satisfy a log-Sobolev inequality (LSI). Among other results, this yields KL guarantees for the randomized midpoint discretization of the Langevin diffusion. Notably, our result: (1) yields the optimal $\widetilde O(\sqrt d/\varepsilon)$ rate in the SLC and LSI settings; (2) is the first result to hold beyond the $2$-Wasserstein metric in the SLC setting; and (3) is the first result to hold in \emph{any} metric in the WLC and LSI settings.
	\end{abstract}

	\newpage
        \small
	\setcounter{tocdepth}{2}
	\tableofcontents	
	\normalsize
	\newpage

\section{Introduction}\label{sec:intro}

How can we control the deviation between two stochastic processes when measured in an information-theoretic divergence $\cD$?
In this paper, $\mc D$ is primarily taken to be the Kullback--Leibler (KL) divergence.
This series of papers introduces and develops the \emph{shifted composition rule}, an information-theoretic principle which enables approaching such questions via the introduction of a third, auxiliary process. 
The previous papers in this series~\cite{scr1,scr2} focused on the setting
\begin{align}
	\cD(\mu P^N \mmid \nu P^N)
	\label{eq:intro:problem-same}
\end{align}
in which the two processes of interest are driven by the \emph{same} Markov kernel, and 
considered applications to
the analysis and geometry of Markov semigroups. 
Among other results, this led to a simple information-theoretic proof of F.-Y.\ Wang's celebrated dimension-free Harnack inequalities~\cite{Wang1997LSINoncompact} as well as the first sharp shift-Harnack inequalities for the Langevin diffusion---results which respectively codify regularity for Kolmogorov's backward and forward equations. 

\par In this paper, we turn to the more general setting 
\begin{align}
	\cD(\mu \hat{P}^N \mmid \nu P^N)
	\label{eq:intro:problem-diff}
\end{align}
in which the two processes of interest are driven by \emph{different} Markov kernels. This is a central problem in numerical analysis
due to the standard use case in which $P$ is an idealized process (typically a stochastic differential equation) and $\hat{P}$ is an algorithmic approximation thereof (typically a discretization of the SDE). Then,~\eqref{eq:intro:problem-diff} controls the deviation of the algorithm from the idealized process when run for $N$ steps.

\par For concreteness, in this paper we focus on applications of~\eqref{eq:intro:problem-diff} to the algorithmic problem of sampling from a distribution given query access to its score function (i.e., the gradient of its log-density), a well-studied problem with diverse applications in applied mathematics, computer science, statistics, and more, see e.g., the textbooks
~\cite{robert1999monte,liu2001monte, andrieu2003introduction,chewibook}.
A canonical approach is to run a discretization $\hat{P}$ (e.g., Langevin Monte Carlo) of a diffusion $P$ (e.g., Langevin diffusion) which converges to $\pi$.
For any such algorithm, it is essential to determine its \emph{mixing time}, i.e., how many iterations $N$ the algorithm must run before its iterates are approximately distributed according to the target $\pi$.
Since $\pi = \pi P$ is stationary under $P$, this amounts to analyzing~\eqref{eq:intro:problem-diff}. In this context, the use of the KL divergence is particularly relevant in light of the celebrated interpretation of sampling as optimization of the functional $\KL(\cdot\mmid \pi)$~\cite{jordan1998variational}, and is in fact crucial for treating settings beyond strong log-concavity; see \S\ref{ssec:intro_sampling} for further discussion.

\subsection{Strengths and weaknesses of standard analysis approaches}\label{ssec:weaknesses}

\paragraph*{Girsanov analysis.} The standard approach for analyzing~\eqref{eq:intro:problem-diff} is Girsanov's theorem which---when it is applicable---provides an exact expression for the density ratio between the entire path measures of the two stochastic processes. This enables computing the divergence between these path measures, which in turn controls the divergence~\eqref{eq:intro:problem-diff} between their final iterates by the data-processing inequality.
However, as discussed below, the standard usage of the Girsanov transformation 1) lacks a user-friendly framework, 2) cannot always be applied (especially to certain ``anticipating'' discretizations), and 3) is loose when the deviation between the path measures overestimates the deviation between the final iterates. For these reasons, the successful use of Girsanov-type analyses in sampling has been case-by-case.

\paragraph*{Local error analysis.} Many of these issues can be remedied if one seeks bounds on~\eqref{eq:intro:problem-diff} in the weaker $2$-Wasserstein metric $W_2$, since then one can appeal to coupling arguments and in particular the popular framework of \emph{local error analysis}, also called \emph{mean-squared analysis}. 
We state a representative version of this framework that is suited to the setting of this paper, and we refer to~\cite{MilTre21StochNum} for a comprehensive overview.
For completeness, we provide a brief proof in \S\ref{app:mean_sq_err}.

\begin{theorem}[{Standard version of local error framework}]\label{thm:local_error}
	Let $\hat P$, $P$ be two Markov kernels over $\R^d$. Assume that for all $x,y\in\R^d$, there are jointly defined random variables $\hat X \sim \delta_x \hat P$, $X \sim \delta_x P$, $Y \sim \delta_y P$ satisfying the following four conditions:
	\begin{itemize}
		\item \underline{Weak error.} $\norm{\E \hat X - \E X} \le \mc E_{\rm weak}(x)$.
		\item \underline{Strong error.} $\norm{\hat X - X}_{L^2} \le \mc E_{\rm strong}(x)$.
		\item \underline{$W_2$-Lipschitz.} $\norm{X - Y}_{L^2} \le L\,\norm{x-y}$ for some $1/2 \le L \le 2$ (for simplicity). 
		\item \underline{Coupling.} $\norm{X - x - (Y - y)}_{L^2} \le \gamma\,\norm{x-y}$.
	\end{itemize}
	For any probability measures $\mu$ and $\nu$,
	\begin{align*}
		W_2^2(\mu \hat{P}^N, \nu P^N) \lesssim
		\begin{dcases}
			L^N\, W_2^2(\mu,\nu) + 
			\bar N^2\,(\weakb + \gamma \strongb)^2 + \bar N\,\strongb^2\,, & L \le 1\,, \\
			L^{3N}\,\Bigl[W_2^2(\mu,\nu) + \frac{(\weakb + \gamma \strongb)^2}{(L-1)^2} + \frac{\strongb^2}{L-1}\Bigr]\,, & L > 1\,,
		\end{dcases}
	\end{align*}
	where $\bar N \deq N \wedge \frac{1}{{(1-L)}_+}$, $\weakb  \deq  \max_{n < N} \|\weak\|_{L^2(\mu \hat{P}^n)}$, and $\strongb  \deq  \max_{n < N} \|\strong\|_{L^2(\mu \hat{P}^n)}$. 
\end{theorem}

This framework enjoys several key features that Girsanov-type approaches do not:

\begin{itemize}
	\item \textbf{User-friendly framework.} 
	Local error analysis is popular largely due to its simplicity. Not only does it admit an elementary proof, more importantly it enables bounding the long-time discretization error using only short-time estimates (a.k.a.\ local error estimates), which are typically simple to compute. %
	\item \textbf{Applicability.} Local error analysis applies in many settings that Girsanov's theorem does not. For example, Girsanov's theorem provides vacuous bounds when the initializations $\mu$ and $\nu$ are singular with respect to each other (e.g., $\delta_x$ and $\delta_y$ where $x \neq y$). Another example is non-adapted stochastic processes---an
	essential feature of recent breakthroughs in sampling algorithms that improve their bias by performing a ``look-ahead step'' (see \S\ref{sec:rmd} for details). To illustrate this failure, consider even the simple setting
	\begin{align}
		\delta_x P  = \msf{Law}(x + Z) \qquad \text{ and } \qquad \delta_x \hat{P} = \msf{Law}(x + Z + \xi(x,Z))\,, 
		\label{eq:intro:ex-applicable}
	\end{align}
	where $\xi$ is a random variable depending on the state $x$ and the noise $Z \sim \cN(0,1)$. 
	Because of the dependency between $\xi$ and $Z$, Girsanov's theorem cannot be applied to bound the deviation $\KL(\delta_x \hat P \mmid \delta_x P)$ even for $N=1$ step from the same Dirac initializations. Yet local error analysis readily applies. (Simply apply Theorem~\ref{thm:local_error} by setting the weak error to be the mean of $\xi$, the strong error to be the $L^2$ norm of $\xi$, $L=1$, and $\gamma=0$.)
	\item \textbf{Tighter bounds by incorporating weak error.} On the quantitative side, a key advantage of local error analysis is that it makes use of the fact that the weak error is smaller than the strong error (by Jensen's inequality), sometimes considerably so. This is a cornerstone of SDE discretization~\cite{MilTre21StochNum} and in the context of sampling, which is the main application we consider, this underpins several recent advances in the design and analysis of algorithms (e.g.,~\cite{Li+19RungeKutta, shen2019randomized, foslyoobe21shiftedode, Li+22Mirror, LiZhaTao22LMCSqrtd}).
    However, this cannot be captured by Girsanov's theorem. 
	For intuition, consider the setting of~\eqref{eq:intro:ex-applicable} and---in order for Girsanov's theorem to apply---let us further simplify $\xi(x,Z) \sim \cN(w,\sig^2)$ to be independent of $x,Z$, and also consider identical initializations $\mu = \nu$. Then,
	\begin{align}\label{eq:intro:ex-tighter}
	\delta_x P = \delta_x \ast \cN(0,1) \qquad \text{ and } \qquad  \delta_x \hat{P} = (\delta_x P)\ast \cN(w,\sig^2)\,.
	\end{align}
	The additional convolution in $\hat{P}$ can be thought of as discretization error.
    What is the answer to~\eqref{eq:intro:problem-diff} here? Let us consider $\max(w,1) \ll \sig$ to illustrate the difference between weak and strong error.
    Applying Theorem~\ref{thm:local_error} with $\mc E_{\rm weak}(x) = w$, $\mc E_{\rm strong}(x) 
    \asymp \sig$, $L=1$, and $\gamma = 0$, yields a tight Wasserstein bound of
	\begin{align}\label{eq:intro:ex-tighter-2}
	W_2(\mu \hat{P}^N, \mu P^N) \asymp N w + \sqrt{N} \sig\,.
	\end{align}
	The key feature here is that the bias $w$ accumulates $N$ times, but the standard deviation only accumulates $\sqrt{N}$ times. This is because the total ``discretization error'' $\sum_{i=1}^N \xi_i \sim \cN(Nw, N\sig^2)$ is of size roughly $Nw + \sqrt{N} \sigma$. Here the first term represents systematic discretization bias that accumulates over all $N$ iterations, and the second term represents stochastic fluctuations which cancel out by the central limit theorem.
    However, Girsanov's theorem cannot take advantage of these cancellations: it can only provide the loose bound
	\begin{align}\label{eq:intro:ex-tighter-girsanov}
	\KL(\mu \hat{P}^N \mmid \mu P^N) \asymp N\,(w^2 + \sig^2)\,,
	\end{align}
	which accumulates $w^2$ and $\sig^2$ in equal portions, in contrast to the truth $\KL(\mu \hat{P}^N \mmid \mu P^N) \asymp  N w^2 + \sig^2$ (details in \S\ref{ssec:first_examples}). 
	Our proposed framework remedies this and all the other aforementioned issues.

\end{itemize}

On the other hand, Girsanov's theorem enjoys several key features that local error analysis does not. An obvious one is that local error analysis is currently limited to bounds in $W_2$, whereas Girsanov's theorem provides bounds in stronger metrics like $\KL$. A more subtle difference is:

\begin{itemize}
	\item \textbf{Beyond $W_2$-contractivity.} If the $W_2$-Lipschitz parameter $L > 1$, then local error analysis blows up exponentially as $L^{\Theta(N)}$. This is unavoidable for any analysis based only on Wasserstein distance. However, Girsanov's theorem often yields bounds on the discretization error scaling as $\Theta(N)$, which can then be combined with any convergence result for $P$.
    For example, the Langevin diffusion $P$ is classically known to converge to the target distribution in KL divergence even when it is $W_2$-Lipschitz with parameter $L \not < 1$, so long as $\pi$ is weakly log-concave or satisfies an isoperimetric inequality. In this and other settings, Girsanov's theorem provides bounds that do not grow exponentially in the number of iterations (e.g.,~\cite{chenetal2023diffusionmodels, Ben+24Diffusion, Che+24LMC, ConDurSil24Diffusion}).
\end{itemize}

In summary, Girsanov's theorem and local error analysis are complementary in the sense that they have distinct advantages and disadvantages. Is it possible to develop an analysis framework that combines the best of both worlds?

\subsection{Local error framework for KL divergence}

In this paper, we develop such a framework via the shifted composition rule.
In particular, it is user-friendly in that one need only check local properties of $\hat{P}$, it is broadly applicable at a similar level to standard local-error analysis (as we illustrate by  resolving open questions about sampling algorithms), it incorporates weak error as well as strong error, it yields guarantees in KL divergence, and it extends beyond settings where $P$ is $W_2$-contractive. This framework is summarized in the following theorem.

\begin{theorem}[KL local error framework]\label{thm:kl-general}
	Let $\hat P$, $P$ be two Markov kernels on $\R^d$.
	Assume that for all $x,y\in\R^d$, there are jointly defined random variables $\hat X \sim \delta_x \hat P$, $X \sim \delta_x P$, $Y \sim \delta_y P$ satisfying the following conditions:
	\begin{enumerate}
		\item \underline{$W_2$ local error assumptions.} The assumptions in Theorem~\ref{thm:local_error} hold.
		\item \underline{Regularity.} $\KL(\delta_x P \mmid \delta_y P) \le c\,\norm{x-y}^2$.
		\item \underline{Cross-regularity.} $\KL(\delta_x \hat P \mmid \delta_y P) \le c'\,\norm{x-y}^2 + b(x)^2$.
	\end{enumerate}
    For any probability measures $\mu$ and $\nu$,
	\begin{align*}
		&\KL(\mu \hat P^N \mmid \nu P^N) \\
		&\qquad \lesssim (c+c')\,\Bigl[\frac{L^{-1}-1}{L^{-N}-1}\,W_2^2(\mu,\nu) + ((L-1)\,N \vee \log \bar N)\, \bar{\mc E}_{\rm strong}^2 + \bar N\,(\bar{\mc E}_{\rm weak} + \gamma \bar{\mc E}_{\rm strong})^2\Bigr] + \bar b^2\,.
	\end{align*}
    where $\bar N$, $\bar{\mc E}_{\rm weak}$, and $\bar{\mc E}_{\rm strong}$ are as defined in Theorem~\ref{thm:local_error}, and we set $\bar{b} \deq \max_{n < N} \|b\|_{L^2(\mu \hat{P}^n)}$.
\end{theorem}

We make several remarks about this theorem. 

\paragraph*{Assumptions.} The assumptions are easy to check since they only involve short-time error estimates. The only additional assumptions over the standard version of the local error framework are regularity and cross-regularity, which are needed to obtain KL guarantees. In the standard use case where $P$ is an idealized process and $\hat{P}$ is an approximation thereof, regularity is typically easy to verify (e.g., for the Langevin diffusion $P$ this is a classical reverse transport inequality, see \S\ref{sec:sampling}). Cross-regularity is the only non-standard assumption to check, but is necessary for such a result (even for $N=1$) and we provide techniques to check it (see Remark~\ref{rmk:last_step_hack} and \S\ref{ssec:shifted_girsanov}).

\paragraph*{Interpretation for $L \le 1$.}
As an illustrative example, in \S\ref{ssec:intro_sampling} below, we consider applications to sampling in which $P$ is the Langevin diffusion (LD) run for time $h$, $\hat P$ is some discretization thereof, and the target distribution $\pi$ is well-conditioned, i.e., $\beta I \succeq \nabla^2 V \succeq \alpha I \succ 0$. By well-known properties of LD (see \S\ref{sec:sampling}), $L = \exp(-\alpha h)$ and the term involving $\gamma$ is typically not dominant. Hence, the standard local error analysis for $W_2$ (Theorem~\ref{thm:local_error}), combined with stationarity $\pi = \pi P^N$, yields
 \begin{align}
	 	W_2^2(\mu \hat{P}^N, \pi)
	 	\lesssim \exp(-\alpha Nh)\,W_2^2(\mu,\nu) + \bar N^2\,\bar{\mc E}_{\rm weak}^2 + \bar N\,\bar{\mc E}_{\rm strong}^2\,,
	 	\label{eq:intro:sampling-ex:local}
	 \end{align}
 where $\bar N \deq \frac{1}{\alpha h} \wedge N$ is the effective time horizon arising from the contraction factor $L = \exp(-\alpha h)$.
\par For this example, $c\asymp 1/h$, and since $c'$, $b$ are only used for the last iteration, they are typically negligible (see \S\ref{sec:multi} for discussion).
Thus, the framework in Theorem~\ref{thm:kl-general} essentially yields
 \begin{align}
	 	\KL(\mu \hat{P}^N \mmid \pi)
	 	&\lesssim \alpha \,\biggl[ \frac{W_2^2(\mu,\nu)}{\exp(\alpha Nh)-1} + \frac{\bar N}{\alpha h}\,\bar{\mc E}_{\rm weak}^2 + \frac{\log\bar N}{\alpha h}\,\bar{\mc E}_{\rm strong}^2 \biggr]\,.
	 	\label{eq:intro:sampling-ex:kl}
	 \end{align}
In the strongly convex case $\alpha > 0$ (i.e., $L < 1$), since $\KL(\cdot \mmid \pi) \ge \frac{\alpha}{2}\,W_2^2(\cdot, \pi)$ by Talagrand's inequality, it is readily seen that~\eqref{eq:intro:sampling-ex:kl} is stronger than~\eqref{eq:intro:sampling-ex:local}.
Moreover, in the weakly convex case $\alpha = 0$ (i.e., $L = 1$),~\eqref{eq:intro:sampling-ex:kl} still makes sense since one can formally interpret $\frac{\alpha}{\exp(\alpha Nh)-1}$ as its limit $\frac{1}{N}$.
Importantly, this coefficient tends to $0$ as $N \to \infty$, which yields convergent bounds for the weakly log-concave case---unlike~\eqref{eq:intro:sampling-ex:local}.

\paragraph*{Interpretation for $L > 1$.}
Even when $L > 1$, whereas the error terms in Theorem~\ref{thm:local_error} blow up exponentially as $L^{\Theta(N)}$ with the number of iterations $N$, the corresponding error terms in Theorem~\ref{thm:kl-general} only grow as $\Theta(N)$ (note that $\frac{L^{-1}-1}{L^{-N}-1} \to \frac{L-1}{L}$ as $N\to\infty$).
In this case, Theorem~\ref{thm:kl-general} furnishes useful discretization bounds, which we use to study non-log-concave sampling.

\subsection{Application to sampling}\label{ssec:intro_sampling}

We apply our framework to the algorithmic problem of sampling from a distribution $\pi \propto \exp(-V)$ on $\R^d$ given query access to $\nabla V$. This is a well-studied problem in its own right with celebrated applications to Bayesian statistics, machine learning, numerical integration, and more, see for example the textbooks~\cite{robert1999monte, liu2001monte, andrieu2003introduction,chewibook}. A canonical approach is to discretize the Langevin diffusion (LD), which is the SDE 
\begin{align*}
	\D Y_t = - \nabla V(Y_t)\, \D t + \sqrt{2}\, \D B_t\,,
\end{align*}
driven by a standard Brownian motion $\{B_t\}_{t \geq 0}$ on $\R^d$. It is classically known that LD converges to $\pi$ under mild conditions~\cite{bhattacharya1978criteria}, yet LD is not directly implementable due to its continuous-time nature. Let $P$ denote the Markov kernel corresponding to running LD for time step $h > 0$, and let $\hat{P}$ denote some discretization thereof---for example, the Langevin Monte Carlo (LMC) is the Euler--Maryuma discretization
\begin{align*}
	X_{(n+1)h} = X_{nh} - h\, \nabla V(X_{nh}) + \sqrt{2} \,(B_{(n+1)h} - B_{nh})\,.
\end{align*}
The key question in sampling is: how many iterations $N$ must the algorithm run before the law of its iterates approximates the target $\pi$ to some prescribed error $\eps$? This iteration complexity amounts to~\eqref{eq:intro:problem-diff}, the motivating problem of this paper. 

\paragraph*{Optimization under strong convexity, weak convexity, and non-convexity.}
Modern research in sampling is spurred by its intimate connection with the theory of optimization.
In that field, convergence rates are typically obtained under three types of assumptions on the objective function: (1) strong convexity; (2) weak convexity (i.e., the Hessian is positive semi-definite but not bounded away from zero); and (3) a Polyak{--}\L{}ojasiewicz (P\L{}) inequality, which has emerged as a tractable condition for non-convex optimization~\cite{Loj1963Top, Pol1963Gradient, KarNutSch16PL}.

\paragraph*{Sampling under strong log-concavity, weak log-concavity, and non-log-concavity.}
Due to a beautiful connection put forth in~\cite{jordan1998variational}, which interprets the Langevin diffusion as a gradient flow of the functional $\KL(\cdot\mmid \pi)$ over the Wasserstein space, we now know that the above picture for optimization has a corresponding analog for sampling.
Namely, strong convexity, weak convexity, and the PL condition for $\KL(\cdot\mmid\pi)$ correspond, respectively, to strong log-concavity (SLC), weak log-concavity (WLC), and a log-Sobolev inequality (LSI) for $\pi$.
Under these three assumptions, convergence for LD with quantative rates is standard.

Unfortunately, the story becomes muddled once we take into account discretization, which is crucial for implementation.
In the SLC case, discretization guarantees abound---at least in $W_2$---thanks to the ubiquity of coupling arguments and local error analysis.
But in the WLC and LSI cases, thus far there has not been a systematic method to directly establish convergence of the discrete-time scheme.
An alternative approach is to leverage the convergence of LD under WLC/LSI and to separately control the discretization error via Girsanov's theorem, but this runs into the issues discussed in \S\ref{ssec:weaknesses}.
Here, we fill this gap with Theorem~\ref{thm:kl-general}, which provides a unified framework for all three settings.

\begin{table}[]
    \small
    \centering
    \begin{tabular}{cccl}
         \textbf{Algorithm} & \textbf{Assumptions} & \textbf{Rate} & \textbf{Reference} \\ \hline
         \multirow{3}{*}{LMC} & SLC & $d/\varepsilon^2$ & Theorem~\ref{thm:lmc} \\
         & WLC & $(d/\varepsilon^2)\,(W/\varepsilon)^4$ & Theorem~\ref{thm:lmc} ($\dagger\dagger$) \\
         & LSI & $d/\varepsilon^2$ & Theorem~\ref{thm:lmc_lsi} \\ \hline
         \multirow{3}{*}{LMC $+$ $3^{\rm rd}$-order smoothness} & SLC & $(d/\varepsilon^2)^{1/2}$ & Theorem~\ref{thm:lmc_smooth} ($\dagger$) \\
         & WLC & $(W/\varepsilon)^4 + (d/\varepsilon^2)^{1/2}\, (W/\varepsilon)^3$ & Theorem~\ref{thm:lmc_smooth} ($\dagger\dagger$) \\
         & LSI & $(d/\varepsilon^2)^{1/2}$ & Theorem~\ref{thm:lmc_smooth_lsi} ($\dagger\dagger$) \\ \hline
         \multirow{3}{*}{RMD} & SLC & $(d/\varepsilon^2)^{1/2}$ & Theorem~\ref{thm:rmlmc} ($\dagger$) \\
         & WLC & $(d/\varepsilon^2)^{1/3}\,(W/\varepsilon)^{8/3}$ & Theorem~\ref{thm:rmlmc} ($\dagger\dagger$) \\
         & LSI & $(d/\varepsilon^2)^{1/2}$ & Theorem~\ref{thm:rmlmc_lsi} ($\dagger\dagger$)
    \end{tabular}
    \caption{\footnotesize We present the rates obtained in this paper for LMC (\S\ref{sec:sampling}), LMC under a $3^{\rm rd}$-order smoothness assumption (\S\ref{sec:lmc_smooth}), and RMD (\S\ref{sec:rmd}), under the three assumptions of SLC, WLC, and LSI\@. 
    A single dagger $(\dagger)$ signifies that the rate is established here for the first time in KL divergence (prior results only held in $W_2$), and two daggers $(\dagger \dagger)$ signifies that the rate is established here for the first time in \emph{any} metric.
    A number of simplifications have been made for ease of presentation: we omit polylogarithmic factors and dependence on other problem parameters (smoothness constants, strong convexity, etc.). Here, $W \deq W_2(\hat\mu_0,\pi)$ denotes the $W_2$ distance at initialization, which is usually at least $d^{1/2}$. The $3^{\rm rd}$-order smoothness assumption we adopt for LMC is that $\norm{\nabla \Delta V} \lesssim 1 +\norm{\nabla V}$; under the weaker assumption of boundedness of $\norm{\nabla^3 V}_{\rm op}$, the rates for LMC should be multiplied by another factor of $d^{1/2}$.}\label{tab:parallels}
\end{table}

\paragraph*{Highlight: analysis of randomized midpoint discretization.}
We envision that our framework can be used to analyze many variants of LMC, including other discretizations or variants with inexact gradient access. For illustrative purposes, we consider three applications: the basic LMC discretization in \S\ref{sec:sampling}; the LMC discretization under higher-order smoothness in \S\ref{sec:lmc_smooth}; and the randomized midpoint discretization (RMD) in \S\ref{sec:rmd}. See Table~\ref{tab:parallels} for a summary of our results. For the purpose of this introduction, we focus this discussion on our results for RMD\@. 

RMD was first introduced in~\cite{shen2019randomized} and applied to the Langevin diffusion in~\cite{hebalasubramanianerdogdu2020randomizedmidpoint}.
It has been the subject of intense recent study since it substantially improves the iteration complexity of LMC from $\widetilde O(d/\varepsilon^2)$ to $\widetilde O(\sqrt{d/\varepsilon^2})$ \emph{without requiring higher-order smoothness}, and moreover it was shown in~\cite{caoluwang2021uldlowerbd} to be an \emph{optimal} discretization in a suitable sense. %

RMD is an intriguing example of a kernel for which Girsanov's theorem does not apply at all, due to the use of a ``look-ahead'' step that renders natural interpolations of the algorithm iterates to be non-adapted.
Moreover, the point of the look-ahead is to improve the weak error, thereby necessitating an analysis that takes this into account.
These difficulties have obstructed attempts to establish \emph{any} guarantees beyond the $W_2$ metric, let alone sharp ones, leading~\cite{YuDal24Parallel} to list KL guarantees for RMD as a ``highly non-trivial open problem''.

Toward this end, the recent work of~\cite{KanNag24PoiMidpt} proposed a closely related variant of RMD and exhibited a TV distance guarantee which improves over the vanilla LMC algorithm, albeit with a suboptimal rate; see the discussion in the related work section (\S\ref{ssec:related_work}) and \S\ref{sec:rmd} for details.

As an application of our framework, we fully resolve this open problem by obtaining the KL divergence guarantees for all three settings discussed above: SLC, WLC, and LSI\@.
Notably, for the original RMD algorithm, our result: (1) yields the optimal $\widetilde O(\sqrt d/\varepsilon)$ rate in the SLC and LSI settings; (2) is the first result to hold beyond the $2$-Wasserstein metric in the SLC setting; and (3) is the first result to hold in \emph{any} metric in the WLC and LSI settings.

\subsection{Techniques}

Here, we provide a brief discussion of the techniques developed in \S\ref{sec:multi} for proving Theorem~\ref{thm:kl-general}.

\paragraph*{Shifted Girsanov and cross-regularity.}
The key idea of~\cite{ArnThaWan06HarnackCurvUnbdd}, which we built upon in our earlier works~\cite{scr1, scr2}, is to bound the deviation between two processes by applying Girsanov's theorem to a third, auxiliary process which is ``shifted'' to hit one of them at a prescribed time.
This enables adapting coupling methods to bound information divergences such as KL between two copies of the \emph{same} SDE\@.
We refer to this as the ``shifted Girsanov'' technique.

A first approach to Theorem~\ref{thm:kl-general}, which requires bounding the deviation between two \emph{different} processes, is to apply the shifted Girsanov technique in which the first process is taken to be a suitable continuous-time interpolation of the $\hat P$ iterates, and the second is the idealized process.
Indeed, this approach is the natural continuous-time version of Theorem~\ref{thm:kl-general}, and we use it to establish the cross-regularity property for LMC in \S\ref{ssec:shifted_girsanov}.

However, this approach shares the three key drawbacks of Girsanov's theorem described earlier. First, it is not fully general; for example it does not apply when there is no natural continuous-time interpolation of the $\hat P$ iterates to which Girsanov's theorem applies, as is the case for RMD\@. Moreover, even when applicable, it is unclear how to incorporate weak error, and it also requires carrying out the shifted Girsanov argument anew for each application, which is ultimately at odds with our goal of developing a user-friendly framework.
It is for these reasons that we move to a discrete-time framework; see \S\ref{ssec:shifted_girsanov} for further discussion.

\paragraph*{Moving to a discrete-time framework.}
A core insight from~\cite{scr1, scr2} is that the shifted Girsanov argument of~\cite{ArnThaWan06HarnackCurvUnbdd} can be naturally formulated in discrete-time, through the use of an information-theoretic principle which we call the \emph{shifted composition rule}.
The discrete-time formulation is more general, requiring no assumption on the existence of a suitable interpolation. %
Moreover, it enables us to integrate the argument with local error analysis: starting with local assumptions on the kernel, we provide an optimized choice of shifts leading to a multi-step bound---in many cases, an optimal one---culminating in Theorem~\ref{thm:kl-general} which can then be applied off-the-shelf.

\subsection{Related work}\label{ssec:related_work}

There are multiple approaches for bounding the divergence between two stochastic processes driven by the \emph{same} Markov kernel, i.e., the setting of problem~\eqref{eq:intro:problem-same}.
One approach, popular in information theory~\cite{raginsky2016strong, polyanskiy2017strong, polyanskiy2024information}, is to determine if $P$ satisfies a non-trivial strong data processing inequality (SDPI). The basic idea is that if the SDPI constant $\eta_P  \deq  \sup_{\mu \neq \nu} \KL(\mu P \mmid \nu P) / \KL(\mu \mmid \nu)$ of $P$ is less than $1$, then one has exponentially decaying bounds of the form $\KL(\mu P^N \mmid \nu P^N) \leq \eta_P^N \KL(\mu \mmid \nu)$.

More closely related to the approach of this paper is to apply Girsanov's theorem to an auxiliary process; this dates back to the seminal work~\cite{ArnThaWan06HarnackCurvUnbdd}, and has been used to great effect for studying the analysis and geometry of Markov diffusion processes,
see~\cite{Wang12Coupling} for a survey, and~\cite{scr1,scr2} for discrete-time formulations and recent accounts.

This paper studies the more general setting~\eqref{eq:intro:problem-diff}, in which the two stochastic processes are driven by \emph{different} Markov kernels.
This is essential for the aforementioned applications such as SDE discretization and sampling.
In this setting, a recent line of work in differential privacy bounds information-theoretic divergences between two different stochastic processes using the technique of shifted divergences~\cite{pabi,AltTal22dp,AltBokTal24} and recently by using an auxiliary interpolating process~\cite{Bok24dp}.
These approaches are also closely related to the shifted composition framework of~\cite{scr1, scr2}, see the discussion therein.
However, those analyses essentially require closed-form updates for the Markov kernels and therefore apply at a significantly diminished level of generality---for example, it was not known how to analyze the bias of sampling algorithms using those techniques (an open question in~\cite{AltTal23Langevin}), which we accomplish here as a direct application of our framework (see the examples in \S\ref{sec:sampling}, \S\ref{sec:lmc_smooth}, \S\ref{sec:rmd}). %

The recent work of~\cite{KanNag24PoiMidpt} proposes a variant of the randomized midpoint method which they call the \emph{Poisson midpoint method}. Their approach controls the KL divergence between the Poisson midpoint method run with step size $h$ for $N$ steps and the LMC discretization run with step size $h/K$ for $KN$ steps, resulting in a rate of $\widetilde O(d^{3/4}/\varepsilon)$ in TV distance in the strongly log-concave and LSI settings. In comparison, the application of our framework to RMD yields the expected rate of $\widetilde O(d^{1/2}/\varepsilon)$ in KL divergence, via a straightforward computation of the local errors (\S\ref{sec:rmd}). Their work also applies to the underdamped Langevin diffusion, which we aim to study in a future work.

\section{Preliminaries}\label{sec:prelim}

First, we recall the definition of the $2$-Wasserstein distance. Let $\cP_2(\R^d)$ denote the space of probability measures on $\R^d$ with finite second moment. The $2$-Wasserstein distance between $\mu,\nu \in \cP_2(\R^d)$ is defined as
\begin{align*}
	W_2^2(\mu,\nu) \deq \inf_{\gamma \in \Coup(\mu,\nu)} \int \|x-y\|^2\,\gamma(\D x, \D y)\,,
\end{align*}
where $\Coup(\mu,\nu)$ denotes the space of joint distributions with first marginal $\mu$ and second marginal $\nu$. 

In the rest of the section, we briefly recall relevant information-theoretic preliminaries. We begin with the definition of the Kullback--Leibler (KL) divergence and its basic properties. Proofs and further background on KL can be found, e.g., in~\cite{CoverThomasBook}. Note that here and throughout, we abuse notation slightly by identifying measures with their densities. 

\begin{defin}\label{def:KL}
	The KL divergence between probability measures $\mu$ and $\nu$ is 
	\begin{align*}
		\KL(\mu \mmid \nu) \deq \int \log\bigl(\frac{\D \mu}{\D \nu}\bigr)\,\D \mu
	\end{align*}
	if $\mu \ll \nu$, and is defined to be $+\infty$ otherwise. 
\end{defin}

\begin{prop}[Basic properties of the KL divergence]\label{prop:kl_prop}
	Let $\mu,\nu$ be probability measures. 
	\begin{itemize}
		\item \underline{Positivity.} $\KL(\mu\mmid \nu) \geq 0$, with equality if and only if $\mu = \nu$. 
		\item \underline{Data-processing inequality.} For any Markov kernel $P$, it holds that $\KL(\mu P \mmid \nu P) \leq \KL(\mu \mmid \nu)$. 
		\item \underline{Convexity.} $\KL(\cdot \mmid \cdot)$ is jointly convex. 
		\item \underline{Gaussian identity.} $\KL(\cN(x,\Sigma) \mmid \cN(y, \Sigma) ) = \frac{1}{2}\, \langle x-y, \Sigma^{-1}\, (x-y)\rangle$. 
	\end{itemize}
\end{prop}

A key reason for the usefulness of the KL divergence is its \emph{chain rule}. This can be stated as follows. Let $\msf X$, $\msf Y$ be jointly defined random variables on a standard probability space $\Omega$. Let $\bs\mu$, $\bs \nu$ be two probability measures over $\Omega$, with superscripts denoting the laws of random variables under these measures. Then
\begin{align}
	\KL(\bs \mu^{\msf Y} \mmid \bs \nu^{\msf Y})
	\leq
	\KL(\bs\mu^{\msf X, \msf Y} \mmid \bs \nu^{\msf X, \msf Y})
	&= \KL(\bs\mu^{\msf X} \mmid \bs \nu^{\msf X}) +  \int \KL(\bs\mu^{\msf Y\mid \msf X=x} \mmid \bs \nu^{\msf Y\mid \msf X=x}) \, \bs \mu^{\msf X}(\D x)\,.
	\label{eq:chain-rule-kl}
\end{align}
Strictly speaking, the equality in the second step is the chain rule for the KL divergence, and the inequality in the first step is due to the data-processing inequality. We join these two inequalities here because it provides contrast to the \emph{shifted chain rule}, an upgraded version of the chain rule developed in~\cite{scr1}.

\begin{theorem}[Shifted chain rule]\label{thm:scr-kl}
	Let $\msf X$, $\msf X'$, $\msf Y$ be three jointly defined random variables on a standard probability space $\Omega$. Let $\bs\mu$, $\bs \nu$ be two probability measures over $\Omega$, with superscripts denoting the laws of random variables under these measures. Then
	\begin{align*}
		\KL(\bs\mu^{\msf Y} \mmid \bs \nu^{\msf Y})
		&\le \KL(\bs\mu^{\msf X'} \mmid \bs \nu^{\msf X}) + \inf_{\gamma \in \Coup(\bs \mu^{\msf X}, \bs \mu^{\msf X'})} \int \KL(\bs\mu^{\msf Y\mid \msf X=x} \mmid \bs \nu^{\msf Y\mid \msf X=x'}) \,\gamma(\D x, \D x')\,.
	\end{align*}
\end{theorem}

The key idea in the shifted chain rule is to introduce an auxiliary, third random variable $\msf X'$. This generalizes the original chain rule (when $\msf X' = \msf X$) and enables many new applications via different choices of $\msf X'$, as developed in this series of papers. The present paper uses this flexibility of $\msf X'$ in order to analyze the evolution of two Markov processes updating with different kernels, see \S\ref{sec:multi}. 

\par This extends from KL divergences to the more general family of R\'enyi divergences, in which case the (shifted) chain rule for KL becomes the (shifted) composition rule for R\'enyi, the namesake of this series of papers. Details in \S\ref{app:renyi}. 

A simple but useful implication of the joint convexity of the KL divergence is the following well-known lemma, which we use for upgrading our KL bounds from Dirac initializations to arbitrary initializations in a black-box manner. See~\cite{scr1} for a short proof and a discussion of how it has been used for related problems. 

\begin{lemma}[Convexity principle]\label{lem:convexity-principle-kl}
	Let $\cP$ be a Markov kernel on a Polish space $\cX$, and let $\rho$ be a measurable function on $\cX \times \cX$ such that $\KL(\delta_x P \mmid \delta_y P) \leq \rho(x,y)$ for all $x,y \in \cX$. Then for any $\mu,\nu \in \cP(\cX)$,  
	\begin{align*}
		\KL(\mu P \mmid \nu P) \leq \inf_{\gamma \in \Coup(\mu,\nu)} \int \rho(x,y)\, \gamma(\D x, \D y)\,.
	\end{align*}
\end{lemma}

\section{Framework}\label{sec:multi}

Here, we develop our framework for analyzing the deviation between two Markov processes in terms of KL divergence. As discussed in the introduction, the core challenge is that traditional Wasserstein bounds use coupling arguments and local error analyses, but neither are available for information divergences like KL\@. Our framework remedies this through the introduction of an auxiliary, third process. In order to provide intuition, we build up to this in two steps. First, in \S\ref{ssec:kl-simple} we show how to adapt coupling arguments to KL in order to establish a simplified version of Theorem~\ref{thm:kl-general} that does not incorporate weak error. Then, in \S\ref{ssec:kl-general} we prove Theorem~\ref{thm:kl-general} by showing how to perform local error analysis in KL, thereby incorporating weak error into our framework.

\subsection{KL coupling analysis: proof of Theorem~\ref{thm:kl-simple}}\label{ssec:kl-simple}

Here, we show how to perform coupling arguments---traditionally restricted to Wasserstein metrics---in KL divergence.
In order to illustrate this idea, in this subsection we prove the following simplified version of Theorem~\ref{thm:kl-general} which does not incorporate local error analysis.

\begin{theorem}[Simplified framework: KL analysis by coupling]\label{thm:kl-simple}
    Suppose that $P$, $\hat{P}$ are Markov kernels on $\R^d$ satisfying the following one-step bounds:
    \begin{enumerate}
        \item \underline{Regularity.} $\KL(\delta_x P \mmid \delta_y P) \leq c\,\|x-y\|^2$.
        \item \underline{Wasserstein bound.} $W_2(\delta_x \hat P, \delta_y P) \leq L\, \|x-y\| + a$.
        \item \underline{Cross-regularity.} $\KL(\delta_x \hat{P} \mmid \delta_y P) \leq c'\,\|x-y\|^2 + b^2$.
    \end{enumerate}
       If $L \leq 1$, then for any probability measures $\hat\mu_0$, $\nu_0$,
     \begin{align}
	     	 \KL(\hat \mu_0 \hat{P}^N \mmid \nu_0 P^N)
	     	\leq
     		\Bigl( c + (c'-c)\, \frac{1-L^2}{1-L^{2N}} \Bigr)\, \Bigl(\frac{1+L}{1-L}\Bigr)\,\frac{( a\,(1-L^{N-1}) +  L^{N-1}\, (1-L)\, 
        W_2(\hat\mu_0,\nu_0)
       ) ^2 } {1-L^{2N}} 
     	+ b^2\,.\label{eq:thm-main:kl}
     \end{align}
\end{theorem}

We remark that this result extends directly from KL divergence to R\'enyi divergence if one replaces $W_2$ by $W_{\infty}$, and in fact this extension covers standard applications of shifted divergences to differential privacy~\cite{pabi}. Details in \S\ref{app:renyi}. 

\begin{remark}[Interpretation of the bound]
    The terms involving $c' - c$ and $b^2$ are typically lower order and only arise due to the final iteration.
    Typically $1 - L \asymp h$ and $c \asymp h^{-1}$. In this case, the bounds in Theorems~\ref{thm:kl-simple} and Appendix~\ref{thm:main-renyi} are of the form $O((L^{2N}\wedge \frac{1}{N})\,W_2^2(\hat\mu_0,\nu_0) + \bar N a^2/h)$, where $\bar N = N$ for $L =1$ and $\bar N \asymp N \wedge \frac{1}{h}$ for $L < 1$.
    This is in agreement with~\eqref{eq:intro:sampling-ex:kl} when $\bar{\mc E}_{\rm weak} = \bar{\mc E}_{\rm strong} = a$.
\end{remark}

\begin{remark}[Cross-regularity]\label{rmk:last_step_hack}
    Among our assumptions, the only non-standard one is cross-regularity, which can be difficult to check depending on the kernel $\hat P$.
    Toward this end, we note that by considering a modified process in which the last step is replaced by another kernel $\hat P_{\msf{last}}$ (i.e., $\hat\mu_N = \hat\mu_0 \hat P^{N-1} \hat P_{\msf{last}}$), the bounds of Theorems~\ref{thm:kl-general} and~\ref{thm:kl-simple} still hold as stated, provided that $c'$, $b$ are taken to be the corresponding parameters for $\hat P_{\msf{last}}$.
    Thus, since we establish cross-regularity for the LMC kernel in Lemma~\ref{lem:cross_reg_lmc}, this can be used to study sampling algorithms in which we replace the last step with a step of LMC\@.
\end{remark}

We now set out to prove Theorem~\ref{thm:kl-simple}. Let us consider Dirac initializations $\hat \mu_0 = \delta_x$ and $\nu_0 = \delta_y$; the result can then be generalized to arbitrary initialization measures via a simple convexity argument (Lemma~\ref{lem:convexity-principle-kl}), as described at the end of the proof. Consider the stochastic processes
\begin{align*}
	\hat X_{n+1} \sim \hat P(\hat X_n,\cdot)\,, \qquad Y_{n+1} \sim P(Y_n,\cdot)\,,
\end{align*}
with $\hat X_0 = x$, and $Y_0 = y$. Recall that our goal is to bound the divergence between $\hat\mu_N$ and $\nu_N$, where $\hat \mu_n  \deq \law(\hat X_n) = \delta_x \hat{P}^n$ and $\nu_n  \deq  \law(Y_n) = \delta_y P^n$. Our analysis proceeds in several steps.

\subsubsection{Construction of the auxiliary process}\label{sssec:aux_process}

Our first step is to construct an auxiliary process ${\{Y_n'\}}_{n=0}^N$, which we do iteratively as follows; see also Figure~\ref{fig:aux} for an illustration. Initialize $Y_0' = y$. For each $n < N$, let $\hat X_n$, $Y_n'$ be optimally coupled w.r.t.\ the metric $W_2$. Define 
\begin{align}
    \tilde{Y}_n  \deq  Y_n' + \eta_n\, (\hat X_n - Y_n')
    \label{eq:aux-tilde}
\end{align}
for ``shift'' parameters $\eta_0, \dots, \eta_{N-1} \in [0,1]$ to be chosen below, and define
\begin{align}
    Y_{n+1}' \sim Q_n(\tilde{Y}_n, \cdot) \qquad \text{where} \qquad  Q_n  \deq  \begin{cases}
        P\,, & n < N-1\,, \\
        \hat{P}\,, & n = N-1\,.
    \end{cases} 
    \label{eq:aux-prime}
\end{align}
The difference in the final iteration ensures that the auxiliary process hits the first process at time $N$: $Y_N' = \hat X_N$. Indeed, by setting the final shift $\eta_{N-1} = 1$ (the remaining shift parameters $\eta_0, \dots, \eta_{N-1}$ are optimized below), we have $\tilde{Y}_{N-1} = \hat X_{N-1}$, and thus the use of the kernel $\hat{P}$ rather than $P$ in~\eqref{eq:aux-prime} ensures that $Y_N' = \hat X_N$.
For easy recall, we record this interpolation in the following observation. 

\begin{figure}
	\centering
    \includegraphics[width=0.7\linewidth]{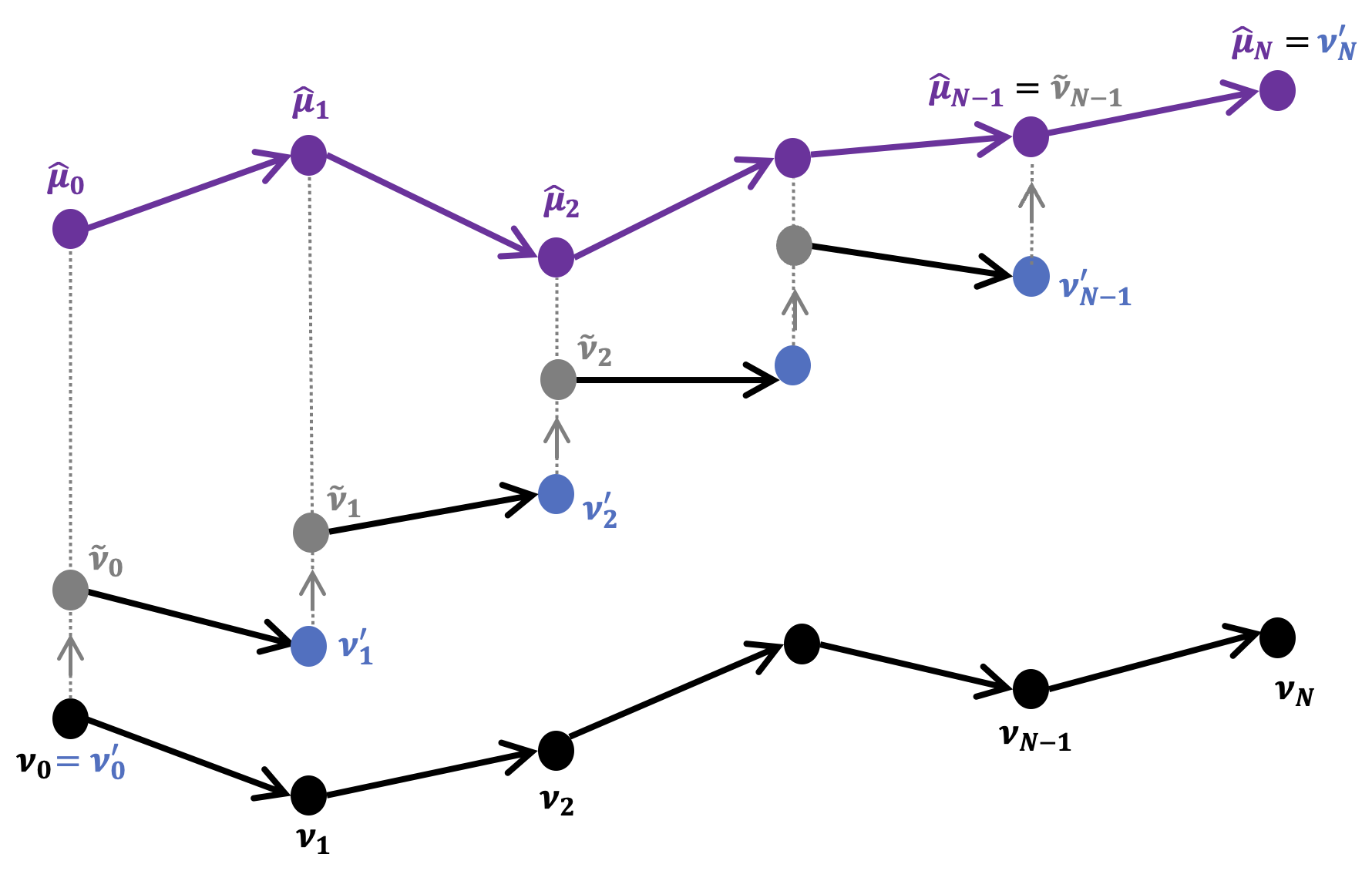}
	\caption{\footnotesize Illustration of the auxiliary process constructed in \S\ref{ssec:kl-simple}. This paper develops a framework to bound the divergence between $\{\hat{\mu}_n = \delta_x \hat{P}^n\}$ (top stochastic process) and $\{\nu_n = \delta_y P^n\}$ (bottom stochastic process). These processes differ both in that they have different initializations $x$ and $y$, and update via different Markov kernels $\hat{P}$ (purple) and $P$ (black), respectively. The auxiliary process $\{\nu_n'\}$ (blue) is constructed to interpolate between one process at initialization ($\nu_0' = \nu_0$) and the other at termination ($\nu_N' = \hat{\mu}_N$). Its update consists of two parts.	First, $\nu_n'$ is shifted along the Wasserstein geodesic towards $\hat{\mu}_n$ (vertical dotted line) to produce $\hat{\nu}_n$; this brings the process closer to the interpolation criteria at termination. Second, $\nu_{n+1}'$ is produced from $\tilde{\nu}_n$ by applying the kernel $P$ (black arrow), except in the last termination where $\hat{P}$ is used to ensure the termination criterion.
	}
	\label{fig:aux}
\end{figure}

\begin{obs}[Interpolation of the auxiliary process]\label{obs:aux-interp}
    Provided we choose $\eta_{N-1} = 1$, the auxiliary process $\{Y_n'\}_{n=0}^N$ satisfies $Y_0' = Y_0 = y$ and $Y_N' = \hat X_N$. 
\end{obs}

Let $\nu_n'$ and $\tilde{\nu}_n$ respectively denote the laws of $Y_n'$ and $\tilde{Y}_n$. 

\subsubsection{Evolution of the auxiliary process}\label{sssec:evolution}

We establish two key properties of the evolution of the auxiliary process $\nu_n'$. First and simpler, we bound the distance $d_n  \deq W_2(\hat \mu_n,\nu_n')$ between the auxiliary process $\nu_n'$ and the process $\hat{\mu}_n$ it needs to hit at termination. Note that we ignore the final iteration since $d_{N} = 0$ by Observation~\ref{obs:aux-interp}. 

\begin{lemma}[Distance recursion for the auxiliary process]\label{lem:aux-dist}
	For all $n < N-2$,
	\[
	d_{n+1} \leq L\, (1 - \eta_n)\, d_n + a\,.
	\]
\end{lemma}
\begin{proof}
	We bound
	\begin{align*}
		d_{n+1}
		&= W_2(\hat\mu_{n+1},\nu_{n+1}')
		= W_2(\hat\mu_n \hat P, \tilde{\nu}_n P)
		\leq L W_2(\hat\mu_n, \tilde{\nu}_n) + a
		= L\,(1 - \eta_n)\, d_n + a\,. 
	\end{align*}
        Above, all the equalities are immediate from the definitions. The inequality follows from
        \begin{align*}
            W_2^2(\hat\mu_n \hat P, \tilde{\nu}_n P)
            \leq
            \E[W_2^2(\delta_{\hat X_n} \hat P, \delta_{\tilde Y_n} P)]
            \leq
            \E[( L\,\|\hat X_n-\tilde Y_n\| + a)^2]
		\leq (L\, W_2(\hat\mu_n, \tilde{\nu}_n) + a)^2\,,
        \end{align*}
        where here the first step is by joint convexity of $W_2$ and optimally coupling $(\hat X_n,\tilde Y_n)$ according to $W_2(\hat\mu_n, \tilde{\nu}_n)$, the second step is by the Wasserstein bias assumption in Theorem~\ref{thm:kl-simple}, and the last step is by expanding the square and using Jensen's inequality $\E[\|\hat X_n-\tilde Y_n\|] \leq W_2(\hat\mu_n, \tilde{\nu}_n)$.
\end{proof}

Second and more substantially, we control how the auxiliary process $\nu_n'$ evolves with respect to the other process $\nu_n$ in KL divergence. Note that bounding this divergence at time $N$ immediately gives the desired divergence between $\hat{\mu}_N$ and $\nu_N$, since the auxiliary process $\nu_N' = \hat{\mu}_N$ by the interpolation condition at termination. The lemma below states two upper bounds; while going forward we only use the latter bound (in terms of the shifts $\eta_n$), we deliberately also state a key intermediate bound (in terms of the one-step divergences) because bounding this is the only place where we use the regularity and cross-regularity assumptions. 

\begin{lemma}[KL divergence bound]\label{lem:aux-kl}
    \begin{align*}
        \KL(\hat{\mu}_N \mmid \nu_N)
        &\le 
        \sum_{n=0}^{N-1} \E\KL\bigl(Q_n(\tilde Y_n,\cdot) \bigm\Vert P(Y_n',\cdot)\bigr) 
        \leq c\sum_{n=0}^{N-2} \eta_n^2 d_n^2 + c' d_{N-1}^2 + b^2\,.
    \end{align*}
\end{lemma}
\begin{proof}
	The key inequality is that for all $n < N$, 
    \begin{align*}
        \KL(\nu_{n+1}' \mmid \nu_{n+1})
        \leq \KL(\nu_n' \mmid \nu_n) + \E \KL\bigl( Q_n (\tilde Y_n, \cdot) \bigm\Vert P(Y_n', \cdot)\bigr)\,.
    \end{align*}
    This is by an application of the shifted chain rule (Theorem~\ref{thm:scr-kl}) where $\bs \mu$ is the joint distribution under which $\msf X = \tilde{Y}_n \sim \tilde{\nu}_n$, $\msf X' = Y_n' \sim \nu_n'$, and $\msf Y = Y_{n+1}' \sim \nu_{n+1}'$, and $\bs \nu$ is the joint distribution under which $\msf X = Y_n \sim \nu_n$ and $\msf Y = Y_{n+1} \sim \nu_{n+1}$.

    Repeating this argument $N$ times gives
    \begin{align*}
        \KL(\nu_N' \mmid \nu_N) \leq \KL(\nu_0' \mmid \nu_0) + \sum_{n=0}^{N-1} \E \KL\bigl( Q_n(\tilde Y_n, \cdot) \bigm\Vert P(Y_n', \cdot)\bigr)\,.
    \end{align*}
    By the interpolation property of the auxiliary process at termination $\nu_N' = \hat\mu_N$ and  initialization $\nu_0' = \nu_0$, we conclude the first inequality in the lemma statement. 
    
    \par For the second inequality, observe that for $n < N-1$ (a.k.a.\ when $Q_n = P$), the regularity assumption in Theorem~\ref{thm:kl-simple} combined with the convexity principle in Lemma~\ref{lem:convexity-principle-kl} gives
    \begin{align*}
        \E \KL\bigl(Q_n(\tilde{Y}_n, \cdot) \bigm\Vert P(Y_n', \cdot)\bigr)
        \leq c\, \E[ \|\tilde{Y}_n - Y_n'\|^2]
        \leq c \eta_n^2\, \E[ \|\hat X_n - Y_n'\|^2]
        = c \eta_n^2 d_n^2\,.
    \end{align*}
    Similarly, when $n = N-1$ (a.k.a.\ when $Q_n = \hat{P}$), the cross-regularity assumption gives
    \begin{align*}
        \E \KL\bigl(Q_{N-1}(\tilde{Y}_{N-1}, \cdot) \bigm\Vert P(Y_{N-1}', \cdot)\bigr)
        \leq c' \eta_{N-1}^2 d_{N-1}^2 + b^2\,.
    \end{align*}
    Combining the above displays and using $\eta_{N-1} = 1$ completes the proof.
\end{proof}

\subsubsection{Shift optimization}\label{sssec:shift_opt}

Combining Lemmas~\ref{lem:aux-dist} and~\ref{lem:aux-kl} gives the following KL bound between the two processes:
\begin{align}
    \KL(\delta_x \hat{P}^N \mmid \delta_y P^N)
    \leq 
    \inf_{\substack{\eta_0, \dots, \eta_{N-2} \in [0,1] \\ \text{s.t. } d_{0} = \|x-y\| \\ d_{n+1} = L(1 - \eta_n) d_n + a, \;\; \forall n \in \{0, \dots, N-2\}}}
    \underbrace{c \sum_{n=0}^{N-2} \eta_n^2 d_n^2}_{\text{\emph{main term}}} + \underbrace{c' d_{N-1}^2 + b^2}_{\text{\emph{final iterate regularity}}}\,.
    \label{eq:shift-opt-disc:opt}
 \end{align}
 It remains to optimize the shifts $\eta_0, \dots, \eta_{N-2}$. This is similar to the type of shift optimization done for the shifted divergence technique in differential privacy (see, e.g.,~\cite[Lemma C.6]{Bok24dp}), and 
 is achieved by the following lemma for the case $L=1$. The case of general $L$ is conceptually identical but requires more involved calculations, hence we focus on $L=1$ here to illustrate the conceptual ideas and defer the general $L$ case to \S\ref{app:shifts-optimal}.

\begin{lemma}[Shift optimization for $L=1$]\label{lem:shift-opt}
    The optimization problem in the right-hand side of~\eqref{eq:shift-opt-disc:opt} has value at most
    \begin{align}\label{eq:opt_shift_val_simplified}
        ((N-1)\,c + c')\,\frac{{(\|x-y\| + (N-1)\,a)}^2}{N^2} + b^2\,.
    \end{align}
\end{lemma}

The idea behind Lemma~\ref{lem:shift-opt} is that we consider a slightly simplified shift optimization problem where $c'$ is replaced by $c$. (In the language of our framework, this amounts to replacing the cross-regularity bound in the final iterate with the regularity bound.) The point is that this simplified problem admits an explicit closed-form solution, and by plugging in these shifts to the original shift optimization problem, we obtain the upper bound in Lemma~\ref{lem:shift-opt}. We remark that although this does not provide an exact solution to the original shift optimization problem (since the optimal solution to the original and simplified problems can differ), it provides a nearly-optimal solution since $c'$ affects only one of the $N$ iterations. 

\par Explicitly, the key step in the proof of Lemma~\ref{lem:shift-opt} is the following lemma which solves in closed form the aforementioned simplified version of the shift optimization problem via a dynamic programming argument.

\begin{lemma}[Closed-form expression for $R_N$]\label{lem:shift-opt-explicit:L1}
    For any $N \geq 0$ and any $a,d_0 \geq 0$, define
    \begin{align}
        R_N(a,d_0)  \deq  \inf_{\substack{\eta_0,\dots,\eta_{N} \in [0,1] \\ \text{s.t. } \eta_N = 1 \\ d_{n+1} \leq (1 - \eta_n) d_n + a, \; \forall n < N}} \sum_{n=0}^N \eta_n^2 d_n^2\,.\label{eq:shift-opt-L1-explicit}
    \end{align}
    Then
    \begin{align}
        R_N(a,d_0) = \begin{cases}
            \frac{(d_0 + Na)^2}{N+1}\,, & d_0 \geq a\,, \\
            d_0^2 + Na^2\,, & d_0 < a \,.
        \end{cases}
    \end{align}
\end{lemma}
\begin{proof}
    We prove by induction on $N$. The base case $N=0$ is trivial. For the induction, express
    \[
        R_{N}(a,d_0)
        = \min_{\eta_0 \in [0,1]}\bigl\{\eta_0^2 d_0^2 + R_{N-1}(a,\,(1-\eta_0)\,d_0 + a) \bigr\}
        =
        \min_{\eta_0 \in [0,1]}\Bigl\{\eta_0^2 d_0^2 + \frac{((1-\eta_0)\, d_0 + Na)^2}{N}\Bigr\}
        \,.
    \]
    Above, the first step is by splitting the joint optimization over $\eta_0, \dots, \eta_{N}$ in $R_{N}$ into two layers of optimization: first over $\eta_0$, and then over $\eta_1, \dots, \eta_{N}$. The second step is by the inductive hypothesis and the trivial observation that $d_1 = (1 - \eta_0)\, d_0 + a \geq a$. The remaining optimization over $\eta_0$ is now solvable in closed form by appealing to the following calculation. 
\end{proof}

\begin{obs}\label{obs:shift-opt-induct:L1}
    For any $a,d,N > 0$, the optimization problem
    \begin{align*}
        \min_{\eta \in [0,1]}\Bigl\{\eta^2 d^2 + \frac{((1-\eta)\,d + Na)^2}{N}\Bigr\}
    \end{align*}
    has unique optimal solution
    \begin{align*}
        \eta = 1\wedge \frac{d+Na}{(N+1)\,d}
        =
        \begin{cases}
            \frac{d+Na}{(N+1)\,d}\,, & d \geq a\,, \\
            1\,, & d < a\,,
        \end{cases}
    \end{align*}
    and optimal value
    \begin{align*}
        \begin{cases}
            \frac{(d+Na)^2}{N+1}\,, & d \geq a\,, \\
            d^2 + Na^2\,, & d < a\,.
        \end{cases}
    \end{align*}
\end{obs}
\begin{proof}
    This is a straightforward calculation using single-variable calculus.
\end{proof}

\begin{remark}\label{rem:shift-opt-explicit:L1}
    The optimal shifts implicit in the proof of Lemma~\ref{lem:shift-opt-explicit:L1} are $\eta_N = 1$ and 
    \begin{align*}
        \eta_n = \frac{d_n+(N-n)\,a}{(N+1-n)\,d_n}\,, \qquad \forall n \in \{0, \dots, N-1\}\,,
    \end{align*}
    where $d_{n+1} = (1 - \eta_n)\,d_n + a$. By substituting in these values and simplifying, one obtains
    \begin{align*}
        \eta_n &= 1\wedge 
        \frac{Na + d_0}{na + (N+1-n)\,d_0}\,, \qquad \forall n \in \{0, \dots, N-1\}\,,
        \\[0.25em]
        d_n &= a\vee \frac{na+(N+1-n)\,d_0}{N+1}\,, \qquad \forall n \in \{1, \dots, N\}\,.
    \end{align*}
    Note that if $d_0 \geq a$, then each ``min'' and ``max'' simplifies to just the latter argument.
\end{remark}

Now that we have the optimal shifts for the simplified shift optimization problem (Lemma~\ref{lem:shift-opt-explicit:L1}), we plug these shifts into the original shift optimization problem to prove Lemma~\ref{lem:shift-opt}. 

\begin{proof}[Proof of Lemma~\ref{lem:shift-opt}]
    This optimization problem~\eqref{eq:shift-opt-disc:opt} differs from $c$ times the optimization problem $R_{N-1}(a,d_0)$ where $d_0 = \|x-y\|$ only in the final term in the objective: $cd_{N-1}^2$ is replaced by $c'\,d_{N-1}^2 + b^2$. Consider using the same shifts as in the simplified problem above; explicit values given in Remark~\ref{rem:shift-opt-explicit:L1}. By the above remark, $d_{N-1} = a \vee \tfrac{(N-1)\,a + d_0}{N}$. We conclude that with these shifts, the value of the optimization problem~\eqref{eq:shift-opt-disc:opt} is
    \begin{align*}
        b^2 +
        \begin{cases}
            \frac{(d_0 + (N-1)\,a)^2}{N^2}\,((N-1)\,c + c')\,, & d_0 \geq a\,,
            \\
            cd_0^2 + a^2\,((N-2)\,c + c')\,, & d_0 < a\,.
        \end{cases}
    \end{align*}
    Since this is increasing in $d_0$, to get an upper bound, it suffices to just consider the case $d_0 \geq a$ (i.e., replace $d_0$ by $a\vee d_0$).
\end{proof}

\begin{remark}[Interpretation]
    The bound in the above lemma can be interpreted as follows. By replacing $N-1 \leq N$ and using a crude bound on the square, 
    \begin{align*}
        \eqref{eq:opt_shift_val_simplified} \le 2\,(Nc + c') \,\bigl( a^2 + \frac{d_0^2}{N^2} \bigr) + b^2\,.
    \end{align*}
\end{remark}

\subsubsection{Putting the pieces together}\label{sssec:pieces}

\begin{proof}[Proof of Theorem~\ref{thm:kl-simple}]
	Construct the auxiliary process as described in \S\ref{sssec:aux_process}. By Lemmas~\ref{lem:aux-dist} and~\ref{lem:aux-kl}, we can bound the desired quantity $\KL(\hat \mu_N \mmid \nu_N)$ via the optimization problem~\eqref{eq:shift-opt-disc:opt} over the non-terminal shifts $\eta_0, \dots, \eta_{N-2} \in [0,1]$. Optimizing these shifts (Lemma~\ref{lem:shift-opt} for $L=1$, or Lemma~\ref{lem:shift-opt-L<1} for more general $L$) yields the desired bound for Dirac initializations $\hat\mu_0 = \delta_x$, $\nu_0 = \delta_y$, namely
	\begin{align*}
     	 \KL(\delta_x \hat{P}^N \mmid \delta_y P^N)
		\leq
		\Bigl( c + (c'-c)\, \frac{1-L^2}{1-L^{2N}} \Bigr)\, \Bigl(\frac{1+L}{1-L} \Bigr)\, \frac{( a\,(1-L^{N-1}) +  L^{N-1}\, (1-L) \,\|x-y\|) ^2} {1-L^{2N}} 
		+ b^2\,.
	\end{align*}
	To prove the theorem for arbitrary initializations $\hat\mu_0$, $\nu_0$, use the joint convexity of KL (Lemma~\ref{lem:convexity-principle-kl}) and expand the square in an identical way as in Lemma~\ref{lem:aux-dist}, namely  use Jensen's inequality to bound the unsquared term $\E_{(X,Y) \sim \gamma }[\|X-Y\|] \leq \E_{(X,Y) \sim \gamma}[\|X-Y\|^2]^{1/2} = W_2(\hat\mu_0,\nu_0)$ where $\gamma$ is the optimal coupling for $W_2(\hat\mu_0,\nu_0)$. 
\end{proof}

\subsection{KL local error analysis: proof of Theorem~\ref{thm:kl-general}}\label{ssec:kl-general}

We now develop the full version of our framework by describing how to incorporate local error analysis into the simplified framework (Theorem~\ref{thm:kl-simple}) that we developed above. In particular, recall that the simple one-step Wasserstein assumption $W_2(\delta_x\hat P, \delta_y P) \le L\,\norm{x-y} + a$ in Theorem~\ref{thm:kl-simple} leads to the following recursion for the auxiliary process, given in Lemma~\ref{lem:aux-dist}:
\begin{align*}
	d_{n+1}
	&\le L\,(1-\eta_n)\,d_n + a\,.
\end{align*}
In contrast, under the assumptions of weak error, strong error, and $W_2$-Lipschitzness in Theorem~\ref{thm:kl-general}, the standard local error framework in $W_2$ %
leads to the following $W_2$ recursion (see Lemma~\ref{lem:new_aux_dist}):
\begin{align*}
	d_{n+1}^2
	&\le L^2\,{(1-\eta_n)}^2 \,d_n^2 + 2\,(\bar{\mc E}_{\rm weak} + \gamma\bar{\mc E}_{\rm strong})\,(1-\eta_n)\,d_n + \bar{\mc E}_{\rm strong}^2\,,
\end{align*}
where $\bar{\mc E}_{\rm weak} \deq \max_{n=0,1,\dotsc,N-1}
\norm{\mc E_{\rm weak}}_{L^2(\hat{\mu}_n)}
$ and similarly for $\bar{\mc E}_{\rm strong}$.
This new distance recursion leads to a conceptually identical analysis as in the proof of Theorem~\ref{thm:kl-simple}, except that now that the resulting shift optimization problem is substantially more complicated. Instead of solving it exactly (as done for the simplified framework in \S\ref{sssec:shift_opt}), here we instead resort to a simpler choice of shifts which yields a good approximation to the optimal value, up to constant factors. The full details are given in \S\ref{app:general_framework}.

\subsection{Illustrative toy example}\label{ssec:first_examples}

Here, we provide details for the toy example~\eqref{eq:intro:ex-tighter} in order to illustrate how our framework combines the advantages of both Girsanov's theorem and local error analysis. Recall that in this example, $P$ and $\hat P$ are the Markov kernels
\begin{align*}
    \delta_x P = \delta_x \ast \cN(0,1) \qquad \text{and} \qquad \delta_x \hat{P} = (\delta_x P) \ast \cN(w,\sig^2) = \delta_x \ast \cN(w,1+\sig^2)\,,
\end{align*}
and we consider the setting $w \ll \min(\sig,\sig^2)$ which, by~\eqref{eq:toy_example_strong_error} below, is equivalent to assuming that the weak error is substantially smaller than the strong error.\footnote{This setting is more general than the one considered in the introduction in that the present setting does not assume that $\sig \gg 1$.} This toy example can be viewed as a warm-up to our applications in sampling in that $P$ amounts to running the Langevin diffusion with zero potential, and the additional convolution in $\hat P$ can be viewed as a discretization error in implementing $P$. We showcase how different analysis approaches bound the deviation between
\begin{align*}
    \mu P^N = \mu \ast \cN(0,N) \qquad \text{ and }\qquad \mu \hat{P}^N = \mu \ast \cN(0,N) \ast \cN(Nw, N\sig^2) = \mu \ast \cN(Nw,N(1+\sig^2))\,.
\end{align*}

\paragraph*{Bounding the deviation in $W_2$.} An explicit computation gives the following tight bound (the only inequality is by a coupling argument which is tight when $\mu$ is a Dirac measure):
\begin{align}
    W_2(\mu P^N,  \mu \hat{P}^N)
    &=
    W_2\big(\mu \ast \cN(0,N),\, \mu \ast \cN(Nw, N(1+\sig^2))\big)
    \nonumber \\ &\leq 
    W_2\big(\cN(0,N),\, \cN(Nw, N(1+\sig^2)\big)
    \nonumber\\ &= \sqrt{N^2w^2 + N\bigl( \sqrt{1 + \sig^2} - 1 \bigr)^2}
    \nonumber \\ &\asymp Nw + \sqrt{N} \min(\sig, \sig^2)\,.
    \label{eq:toy-example:W-exact}
\end{align}
Above, the penultimate step is by the well-known identity for the $2$-Wasserstein distance between Gaussians, which is classical~\cite{OlkPuk1982, KnoSmi1984}.
The key feature in~\eqref{eq:toy-example:W-exact} is that the bias $w$ accumulates $N$ times, but the standard deviation term $\min(\sigma,\sigma^2)$ only accumulates $\sqrt{N}$ times due to cancellations from the stochastic fluctuations a l\'a the central limit theorem. As mentioned in~\eqref{eq:intro:ex-tighter-2}, the standard version of local error analysis recovers this up to constants: simply apply Theorem~\ref{thm:local_error} with $L=1$, $\gamma = 0$, $\mc E_{\rm weak}(x) = w$, and 
\begin{align}\label{eq:toy_example_strong_error}
    \mc E_{\rm strong}(x) 
    = W_2(\delta_x \hat{P}, \delta_x P)
    = W_2(N(x+w,1+\sig^2),\, \cN(x,1))
    = \sqrt{ w^2 + (\sqrt{1+\sig^2} - 1)^2}
    \asymp \min(\sig,\sig^2)\,.
\end{align}

\paragraph*{Bounding the deviation in $\KL$.} An explicit calculation gives the following tight bound (the only inequality is by a coupling argument which is tight when $\mu$ is a Dirac measure):
\begin{align}
    \KL(\mu \hat P^N \mmid \mu P^N)
    &\leq
    \E_{X \sim \mu} \KL(\delta_x \hat P^N \mmid  \delta_x P^N) 
    = \KL(\delta_0 \hat P^N \mmid  \delta_0 P^N)
    \nonumber \\ &= \KL(\cN(Nw, N+N\sig^2) \mmid \cN(0, N))
    \nonumber \\ &= \frac{Nw^2 + \sig^2 - \log(1+\sig^2)}{2}
    \nonumber \\ &\asymp Nw^2 + \min(\sig^2,\sig^4)\,
    \label{eq:toy-example:KL-exact}
\end{align}
Above, we used in order: joint convexity of the KL divergence, translation invariance of the KL divergence, the definition of $P$ and $\hat{P}$, and the identity for the KL divergence between Gaussians. 
The key feature, as in the exact $W_2$ bound~\eqref{eq:toy-example:W-exact}, is that the bias $w$ accumulates more than the stochastic fluctuations $\sigma$. Indeed, comparing these bounds on equal footing would, by Talagrand's inequality, amount to comparing $\sqrt{\KL}/\sqrt{N}$ and $W_2$ in which case $\sqrt{\KL}/\sqrt{N} \asymp N w + \sqrt{N} \min(\sig,\sig^2)$ which recovers the asymptotics there. 

\par Observe that standard usages of Girsanov's theorem are \emph{not} able to obtain this benefit of weak error.
Indeed, the analog of Girsanov's theorem in this context is to use the data-processing inequality (Proposition~\ref{prop:kl_prop}) to bound the KL divergence between the joint measures $\hat{\bs\mu}$ and $\bs\mu$, where $\hat{\bs \mu} \deq \law(\hat X_0,\hat X_1,\dotsc,\hat X_N)$ and $\bs\mu \deq \law(X_0, X_1,\dotsc,X_N)$, and then to apply the KL chain rule~\eqref{eq:chain-rule-kl}:
\begin{align*}
    \KL(\mu \hat P^N \mmid \mu P^N)
    \le \KL(\hat{\bs \mu} \mmid \bs \mu)
    &= \sum_{n=0}^{N-1} \E_{\hat X_n\sim \mu\hat P^n} \KL(\delta_{\hat X_n} \hat P \mmid \delta_{\hat X_n} P) \\
    &= N \KL(\cN(0,1) \mmid \cN(w, 1+\sigma^2))
    \asymp Nw^2 + N\min(\sigma^2,\sigma^4)\,.
\end{align*}

\par On the other hand, the proposed framework for KL local error analysis (Theorem~\ref{thm:kl-general}) does benefit from the weak error. To apply this result, use the same parameters for $\mc E_{\rm weak}(x) = w$, $\mc E_{\rm strong}(x) \asymp \min(\sig,\sig^2)$, $L=1$, $\gamma = 0$ as for standard local error analysis. For the regularity, take $c = 1$ since $\KL(\delta_x P \mmid \delta_y P) = \|x-y\|^2/2$. For the cross-regularity, it suffices to take $c' = 1$ and $b^2 = w^2 + (\sig^2 - \log(1 + \sig^2))/2 \asymp w^2 + \min(\sig^2,\sig^4)$ since by the identity for the KL divergence between Gaussians and then a crude bound $(x-y+w)^2 \leq 2\,(x-y)^2 + 2w^2$, 
\[
    \KL(\delta_x \hat P \mmid \delta_y P)
    = \KL(N(x+w, 1+\sig^2) \mmid \cN(y,1))
    = \frac{(x-y+w)^2 + \sig^2 - \log(1 + \sig^2)}{2}
    \leq {(x-y)}^2 + b^2\,.
\]
Applying Theorem~\ref{thm:kl-general} then yields
\begin{align*}
    \KL(\mu P^N \mmid  \mu \hat{P}^N)
    \lesssim
    Nw^2 + (\log N)\min(\sig^2,\sig^4)
\end{align*}
which asymptotically matches the tight bound~\eqref{eq:toy-example:KL-exact} up to the logarithmic term $\log N$.

\section{Warm-up: Langevin Monte Carlo (LMC)}\label{sec:sampling}

In this and subsequent sections, we apply our techniques to obtain KL divergence guarantees for sampling from a continuous distribution $\pi$ over $\R^d$.

We begin with a warm-up by applying our framework to perhaps the simplest sampling algorithm: Langevin Monte Carlo (LMC).
Although the results we obtain in this setting are not new, this example illustrates the main features of the KL local error framework, and the proofs provide a template for the applications in \S\ref{sec:lmc_smooth} and \S\ref{sec:rmd}.

In \S\ref{ssec:shifted_girsanov}, we show how to check the cross-regularity assumption of Theorem~\ref{thm:kl-general} via the application of Girsanov's theorem to an auxiliary, \emph{shifted} process.
This argument, which we term the \emph{shifted Girsanov} argument, is the continuous-time analog of Theorem~\ref{thm:kl-simple}; see the discussion in \S\ref{ssec:shifted_girsanov}.

\subsection{Results for LMC}

Let $\pi \propto \exp(-V)$ be a probability density over $\R^d$.
The Langevin diffusion (LD) corresponding to $\pi$ is the solution to the SDE
\begin{align}\label{eq:langevin}\tag{$\msf{LD}$}
    \D Y_t = -\nabla V(Y_t) \, \D t + \sqrt 2\,\D B_t\,,
\end{align}
where ${\{B_t\}}_{t\ge 0}$ is a standard Brownian motion.
For a fixed step size $h > 0$, let $P = P_h$ denote the Markov transition kernel corresponding to~\ref{eq:langevin}, run for time $h$.

The standard Euler{--}Maruyama discretization applied to~\ref{eq:langevin} yields the Langevin Monte Carlo (LMC) algorithm:
\begin{align}\label{eq:lmc}\tag{$\msf{LMC}$}
    \hat X_{(n+1)h}
    &= \hat X_{nh} - h\,\nabla V(\hat X_{nh}) + \sqrt 2\,(B_{(n+1)h} - B_{nh})\,, \qquad n=0,1,2,\dotsc\,.
\end{align}
Let $\hat P$ denote the Markov transition kernel corresponding to a single iteration of~\ref{eq:lmc}.
As an illustration and a warm-up, we establish iteration complexity guarantees for~\ref{eq:lmc} via our framework. We state these results next, and then contextualize them with existing results in \S\ref{sec:lmc_discussion} below. 
\par The first result is for the strongly log-concave ($\alpha > 0$) and weakly log-concave ($\alpha=0$) settings. In the strongly log-concave case, we make the following assumption on the initialization for simplicity.

\begin{remark}[Strongly log-concave initialization]\label{rmk:slc_init}
    In the strongly log-concave case $0 \prec \alpha I \preceq \nabla^2 V$, we make the standard assumption that the initialization satisfies $W_2(\hat\mu_0,\pi) \le \sqrt{d/\alpha}$.
    This is satisfied by initialization at the mode: $\hat\mu_0 = \delta_{x_\star}$, $x_\star = \argmin V$ (cf.\ the ``basic lemma'',~\cite[\S 4]{chewibook}), and the mode can be computed via standard convex optimization routines and does not dominate the cost of sampling.
    Alternatively, the initialization requirement can be met by using Theorem~\ref{thm:local_error}.
    More general initializations can be considered via straightforward modifications of the proof.
\end{remark}

\begin{theorem}[LMC]\label{thm:lmc}
    Let $\pi \propto \exp(-V)$, and let $\hat P$ denote the kernel for~\ref{eq:lmc}.
    We impose the following assumption:
    \begin{itemize}
        \item $V$ is twice-continuously differentiable and satisfies $0 \preceq \alpha I \preceq \nabla^2 V \preceq \beta I$ over $\R^d$.
        Write $\kappa \deq \beta/\alpha$.
    \end{itemize}
    Then, the following statements hold.
    \begin{enumerate}
        \item \underline{$\alpha > 0$ case:} If $\varepsilon \in [0,\sqrt d]$, $h\asymp \frac{\varepsilon^2}{\beta \kappa d}$, and $W_2(\hat\mu_0,\pi) \le \sqrt{d/\alpha}$, then $\KL(\hat\mu_0\hat P^N\mmid \pi) \le\varepsilon^2$ for all
        \begin{align}\label{eq:lmc_str_cvx}
            N \gtrsim \frac{\kappa^2 d}{\varepsilon^2} \log \frac{d}{\varepsilon^2}\,.
        \end{align}
        \item \underline{$\alpha = 0$ case:} If $\varepsilon > 0$ is sufficiently small and $h\asymp \frac{\varepsilon^4}{\beta^2 dW^2}$, then $\KL(\hat\mu_0\hat P^N \mmid \pi) \le \varepsilon^2$ for
        \begin{align*}
            N \asymp \frac{\beta^2 dW^4}{\varepsilon^6}\,,
        \end{align*}
        where $W \deq W_2(\hat\mu_0,\pi)$.
    \end{enumerate}
\end{theorem}

Next, we show that our framework can also analyze LMC under the assumption that $\pi$ satisfies a log-Sobolev inequality (LSI), thereby allowing for non-convexity of the potential $V$.
We recall that $\pi$ satisfies LSI with constant $1/\alpha$ if for all compactly supported, non-negative, and smooth functions $f : \R^d\to\R$, it holds that
\begin{align}\label{eq:lsi}\tag{$\msf{LSI}$}
    \ent_\pi(f) \deq \E_\pi\bigl[f\log\frac{f}{\E_\pi f}\bigr]\le \frac{1}{2\alpha}\, \E_\pi[\norm{\nabla f}^2]\,.
\end{align}
The LSI is equivalent to exponential decay of the KL divergence for the continuous-time Langevin diffusion~\eqref{eq:langevin}, and---as was first shown in~\cite{VempalaW19}---leads to guarantees in discrete time as well.
We also recall that if $\pi \propto \exp(-V)$ with $\nabla^2 V \succeq \alpha I$, then $\pi$ satisfies an LSI with constant $1/\alpha$, but there are many examples of non-convex potentials $V$ for which $\pi$ still satisfies an LSI\@. Thus, the LSI assumption is more general than the strongly log-concave case of Theorem~\ref{thm:lmc}.
We refer to~\cite[\S 5]{bakry2014analysis} for a more thorough exposition.

\begin{remark}[LSI initialization]\label{rmk:lsi_init}
    In the LSI setting, we assume that the initialization satisfies $\log \chi^2(\hat\mu_0 \mmid \pi) = \widetilde O(d)$.
    For example, when $\pi$ is in fact strongly log-concave, then the initialization $\hat\mu_0 = \mc N(x_\star, \beta^{-1} I)$ satisfies $\log \chi^2(\hat\mu_0\mmid \pi) \le \frac{d}{2}\log \kappa$.
    In general, the assumption on the initialization is fairly mild (cf.~\cite[\S A]{Che+24LMC} for further discussion) and we make it for simplicity of the resulting bound.
    Results avoiding this assumption can extracted from the proof.
\end{remark}

\begin{theorem}[LMC under LSI]\label{thm:lmc_lsi}
    Let $\pi \propto \exp(-V)$, and let $\hat P$ denote the kernel for~\ref{eq:lmc}.
    We impose the following assumptions:
    \begin{itemize}
        \item $\pi$ satisfies~\eqref{eq:lsi} with constant $1/\alpha$.
        \item $V$ is twice-continuously differentiable and satisfies $-\beta I \preceq \nabla^2 V \preceq \beta I$ over $\R^d$.
        Write $\kappa \deq \beta/\alpha \ge 1$. 
        \item The initialization satisfies $\log \chi^2(\hat\mu_0 \mmid \pi) = \widetilde O(d)$.
    \end{itemize}
    If $\varepsilon \in [0,\sqrt d]$ and $h = \widetilde\Theta(\frac{\varepsilon^2}{\beta \kappa d})$, then $\KL(\hat\mu_0\hat P^N \mmid \pi) \le \varepsilon^2$ for
    \begin{align*}
        N = \widetilde\Theta\Bigl(\frac{\kappa^2 d}{\varepsilon^2}\Bigr)\,.
    \end{align*}
\end{theorem}

\subsection{Discussion}\label{sec:lmc_discussion}

As discussed earlier, our goal in Theorems~\ref{thm:lmc} and~\ref{thm:lmc_lsi} is simply to illustrate the KL local error framework on the most basic sampling algorithm before tackling more challenging settings in later sections (where we provide new results), but we pause to give a few comparisons with the literature.
There are three main techniques for controlling the discretization error in KL divergence.
\begin{itemize}
    \item The \textbf{Girsanov method} (see, e.g.,~\cite[\S 4]{chewibook}) directly controls the KL divergence between the discretized and idealized processes using Girsanov's theorem from stochastic calculus.
    As we show in \S\ref{ssec:shifted_girsanov}, the KL local error framework is closely related to the application of Girsanov's theorem to a carefully chosen \emph{auxiliary} process, see \S\ref{sec:intro} for a full discussion.

    \item The \textbf{interpolation method} of~\cite{VempalaW19} is based on writing a differential inequality for the KL divergence along a continuous-time interpolation of LMC\@. It also applies to the setting of \S\ref{sec:lmc_smooth}, see the discussion therein.
    However, at present, it is not known how to generally incorporate weak error into this method.
    \item The remarkable \textbf{proof via convex optimization} of~\cite{durmus2019analysis} makes deep use of the interpretation of the Langevin diffusion as a Wasserstein gradient flow~\cite{jordan1998variational}, leading to the state-of-the-art guarantees for LMC\@.
    However, this proof technique is tailored for situations admitting such a gradient flow structure and has thus far only been applied to LMC and to mirror LMC~\cite{ahn2021efficient}. In particular, it remains open if this analysis technique can be used for more general algorithms and/or in more general settings. 
\end{itemize}
We also note that the first two methods listed above typically require the initial KL divergence, $\KL(\hat \mu_0\mmid \pi)$ to be finite, prohibiting initialization at a point mass, whereas the KL local error framework exploits ``$W_2$-to-$\KL$'' regularization and hence can yield non-vacuous bounds for any initialization $\hat\mu_0$ with $W_2(\hat\mu_0,\pi) < \infty$, see Theorem~\ref{thm:lmc}.\footnote{To simplify the proofs, we do not make use of this regularization in the LSI setting.}

In the \emph{specific} context of LMC, these three methods are applicable, permitting comparison with our results.
The Girsanov approach to LMC has been fleshed out in~\cite{Che+24LMC} and can reach Theorem~\ref{thm:lmc} (strongly log-concave case), albeit with the caveats that it requires finite KL at initialization and has diverging bounds as $N \to \infty$, as well as Theorem~\ref{thm:lmc_lsi}.
We remark that~\cite{Che+24LMC} considered weaker assumptions as well, such as a Poincar\'e inequality, and such cases could also be covered via the KL local error framework although we refrain from doing so in the interest of brevity.
The interpolation method was applied to LMC in~\cite{VempalaW19, Che+24LMC}, and yields a slightly stronger version of Theorem~\ref{thm:lmc_lsi} (namely, with a bound that does not diverge as $N\to\infty$). The approach of~\cite{durmus2019analysis} yields the sharpest bounds for LMC\@: $\widetilde O(\kappa d/\varepsilon^2)$ in the strongly log-concave case and $O(\beta dW^2/\varepsilon^4)$ in the weakly log-concave case (but is inapplicable in the LSI case). These are substantially smaller than the rates obtained in Theorem~\ref{thm:lmc}.
Interestingly, however, for the weakly log-concave case, the result of~\cite{durmus2019analysis} requires averaging the iterates.
To the best of our knowledge, Theorem~\ref{thm:lmc} is actually the first result in this setting that holds for the last iterate.

\subsection{Proof of Theorem~\ref{thm:lmc}}

Theorem~\ref{thm:lmc} follows from verifying the assumptions of Theorem~\ref{thm:kl-general}.
We begin with properties of the continuous-time diffusion.

\begin{lemma}[Properties of the Langevin diffusion]\label{lem:ld_properties}
    Let $P$ denote the kernel corresponding to~\ref{eq:langevin} run for time $h$.
    Assume that $\alpha I \preceq \nabla^2 V \preceq \beta I$.
    Then, for all $x,y\in\R^d$, there is a coupling $X_h \sim \delta_x P$, $Y_h \sim\delta_y P$ such that the following hold.
    \begin{enumerate}
        \item \underline{$W_2$-Lipschitz.} $\norm{X_h-Y_h}_{L^2} \le \exp(-\alpha h)\,\norm{x-y}$.
      \item \underline{Coupling.} $\norm{X_h-Y_h-(x-y)}_{L^2} \le \frac{\beta\, (1 - \exp(-\alpha h))}{\alpha} \,\norm{x-y}$.
      \item \underline{Regularity.} $\KL(\delta_x P \mmid \delta_y P) \le \frac{\alpha}{2\,(\exp(2\alpha h) - 1)} \, \norm{x-y}^2$.
    \end{enumerate}
\end{lemma}
\begin{proof}
    The first two statements are standard, see, e.g.,~\cite[\S 5]{chewibook}.
    The last statement is a reverse transport inequality which dates back to~\cite{bobgenled2001hypercontractivity}; see also Part I~\cite{scr1} for a proof using the shifted composition rule.
\end{proof}

In particular, for the common setting of $\alpha \ge 0$ and $h \le 1/\beta$, we can take $L = \exp(-\alpha h)$, $\gamma \le \beta h$, and $c \le 1/(4h)$ in Theorem~\ref{thm:kl-general}.

The next lemma only considers the strong error for LMC\@.
Indeed, for this application, there is no gain from considering the weak error separately, because it cannot be shown to be substantially smaller than the strong error; hence, we only use the trivial bound $\mc E_{\rm weak} \le \mc E_{\rm strong}$.
Thus, for LMC, we do not require the full power of the local error framework in Theorem~\ref{thm:kl-general}. (In the following sections \S\ref{sec:lmc_smooth} and \S\ref{sec:rmd} we consider settings where the weak error is smaller than the strong error.)

\begin{lemma}[Strong error for LMC]\label{lem:lmc_strong_err}
    Let $P$, $\hat P$ denote the kernels corresponding to~\ref{eq:langevin} run for time $h$ and~\ref{eq:lmc}, respectively.
    Assume that $-\beta I \preceq \nabla^2 V \preceq \beta I$ and that $h \lesssim 1/\beta$ for a sufficiently small implied constant.
    Then, for all $x\in\R^d$, there is a coupling $\hat X_h \sim \delta_x \hat P$, $X_h \sim \delta_x P$ such that
    \begin{align*}
        \norm{\hat X_h - X_h}_{L^2}
        &\lesssim \beta h^2\,\norm{\nabla V(x)} + \beta d^{1/2} h^{3/2}\,.
    \end{align*}
\end{lemma}
\begin{proof}
    This is also a standard computation, cf.~\cite[\S 4]{chewibook}.
\end{proof}

The only non-standard assumption to check is the cross-regularity. We defer the proof of the following result to \S\ref{ssec:shifted_girsanov}.

\begin{lemma}[Cross-regularity for LMC]\label{lem:cross_reg_lmc}
    Let $P$, $\hat P$ denote the kernels corresponding to~\ref{eq:langevin} run for time $h$ and~\ref{eq:lmc}, respectively.
    Assume that $-\beta I \preceq \nabla^2 V \preceq \beta I$ and that $h \lesssim 1/\beta$ for a sufficiently small implied constant.
    Then, for all $x,y\in\R^d$,
    \begin{align*}
        \KL(\delta_x \hat P \mmid \delta_y P)
        \lesssim \frac{\norm{x-y}^2}{h} + \beta^2 h^3 \,\norm{\nabla V(x)}^2 + \beta^2 dh^2\,.
    \end{align*}
\end{lemma}

Equipped with these bounds, we can now prove Theorem~\ref{thm:lmc}.

\begin{proof}[Proof of Theorem~\ref{thm:lmc}]
    We apply Theorem~\ref{thm:kl-general} with the following parameters: $L = \exp(-\alpha h)$, $\gamma = \beta h$, $b(x) = O(\beta h^{3/2}\,\norm{\nabla V(x)} + \beta d^{1/2} h)$, $c, c' = O(1/h)$, and
    \begin{align*}
        \mc E_{\rm weak}(x) \le \mc E_{\rm strong}(x) = O\bigl(\beta h^2\,\norm{\nabla V(x)} + \beta d^{1/2} h^{3/2}\bigr)\,.
    \end{align*}
    Hence, by Theorems~\ref{thm:local_error} and~\ref{thm:kl-general}, for $G_n^2 \deq \max_{k < n} \E_{\hat \mu_0\hat P^k}[\norm{\nabla V}^2]$,
    \begin{align*}
        W_2^2(\hat \mu_0\hat P^n, \pi)
        &\lesssim \exp(-\alpha nh)\, W_2^2(\hat\mu_0,\pi) + \bigl(n\wedge \frac{1}{\alpha h}\bigr)^2\,(\beta^2 h^4 G_n^2 + \beta^2 dh^3)\,,
    \end{align*}
    and
    \begin{align}\label{eq:lmc_kl_recursion}
        \KL(\hat\mu_0 \hat P^n \mmid \pi)
        &\lesssim \frac{\alpha W_2^2(\hat\mu_0,\pi)}{\exp(\alpha nh)-1} + \bigl(n \wedge \frac{1}{\alpha h}\bigr)\,(\beta^2 h^3 G_n^2 + \beta^2 dh^2)\,.
    \end{align}
    Intuitively, the $G_n$ terms do not dominate because they carry higher powers of $h$.
    In order to rigorously deal with these terms, we provide Lemma~\ref{lem:recursive_grad} in \S\ref{app:grad}.

    In the case $\alpha > 0$, we invoke Lemma~\ref{lem:recursive_grad} with $n_0 = \infty$, $\msf A^2 = O(\kappa^2 h^2)$, $\msf B^2 = O(W_2^2(\hat\mu_0,\pi) + \kappa^2 dh) = O(d/\alpha)$, provided $h \lesssim 1/(\beta \kappa)$.
    This yields $G_n^2\lesssim \beta \kappa d$ for all $n$.
    Substituting this into~\eqref{eq:lmc_kl_recursion} completes the proof in this case.

    In the case $\alpha = 0$, we invoke Lemma~\ref{lem:recursive_grad} with the total iteration count $N = \Theta(W^2/(\varepsilon^2 h))$, $n_0 = \Theta(1/(\beta h))$, $\msf A^2 = O(h^2)$, $\msf B^2 = O(W^2)$, $\msf C^2 = O(\beta^2 h^2 W^2/\varepsilon^2)$, $\msf D^2 = O(\beta W^2 + \beta^2 dhW^2/\varepsilon^2) = O(\beta W^2)$, provided that
    \begin{align*}
        h\lesssim \frac{W^2}{d} \wedge \frac{\varepsilon}{\beta^{3/2} W} \wedge \frac{\varepsilon^2}{\beta d}\,,
    \end{align*}
    where $W \deq W_2(\hat\mu_0,\pi)$.
    This yields $G_N^2\lesssim \beta d + \beta^2 W^2$.
    The proof is completed by substituting the gradient bound into~\eqref{eq:lmc_kl_recursion} again.
\end{proof}

\subsection{Proof of Theorem~\ref{thm:lmc_lsi}}

Next, we show how to adapt the analysis to remove the assumption of (strong) log-concavity. The idea is simply that as long as the potential $V$ is \emph{$\beta$-smooth}, then regardless of (strong) convexity of $V$, the assumptions of Theorem~\ref{thm:kl-general} are still met with $L = \exp(\beta h) > 1$ (simply apply the first item of Lemma~\ref{lem:ld_properties} with $\alpha = -\beta$). 
Despite the fact that $L > 1$, the bound of Theorem~\ref{thm:kl-general} does \emph{not} grow exponentially with the number of iterations and hence produces a fairly good discretization bound, which can be combined with the convergence of the continuous-time Langevin diffusion via the LSI\@. This is in contrast with the exponential growth of standard local error analyses in $W_2$ (cf., Theorem~\ref{thm:local_error}). We begin by recalling the convergence of the Langevin diffusion under LSI\@.

\begin{theorem}[{\cite[Theorem 3]{VempalaW19}}]\label{thm:langevin_renyi_conv}
    Assume that $\pi$ satisfies~\eqref{eq:lsi} with constant $1/\alpha$.
    Then, the law ${(\nu_t)}_{t\ge 0}$ of the Langevin diffusion~\eqref{eq:langevin} with stationary distribution $\pi$ satisfies, for all $t\ge 0$,
    \begin{align*}
        \log(1+\chi^2(\nu_t \mmid \pi)) \le \exp(-\alpha t)\log(1+\chi^2(\nu_0 \mmid \pi))\,.
    \end{align*}
\end{theorem}

\begin{proof}[Proof of Theorem~\ref{thm:lmc_lsi}]
    We apply Theorem~\ref{thm:kl-general}, now with $L = \exp(\beta h)$ and $\gamma = O(\beta h)$.
    It yields
    \begin{align*}
        \KL(\hat\mu_0 \hat P^n \mmid \hat \mu_0 P^n)
        &\lesssim n\,(\beta^2 h^3 G_n^2 + \beta^2 dh^2)\,.
    \end{align*}
    By the R\'enyi weak triangle inequality (Corollary~\ref{kl:weak-triangle}) and Theorem~\ref{thm:langevin_renyi_conv},
    \begin{align*}
        \KL(\hat\mu_0\hat P^n \mmid \pi)
        &\le \log(1+\chi^2(\hat\mu_0 P^n \mmid \pi)) + 2\KL(\hat\mu_0\hat P^n \mmid \hat \mu_0 P^n) \\
        &\lesssim \exp(-\alpha nh)\,\widetilde O(d) + n\,(\beta^2 h^3 G_n^2 + \beta^2 dh^2)\,.
    \end{align*}
    We take $Nh = \widetilde \Theta(1/\alpha)$ and for $n \le N$, we apply Lemma~\ref{lem:recursive_grad} with $n_0 = 0$, $\msf C^2 = \widetilde O(\beta \kappa h^2)$, $\msf D^2 = \widetilde O(d + \beta \kappa dh) = \widetilde O(d)$, provided $h \le \widetilde O(1/(\beta \kappa))$.
    It yields $G_N^2 = \widetilde O(\beta d)$, and substituting this into the inequality above completes the proof.
\end{proof}

\subsection{Verifying cross-regularity via shifted Girsanov}\label{ssec:shifted_girsanov}

We now prove the cross-regularity lemma for LMC (Lemma~\ref{lem:cross_reg_lmc}).
It is not straightforward to bound $\KL(\delta_x \hat P \mmid \delta_y P)$, since we do not have a closed-form expression for the Langevin kernel $P$.
Note also that we cannot directly apply Girsanov's theorem to the (interpolated) LMC algorithm and the diffusion, since they are singular with respect to each other (they start at $\delta_x$, $\delta_y$ respectively), thereby resulting in a vacuous bound.
Hence, we instead apply Girsanov's theorem to an \emph{auxiliary} process which hits the LMC algorithm at time $h$.
The idea of using Girsanov's theorem to an auxiliary process goes back to~\cite{ArnThaWan06HarnackCurvUnbdd} (see the related work in \S\ref{ssec:related_work}). The main difference here is that we apply it to two \emph{different} stochastic processes.

\begin{remark}
    An alternative approach is to use the R\'enyi weak triangle inequality (Corollary~\ref{kl:weak-triangle}),
    \begin{align*}
        \KL(\delta_x \hat P \mmid \delta_y P)
        &\le 2\KL(\delta_x \hat P \mmid \delta_y \hat P) + \log(1+\chi^2(\delta_y \hat P \mmid \delta_y P))\,.
    \end{align*}
    The first term can be computed explicitly, and the second term can be controlled via a standard Girsanov argument (see~\cite{chewibook, Che+24LMC}).
    However, we do not take this route here, in order to develop the shifted Girsanov argument and also to avoid the unnecessary use of R\'enyi divergences.
\end{remark}

Consider the following stochastic processes, which are the continuous-time limits of the discrete-time processes constructed in \S\ref{sec:multi}.
\begin{align*}
    \D \hat X_t
    &= -\nabla V(x) \,\D t + \sqrt 2 \, \D B_t'\,, & \hat X_0 &= x\,, \\
    \D Y_t'
    &= \{-\nabla V(Y_t') + \eta_t\,(\hat X_t - Y_t')\}\,\D t + \sqrt 2 \,\D B_t'\,, & Y_0' &= y\,, \\
    &= -\nabla V(Y_t') \, \D t + \sqrt 2 \,\D B_t\,,
\end{align*}
where $\D B_t' = \D B_t - \frac{\eta_t}{\sqrt 2}\,(\hat X_t - Y_t')\,\D t$.
We will choose the deterministic process ${\{\eta_t\}}_{t\ge 0}$ so that $Y_h' = \hat X_h$.
This necessitates taking $\eta_t \to \infty$ as $t\nearrow h$, but the validity of the following arguments can be argued in the usual way, cf.~\cite[Remark 4.1]{scr1}.
We define the path measures $\mb P$, $\mb P'$ so that under $\mb P'$, $B'$ is a standard Brownian motion and $Y_h' = \hat X_h \sim \delta_x \hat P$, and under $\mb P$, $B$ is a standard Brownian motion and $Y_h' \sim \delta_y P$.
The data-processing inequality (Proposition~\ref{prop:kl_prop}) and Girsanov's theorem~\cite[Theorem 5.22]{legall2016stochasticcalc} imply
\begin{align*}
    \KL(\delta_x \hat P \mmid \delta_y P)
    \le \KL(\mb P' \mmid \mb P)
    &= \frac{1}{4}\, \E^{\mb P'} \int_0^h \eta_t^2\,\norm{\hat X_t - Y_t'}^2 \, \D t\,.
\end{align*}
It remains to bound the quantity on the right-hand side.

\begin{proof}[Proof of Lemma~\ref{lem:cross_reg_lmc}]
    We focus on the convex case $\nabla^2 V \succeq 0$ and we defer the non-convex case to \S\ref{app:shifted_girsanov_rmk}.
    We begin by computing
    \begin{align*}
        \D(\hat X_t - Y_t')
        &= \{-\nabla V(x)+\nabla V(Y_t') - \eta_t\,(\hat X_t - Y_t')\} \, \D t\,,
    \end{align*}
    which implies
    \begin{align*}
        \frac{1}{2}\,\partial_t\,\E^{\mb P'}[\norm{\hat X_t - Y_t'}^2]
        &= -\eta_t\,\E^{\mb P'}[\norm{\hat X_t - Y_t'}^2] - \E^{\mb P'}\langle \hat X_t - Y_t', \nabla V(x) - \nabla V(Y_t')\rangle \\
        &= -\eta_t\,\E^{\mb P'}[\norm{\hat X_t - Y_t'}^2] - \E^{\mb P'}\langle \hat X_t - Y_t', \nabla V(\hat X_t) - \nabla V(Y_t') + \nabla V(x) - \nabla V(\hat X_t)\rangle \\
        &\le -\eta_t\,\E^{\mb P'}[\norm{\hat X_t - Y_t'}^2] + \sqrt{\E^{\mb P'}[\norm{\hat X_t - Y_t'}^2]\,\E^{\mb P'}[\norm{\nabla V(\hat X_t) - \nabla V(x)}^2]}\,.
    \end{align*}
    Taking square roots and writing $d_t \deq \sqrt{\E^{\mb P'}[\norm{\hat X_t - Y_t'}^2]}$, it leads to
    \begin{align*}
        \partial_t d_t
        &\le -\eta_t d_t + \sqrt{\E^{\mb P'}[\norm{\nabla V(\hat X_t) - \nabla V(x)}^2]}
        \le -\eta_t d_t + \beta \sqrt{\E^{\mb P'}[\norm{\hat X_t - x}^2]} \\
        &= -\eta_t d_t + \beta \sqrt{\E^{\mb P'}[\norm{-t\,\nabla V(x)+\sqrt 2\,B_t'}^2]}
        = -\eta_t d_t + \underbrace{O\bigl(\beta h\,\norm{\nabla V(x)} + \beta d^{1/2} h^{1/2}\bigr)}_{\eqqcolon \mf a}\,.
    \end{align*}
    We take the choice
    \begin{align*}
        \eta_t = \frac{d_0 + \mf a h}{d_0} \, \frac{1}{h-t}\,,
    \end{align*}
    which can be derived as a continuous-time limit of the optimal shifts used in \S\ref{sssec:shift_opt}.
    The differential inequality then yields
    \begin{align*}
        d_t
        &\le \exp\Bigl(-\int_0^t \eta_s \,\D s\Bigr)\,d_0 + \mf a\exp\Bigl(-\int_0^t \eta_s \, \D s\Bigr) \int_0^t \exp\Bigl(\int_0^s\eta_r\,\D r\Bigr)\,\D s \\
        &\le \Bigl(\frac{h-t}{h}\Bigr)^{1+\mf a h/d_0}\,d_0 + \Bigl(\frac{h-t}{h}\Bigr)^{1+\mf ah/d_0}\,\mf a \int_0^t \Bigl(\frac{h}{h-s}\Bigr)^{1+\mf ah/d_0}\,\D s
        = \frac{h-t}{h} \, d_0\,.
    \end{align*}
    Hence,
    \begin{align*}
        \KL(\mb P' \mmid \mb P)
        &\le \frac{1}{4} \,\E^{\mb P'} \int_0^h \eta_t^2 d_t^2 \, \D t
        = \frac{{(d_0 + \mf a h)}^2}{4h}\,. \qedhere
    \end{align*}
\end{proof}

\begin{remark}
    In a similar manner, we can also establish a cross-regularity bound for LMC that holds in R\'enyi divergence.
    See Lemma~\ref{lem:cross_reg_lmc_renyi} in \S\ref{app:renyi_shifted_girsanov}.
\end{remark}

The shifted Girsanov argument can in fact be used to establish Theorem~\ref{thm:lmc} directly, bypassing the use of Theorem~\ref{thm:kl-general}; in fact, the shifted Girsanov argument is exactly the continuous-time analog of Theorem~\ref{thm:kl-general}.
Details in \S\ref{app:shifted_girsanov_rmk}.
In light of this remark, the reader may wonder why we developed the discrete-time scaffolding of Theorem~\ref{thm:kl-general}. This is for two reasons. The first reason is that for more complicated kernels, such as the randomized midpoint discretization considered in \S\ref{sec:rmd}, there may not be a natural SDE interpolation of the algorithm iterates, and hence the shifted Girsanov argument cannot be directly applied.
Theorem~\ref{thm:kl-general} circumvents this issue by considering the entire kernel $\hat P$ at once and hence has wider applicability. The second reason is that the framework of Theorem~\ref{thm:kl-general} is more modular, as it reduces the problem to checking simpler one-step bounds (instead of redoing the shifted Girsanov argument ``from scratch'' for each application).
Indeed, inherent in the statement of Theorem~\ref{thm:kl-general} is an optimized choice of shifts (carried out in \S\ref{app:general_framework}), resulting in a more user-friendly framework.

    \section{LMC under higher-order smoothness}\label{sec:lmc_smooth}

Here, we consider~\ref{eq:lmc} under a higher-order smoothness assumption.
In the previous section, we used the local error framework only in a shallow way, namely via the trivial bound $\mc E_{\rm weak} \le \mc E_{\rm strong}$.
In this section, we take advantage of the fact that under an additional, \emph{higher-order} smoothness condition (specifically, a bound on $\norm{\nabla \Delta V}$), the weak error improves, as was first noted in~\cite[Lemma D.2]{LiZhaTao22LMCSqrtd}. This improves the iteration complexity of LMC from scaling in the dimension as $\Otilde(d)$ to $\Otilde(d^{1/2})$. Such a result is tight for Gaussian $\pi$ and was previously only known for the $W_2$ metric.

\par We begin by recalling the fact that the weak error improves under higher-order smoothness. For completeness, we provide a proof in \S\ref{app:lmc_higher_local}.

\begin{lemma}[{Weak error for LMC under higher-order smoothness}]\label{lem:lmc_higher_local}
    Let $P$, $\hat P$ denote the kernels corresponding to~\ref{eq:langevin} run for time $h$ and~\ref{eq:lmc} respectively. Assume that $V \in C^3(\R^d)$, with $-\beta I \preceq \nabla^2 V \preceq \beta I$, and furthermore, that $\norm{\nabla \Delta V} \leq \zeta_0 + \zeta_1\,\norm{\nabla V}$ everywhere. Then, for all $x \in \R^d$ and $h \lesssim 1/\beta$, there is a coupling $\hat X_h \sim \delta_x \hat P$, $X_h \sim \delta_x P$ such that
    \begin{align*}
        \norm{\E \hat X_h - \E X_h} \lesssim (\beta + \zeta_1)\, h^2\, \norm{\nabla V(x)} + (\beta+\zeta_1)\,\beta d^{1/2} h^{5/2} +\zeta_0 h^2\,.
    \end{align*}
\end{lemma}

We now combine this weak error guarantee with the properties of the Langevin diffusion (Lemma~\ref{lem:ld_properties}), the strong error (Lemma~\ref{lem:lmc_strong_err}), and the cross-regularity (Lemma~\ref{lem:cross_reg_lmc}) already established in the previous section.
We also recall Remarks~\ref{rmk:slc_init} and~\ref{rmk:lsi_init} on initialization, and we state our results in terms of scale-invariant parameters $\kappa$, $\bar\kappa_0$, $\bar\kappa_1$.

\begin{theorem}[LMC under higher-order smoothness]\label{thm:lmc_smooth}
    Let $\pi \propto \exp(-V)$, and let $\hat P$ denote the kernel for~\ref{eq:lmc}.
    We impose the following assumptions:
    \begin{itemize}
        \item $V \in C^3(\R^d)$ satisfies $0 \preceq \alpha I \preceq \nabla^2 V \preceq \beta I$ over $\R^d$.
        Write $\kappa \deq \beta/\alpha$.
        \item $V$ satisfies $\norm{\nabla \Delta V} \le \zeta_0 + \zeta_1\,\norm{\nabla V}$.
        Write $\bar \kappa_0 \deq \zeta_0/\alpha^{3/2}$ and $\bar\kappa_1 \deq \zeta_1/\alpha$.
    \end{itemize}
    Then, the following statements hold.
    \begin{enumerate}
        \item \underline{$\alpha > 0$ case:} If $\varepsilon \in [0,\sqrt{d/\kappa}]$, $h\asymp \frac{\varepsilon}{\beta}\min\{\frac{\kappa}{\bar\kappa_0}, \, \frac{1}{\sqrt{(1+\zeta_1/\beta)\,(\kappa+\bar\kappa_1)\,\kappa d \log((\kappa+\bar\kappa_1)\,d/\varepsilon^2)}}\}$, and $W_2(\hat\mu_0,\pi)\le\sqrt{d/\alpha}$, then $\KL(\hat\mu_0 \hat P^N \mmid \pi) \le \varepsilon^2$ for all
        \begin{align*}
            N \ge \widetilde\Omega\Bigl(\frac{\bar\kappa_0 + (\kappa^2 + \kappa\bar\kappa_1)\,d^{1/2}}{\varepsilon}\Bigr)\,.
        \end{align*}
        \item \underline{$\alpha = 0$ case:} If $\varepsilon > 0$ is sufficiently small and $h = \widetilde \Theta(\frac{\varepsilon^2}{\max\{\zeta_0 W, \,(\beta+\zeta_1)\,\beta W^2, \, (\beta+\zeta_1)\,\sqrt{\beta d W^2}\}})$, where $W \deq W_2(\hat\mu_0,\pi)$, then $\KL(\hat\mu_0\hat P^N \mmid \pi) \le \varepsilon^2$ for
        \begin{align*}
            N = \widetilde\Theta\Bigl(\frac{\zeta_0 + (\beta+\zeta_1)\,(\beta W + \beta^{1/2} d^{1/2})}{\varepsilon^4}\, W^3\Bigr)\,.
        \end{align*}
    \end{enumerate}
\end{theorem}

\begin{theorem}[LMC under higher-order smoothness and LSI]\label{thm:lmc_smooth_lsi}
    Let $\pi \propto \exp(-V)$, and let $\hat P$ denote the kernel for~\ref{eq:lmc}.
    We impose the following assumptions:
    \begin{itemize}
        \item $\pi$ satisfies~\eqref{eq:lsi} with constant $1/\alpha$.
        \item $V \in C^3(\R^d)$ satisfies $-\beta I \preceq \nabla^2 V \preceq \beta I$ over $\R^d$.
        Write $\kappa \deq \beta/\alpha \ge 1$.
        \item $V$ satisfies $\norm{\nabla \Delta V} \le \zeta_0 + \zeta_1\,\norm{\nabla V}$.
        Write $\bar \kappa_0 \deq \zeta_0/\alpha^{3/2}$ and $\bar\kappa_1 \deq \zeta_1/\alpha$.
        \item The initialization satisfies $\log \chi^2(\hat\mu_0,\pi) = \widetilde O(d)$.
    \end{itemize}
    If $\varepsilon \in [0,\sqrt d]$ and $h = \widetilde\Theta(\frac{\varepsilon}{\max\{(\beta+\zeta_1)\sqrt{\kappa d}, \, \zeta_0/\sqrt\alpha\}})$, then $\KL(\mu_0 \hat P^N \mmid \pi) \le \varepsilon^2$ for
    \begin{align*}
        N = \widetilde\Theta\Bigl(\frac{\bar\kappa_0 + (\kappa^{3/2} + \kappa^{1/2}\bar\kappa_1)\,d^{1/2}}{\varepsilon}\Bigr)\,.
    \end{align*}
\end{theorem}

For brevity we defer the proofs to \S\ref{app:lmc_smooth_pf}. We make several remarks about these results, focusing on the dependence on $d$ and $\eps$ which is the main point of these results from a complexity perspective.

\begin{remark}[Dependence on $d,\eps$]\label{rmk:op_norm_lip}
    Theorem~\ref{thm:lmc_smooth} shows that under strong log-concavity and a stringent third-order condition (which is nevertheless satisfied for, e.g., a Gaussian target measure), the iteration complexity in KL divergence is $\widetilde O(d^{1/2}/\varepsilon)$.
    Previously, such a result was only known to hold in the $W_2$ metric.
    
    Under the more usual third-order smoothness condition, namely, that $\nabla^2 V$ is $\rho$-Lipschitz in the operator norm, one has $\norm{\nabla \Delta V} \le \rho \sqrt d$.
    Theorems~\ref{thm:lmc_smooth} and~\ref{thm:lmc_smooth_lsi} therefore immediately yield corollaries in this setting as well.
    For example, in the strongly log-concave and LSI cases, the iteration complexity becomes $\widetilde O(d/\varepsilon)$ (compared to $\widetilde O(d/\varepsilon^2)$ for LMC).
\end{remark}

In these settings, we are unaware of any prior analyses of LMC based on either the Girsanov approach or the convex optimization approach (see \S\ref{sec:lmc_discussion} for descriptions of these) which can exploit higher-order smoothness for LMC\@.
See also \S\ref{sec:lmc_discussion} for further discussion about the advantages of our framework over those standard analysis techniques. The closest prior result to ours is in~\cite{mou2022improved}, which is based on a sophisticated application of the interpolation method.
The main result of~\cite{mou2022improved} establishes an $\widetilde O(d/\varepsilon)$ iteration complexity guarantee for LMC in KL divergence under LSI and the Hessian Lipschitz assumption (cf.\ Remark~\ref{rmk:op_norm_lip}).
Here, we have recovered their result as Theorem~\ref{thm:lmc_smooth_lsi}, as a direct application of our KL local error framework via a straightforward calculation of the weak error (Lemma~\ref{lem:lmc_higher_local}); moreover, our argument shows that the complexity improves to $\widetilde O(\sqrt d/\varepsilon)$ under the more stringent higher-order smoothness condition.

Finally, we remark in passing that for the strongly log-concave setting, the dependence on $\kappa$ is sharper for Theorem~\ref{thm:lmc_smooth_lsi} than in Theorem~\ref{thm:lmc_smooth} due to the stronger initialization assumption, which leads to better control of the error terms. The dependence on $\kappa$ can be similarly sharpened in Theorem~\ref{thm:lmc_smooth} by using that initialization, or directly as a corollary of Theorem~\ref{thm:lmc_smooth_lsi} (since LSI implies strong log-concavity). The reason we present our results in this way is to emphasize that Dirac initialization suffices for our framework, as opposed to standard analyses (see the discussion in \S\ref{sec:lmc_discussion}).

    \section{Randomized midpoint discretization}\label{sec:rmd}

The randomized midpoint discretization, first introduced in~\cite{shen2019randomized} and applied to the Langevin diffusion in~\cite{hebalasubramanianerdogdu2020randomizedmidpoint}, improves the iteration complexity of LMC substantially from $\widetilde O(d/\varepsilon^2)$ to $\widetilde O(\sqrt{d/\varepsilon^2})$, \emph{without requiring higher-order smoothness}.
Moreover, it was shown to be an \emph{optimal} discretization in a suitable sense in~\cite{caoluwang2021uldlowerbd}.
When applied to the underdamped (or kinetic) Langevin diffusion, it yields the only known sampling algorithm with a complexity of $\widetilde O(\sqrt[3]{d/\varepsilon^2})$ without higher-order smoothness assumptions~\cite{shen2019randomized}.\footnote{The shifted ODE method~\cite{foslyoobe21shiftedode} also attains dimension dependence $\widetilde O(d^{1/3})$ with an idealized ODE, but there is currently no rate-matching discretization thereof.}
For these reasons, it remains an important algorithm for study.
However, current guarantees only hold in $W_2$, since the use of uniformly random integration times has thus far prohibited its analysis in KL divergence (or even the simpler metric of total variation) with rates better than the vanilla LMC discretization.
This was listed as a ``highly non-trivial open problem'' in~\cite{YuDal24Parallel}.

As discussed in the related work section (\S\ref{ssec:related_work}), the recent work of~\cite{KanNag24PoiMidpt} makes progress toward this problem.
Namely, they show that a variant of the randomized midpoint method, when applied to the Langevin diffusion, achieves the rate $\widetilde O(d^{3/4}/\varepsilon)$ in total variation distance for the strongly log-concave and LSI settings. Their clever proof technique, which compares the Poisson midpoint method to LMC with a smaller step size, is designed to circumvent some of the difficulties with analyzing RMD\@. However, their technique is currently unable to reach the expected $\widetilde O(d^{1/2}/\varepsilon)$ rate, and it is also observed therein that existing techniques are unable to analyze the original randomized midpoint method beyond the strongly log-concave setting in $W_2$.

We resolve this open problem by obtaining the first $\widetilde O(\sqrt{d/\varepsilon^2})$ rate for the randomized midpoint discretization that holds in KL divergence.
This is obtained as an application of our KL local error framework, which allows us to ``upgrade'' $W_2$ guarantees into KL guarantees.
Moreover, we establish the first $\widetilde O(\sqrt{d/\varepsilon^2})$ rate for the randomized midpoint discretization, in \emph{any} metric, when $\pi$ is only assumed to be weakly log-concave or to satisfy LSI\@.

First, we recall the algorithm.
Let $u_0,u_1,u_2,\dotsc$ be a sequence of i.i.d.\ uniform random variables over $[0,1]$, independent of the Brownian motion.
Then, we set
\begin{align}\label{eq:rmlmc}\tag{$\msf{RM\text{--}LMC}$}
    \begin{aligned}
        \hat X_{nh + u_n h}^+
        &\deq \hat X_{nh} - u_n h\,\nabla V(\hat X_{nh}) + \sqrt 2\,(B_{nh + u_n h} - B_{nh})\,, \\
        \hat X_{(n+1)h}
        &\deq \hat X_{nh} - h\,\nabla V(\hat X_{nh+u_n h}^+) + \sqrt 2\,(B_{(n+1)h} - B_{nh})\,, \qquad n=0,1,2,\dotsc\,.
    \end{aligned}
\end{align}
In words, the algorithm takes a preliminary step of LMC with a \emph{random} step size of length $u_n h$ to obtain $\hat X_{nh+u_n h}^+$, and then takes the next iterate $\hat X_{(n+1)h}$ to be a step of LMC from $\hat X_{nh}$ except that the gradient is evaluated at the preliminary point $\hat X_{nh + u_n h}^+$.
Crucially, the Brownian increments in the two steps of the algorithm are correlated, but the algorithm can still be implemented; see~\cite{shen2019randomized, hebalasubramanianerdogdu2020randomizedmidpoint, chewibook} for details.

The randomness from the uniform variables engenders favorable cancellations, which readily leads to an improved $W_2$ rate through the local error framework---in particular, via improved weak error.
The difficulty of extending the guarantee to other metrics is that if, say, we try to directly apply Girsanov's theorem to bound the discretization error, there is no natural (adapted) SDE which interpolates the iterates of~\ref{eq:rmlmc}.
For similar reasons, it is also not known how to successfully apply the interpolation method or the proof via convex optimization.
However, our KL local error framework (Theorem~\ref{thm:kl-general}) allows us to take full advantage of the improved weak error and thereby obtain the correct rate.

We again make use of the properties of the continuous-time Langevin diffusion (Lemma~\ref{lem:ld_properties}).
Additionally, the following estimates are standard, see, e.g.,~\cite{hebalasubramanianerdogdu2020randomizedmidpoint} or~\cite[\S 5]{chewibook}. Observe that the strong error for RM--LMC is of the same order as LMC (Lemma~\ref{lem:lmc_strong_err}), but the weak error improves by a factor of $h$.

\begin{lemma}[Local errors for RM--LMC]\label{lem:rmlmc_local}
    Let $P$, $\hat P$ denote the kernels corresponding to~\ref{eq:langevin} run for time $h$ and~\ref{eq:rmlmc} respectively. Assume that $-\beta I \preceq \nabla^2 V \preceq \beta I$. Then, for all $x \in \R^d$ and $h \lesssim 1/\beta$, there is a coupling $\hat X_h \sim \delta_x \hat P$, $X_h \sim \delta_x P$ such that the following hold.
    \begin{enumerate}
        \item \underline{Weak error.} $\norm{\E\hat X_h - \E X_h} \lesssim \beta^2 h^3\,\norm{\nabla V(x)} + \beta^2 d^{1/2} h^{5/2}$.
        \item \underline{Strong error.} $\norm{\hat X_h - X_h}_{L^2} \lesssim \beta h^2\,\norm{\nabla V(x)} + \beta d^{1/2} h^{3/2}$.
    \end{enumerate}
\end{lemma}

On the other hand, the cross-regularity is problematic.
Although we expect that RM--LMC enjoys an improved version of cross-regularity compared to LMC\@, we are unable to prove it. Nevertheless, the reason we are ultimately able to obtain our guarantee for RM--LMC is that the cross-regularity only enters the analysis of Theorem~\ref{thm:kl-general} for one step at the very end, contributing a lower-order term to the final bound.
Indeed, by using the cross-regularity for LMC (Lemma~\ref{lem:cross_reg_lmc}), we would already be able to show that $N-1$ steps of RM--LMC, followed by a single step of LMC\@, suffices to achieve the $\widetilde O(\sqrt{d/\varepsilon^2})$ rate in KL divergence (see Remark~\ref{rmk:last_step_hack}).
However, this approach is not fully satisfactory, as it requires modifying the algorithm.
To address this, we prove the following lemma, which shows that the cross-regularity for RM--LMC is \emph{not much worse} than that of LMC\@.

\begin{lemma}[Cross-regularity for RM--LMC]\label{lem:cross_reg_rmlmc}
    Let $P$, $\hat P$ denote the kernels corresponding to~\ref{eq:langevin} and~\ref{eq:rmlmc} respectively.
    Assume that $-\beta I \preceq \nabla^2 V \preceq \beta I$ and that $h \lesssim 1/\beta$ for a sufficiently small implied constant.
    Then, for all $x,y\in\R^d$,
    \begin{align*}
        \KL(\delta_x \hat P \mmid \delta_y P)
        \lesssim \frac{\norm{x-y}^2}{h} \log \frac{1}{\beta h} + \beta^2 h^3 \,\norm{\nabla V(x)}^2 + \beta^2 dh^2\,.
    \end{align*}
\end{lemma}

We defer the proof to \S\ref{app:cross_reg_rmlmc}. With this lemma in hand, we can now establish the following KL guarantees for RM--LMC\@.
We recall Remarks~\ref{rmk:slc_init} and~\ref{rmk:lsi_init} on initialization.

\begin{theorem}[RM--LMC]\label{thm:rmlmc}
    Let $\pi \propto \exp(-V)$, and let $\hat P$ denote the kernel for~\ref{eq:rmlmc}.
    We impose the following assumption:
    \begin{itemize}
        \item $V$ is twice-continuously differentiable and satisfies $0 \preceq \alpha I \preceq \nabla^2 V \preceq \beta I$ over $\R^d$.
        Write $\kappa \deq \beta/\alpha$.
    \end{itemize}
    Then, the following statements hold.
    \begin{enumerate}
        \item \underline{$\alpha > 0$ case:} If $\varepsilon \in [0,\sqrt d/\kappa]$, $h = \widetilde\Theta(\frac{\varepsilon}{\beta d^{1/2}})$, and $W_2(\hat\mu_0,\pi) \le \sqrt{d/\alpha}$, then $\KL(\hat\mu_0\hat P^N \mmid \pi) \le \varepsilon^2$ for
        \begin{align*}
            N \ge \widetilde \Omega\Bigl(\frac{\kappa d^{1/2}}{\varepsilon}\Bigr)\,.
        \end{align*}
        \item \underline{$\alpha = 0$ case:} If $\varepsilon > 0$ is sufficiently small and $h = \widetilde\Theta(\frac{\varepsilon^{4/3}}{\beta^{4/3} d^{1/3} W^{2/3}})$, then $\KL(\hat\mu_0\hat P^N \mmid \pi) \le \varepsilon^2$ for
        \begin{align}\label{eq:rmd_weak_cvx}
            N = \widetilde\Theta\Bigl(\frac{\beta^{4/3} d^{1/3} W^{8/3}}{\varepsilon^{10/3}}\Bigr)\,,
        \end{align}
        where $W \deq W_2(\hat\mu_0,\pi)$.
    \end{enumerate}
\end{theorem}

\begin{theorem}[RM--LMC under LSI]\label{thm:rmlmc_lsi}
    Let $\pi \propto \exp(-V)$, and let $\hat P$ denote the kernel for~\ref{eq:rmlmc}.
    We impose the following assumptions:
    \begin{itemize}
        \item $\pi$ satisfies~\eqref{eq:lsi} with constant $1/\alpha$.
        \item $V$ is twice-continuously differentiable and satisfies $-\beta I \preceq \nabla^2 V \preceq \beta I$ over $\R^d$.
        Write $\kappa \deq \beta/\alpha \ge 1$.
        \item The initialization satisfies $\log\chi^2(\hat \mu_0 \mmid \pi) = \widetilde O(d)$.
    \end{itemize}
    If $\varepsilon \in [0,\sqrt d]$ and $h = \widetilde \Theta(\frac{\varepsilon}{\beta \kappa^{1/2} d^{1/2}})$, then $\KL(\mu_0\hat P^N \mmid \pi) \le \varepsilon^2$ for
    \begin{align*}
        N = \widetilde\Theta\Bigl(\frac{\kappa^{3/2} d^{1/2}}{\varepsilon}\Bigr)\,.
    \end{align*}
\end{theorem}

We defer the proofs to \S\ref{app:rmlmc}.

As discussed above, none of the standard analysis techniques discussed in \S\ref{sec:lmc_discussion} have successfully been applied to RM{--}LMC, and our KL guarantees are the first in each of the three settings.
We conclude with a few remarks about our rate~\eqref{eq:rmd_weak_cvx} for the weakly log-concave setting. This rate should be interpreted as $\widetilde O((d/\varepsilon^2)^{1/3}\, (\beta W^2/\varepsilon^2)^{4/3})$, which is incomparable to the rate $\widetilde O((d/\varepsilon^2)\, (\beta W^2/\varepsilon^2))$ of~\cite{durmus2019analysis} in that the dependence is better for the former term, but worse for the latter term. This suggests potential room for further improvement in our rate, cf.\ Remark~\ref{rmk:rmd_weak_cvx}.
We also note that for this rate~\eqref{eq:rmd_weak_cvx}, we only kept track of the leading-order term as $\varepsilon \searrow 0$; a more detailed rate can be extracted from the proof.

\begin{remark}
    It is natural to ask whether we can also obtain KL guarantees for the randomized midpoint discretization of the underdamped Langevin diffusion, which was originally considered in~\cite{shen2019randomized}.
    The study of the underdamped Langevin diffusion requires the development of further machinery to handle the degeneracy of the noise, and will be handled in a future work.
\end{remark}

	\paragraph*{Acknowledgments.} We thank Matthew S.\ Zhang for numerous helpful conversations. JMA acknowledges the support of a Seed Grant Award from Apple and an NYU Faculty Fellowship. SC acknowledges the support of the Eric and Wendy Schmidt Fund at the Institute for Advanced Study. This work was partially completed while the authors were visiting the Institute for Mathematical Sciences at the National University of Singapore in January 2024. 

    \newpage

    \appendix

\section{Extension to R\'enyi divergences}\label{app:renyi}

In this Appendix section, we extend Theorem~\ref{thm:kl-simple} to R\'enyi divergences. In \S\ref{app:renyi-prelim}, we first briefly recall relevant background about R\'enyi divergences, and then in \S\ref{app:renyi-result} we detail this extension.

\subsection{Background on R\'enyi divergence}\label{app:renyi-prelim}

\begin{defin}\label{def:renyi}
	The R\'enyi divergence of order $q \in (1,\infty)$ between probability measures $\mu,\nu$ is defined as
	\begin{align*}
		\Ren_q(\mu \mmid \nu)  \deq  \frac{1}{q-1} \log \int \Bigl( \frac{\D \mu}{\D \nu} \Bigr)^q\, \D \nu
	\end{align*}
	if $\mu \ll \nu$, and otherwise is defined to be $+\infty$. 
\end{defin}

The R\'enyi divergences of order $q = 1$ and $q = \infty$ are interpreted in the limiting sense: $\Ren_1 = \KL$ coincides with the KL divergence, and $\Ren_{\infty}(\mu,\nu) = \log \| \D \mu / \D \nu \|_{L^{\infty}(\nu)}$. Another famous special case is $\Ren_2 = \log(1 + \chi^2)$ due to the connection with the $\chi^2$-divergence. R\'enyi divergences can also be extended naturally to $q \in (0,1)$, see~\cite{van2014renyi}. R\'enyi divergences enjoy similar properties to KL, albeit sometimes in weaker forms. The following proposition collects some such properties; proofs and further background on R\'enyi divergence can be found in surveys such as~\cite{van2014renyi,mironov2017renyi}.

\begin{prop}[Basic properties of R\'enyi divergences]\label{prop:renyi_prop}
	Let $q \in [1,\infty]$ and let $\mu$, $\nu$ be probability measures.
	\begin{enumerate}
		\item (Positivity) $\Ren_q(\mu \mmid \nu) \ge 0$, with equality if and only if $\mu = \nu$.
		\item (Monotonicity) R\'enyi divergences are increasing in the order, i.e., $q \mapsto \Ren_q(\mu \mmid \nu)$ is increasing.
		\item (Data-processing inequality) For any Markov kernel $P$, it holds that $\Ren_q(\mu P \mmid \nu P) \le \Ren_q(\mu \mmid \nu)$.
            \item (Continuity) $\Ren_q$ is continuous in $q$ for all $q$ such that $\Ren_q < \infty$.
		\item (Weak triangle inequality) For any third probability measure $\pi$, and any $\lambda \in (0,1)$, 
		\begin{align*}
			\Ren_q(\mu \mmid \pi) \le \frac{q-\lambda}{q-1}\,\Ren_{q/\lambda}(\mu \mmid \nu) + \Ren_{(q-\lambda)/(1-\lambda)}(\nu \mmid \pi)\,.
		\end{align*}
	\end{enumerate}
\end{prop}

We often make use of the following special case of this weak triangle inequality. 

\begin{cor}[KL weak triangle inequality]\label{kl:weak-triangle}
    For any probability measures $\mu$, $\nu$, $\pi$,
    \begin{align*}
        \KL(\mu \mmid \pi) \le 2 \KL(\mu \mmid \nu) + \log(1+\chi^2(\nu \mmid \pi))\,.
    \end{align*}
\end{cor}
\begin{proof}
    Apply the R\'enyi weak triangle inequality with $\lambda = 1-\varepsilon$, $q = 1+\varepsilon$, and let $\varepsilon\searrow 0$.
\end{proof}

The R\'enyi divergences satisfy a ``composition rule'' similar to the chain rule~\eqref{eq:chain-rule-kl} satisfied by the KL divergence. Namely, in the notation of \S\ref{sec:prelim}, 
\begin{align}
	\Ren_q(\bs \mu^{\msf Y} \mmid \bs \nu^{\msf Y})
	\leq
	\Ren_q(\bs\mu^{\msf X,\msf Y} \mmid \bs \nu^{\msf X,\msf Y})
	&\le \Ren_q(\bs \mu^{\msf X} \mmid \bs \nu^{\msf X}) + \esssup{\bs\mu^{\msf X}}{\bigl[\Ren_q(\bs \mu^{\msf Y\mid \msf X=\bullet} \mmid \bs \nu^{\msf Y\mid \msf X=\bullet})\bigr]}\,.
	\label{eq:composition-rule-renyi}
\end{align}
In words, the key difference between the KL chain rule~\eqref{eq:chain-rule-kl} and this R\'enyi composition rule~\eqref{eq:composition-rule-renyi} and is that the expectation is replaced by a supremum. The shifted composition rule---the namesake of this series of papers---generalizes this standard R\'enyi composition rule to allow for shifts in an analogous way as the shifted chain rule (Theorem~\ref{thm:scr-kl}) generalizes the standard KL chain rule. This is Theorem 3.1 of~\cite{scr1} (with a slightly simplified statement here since $q < 1$ is not the relevant regime in the applications to sampling pursued in this paper).

\begin{theorem}[Shifted composition rule]\label{thm:scr}
	Let $\msf X$, $\msf X'$, $\msf Y$ be three jointly defined random variables on a standard probability space $\Omega$.
	Let $\bs\mu$, $\bs \nu$ be two probability measures over $\Omega$, with superscripts denoting the laws of random variables under these measures. For any $q \in [1,\infty]$,
	\begin{align*}
		\Ren_q(\bs\mu^{\msf Y} \mmid \bs \nu^{\msf Y})
		&\le \Ren_q(\bs\mu^{\msf X'} \mmid \bs \nu^{\msf X}) + \inf_{\gamma \in \Coup(\bs \mu^{\msf X}, \bs \mu^{\msf X'})} {\esssup{\gamma}_{(x,x') \in \Omega \times \Omega}{\Ren_q(\bs\mu^{\msf Y\mid \msf X=x} \mmid \bs \nu^{\msf Y\mid \msf X=x'})}}\,.
	\end{align*}
\end{theorem}

\subsection{Local error framework for R\'enyi divergences}\label{app:renyi-result}

Here, we extend Theorem~\ref{thm:kl-simple} to R\'enyi divergences. The result replaces $\KL$ with $\Ren_q$, and $W_2$ with\footnote{Recall the definition $W_{\infty}(\mu,\nu) = \inf_{\gamma \in \Coup(\mu,\nu)} \esssup{\gamma} \|X-Y\|$.} $W_{\infty}$, and provides an otherwise identical bound in the case that the initialization measures $\hat\mu_0$, $\nu_0$ are Dirac distributions. The result immediately extends in a black-box way to arbitrary initialization measures $\hat\mu_0$, $\nu_0$ using the convexity principle for R\'enyi divergences~\cite[Theorem 3.7]{scr1}; however, that leads to a more complicated final expression (since the convexity principle is more involved for R\'enyi divergences than for KL), so we omit this extension for simplicity. 

We remark that the reason we did not emphasize this R\'enyi framework in the main text is that the $W_\infty$ assumption is too stringent for the applications to sampling we have in mind, and future work will be aimed at relaxing this assumption. 
However, we note that the theorem below already covers standard applications of shifted divergences to differential privacy~\cite{pabi}.

\begin{theorem}[Coupling analyses for R\'enyi divergences]\label{thm:main-renyi}
	Let $q \in [1,\infty]$. Suppose that $P$, $\hat{P}$ are Markov kernels on $\R^d$ satisfying the following one-step bounds:
	\begin{enumerate}
		\item \underline{Regularity.} $\Ren_{q}(\delta_x P \mmid \delta_y P) \le c\,\norm{x-y}^2$\,.
		\item \underline{Wasserstein bound.} $W_\infty(\delta_x \hat P, \delta_y P) \le L\,\norm{x-y} + a$.
		\item \underline{Cross-regularity.} $\Ren_{q}(\delta_x \hat P \mmid \delta_y P) \le c'\,\norm{x-y}^2 + b^2$.
	\end{enumerate}
	If $L \leq 1$, then for any $x,y \in\R^d$,
	\begin{align}
		\Ren_q(\delta_x \hat{P}^N \mmid \delta_y P^N)
		\leq
		\Bigl( c + (c'-c)\, \frac{1-L^2}{1-L^{2N}} \Bigr)\, \Bigl(\frac{1+L}{1-L}\Bigr)\, \frac{( a\,(1-L^{N-1}) +  L^{N-1}\, (1-L)\, \norm{x-y})^2} {1-L^{2N}} 
		+ b^2\,.
		\label{eq:thm-main:renyi}
	\end{align}

\end{theorem}

The proof is essentially identical to that of Theorem~\ref{thm:kl-simple} in \S\ref{ssec:kl-simple}, with $\KL$ replaced by $\Ren_q$, and $W_2$ replaced by $W_{\infty}$. Indeed, after making these changes, the argument extends unchanged except that the KL divergence bound in Lemma~\ref{lem:aux-kl} should be adapted to the following R\'enyi divergence bound. The proof of this lemma is omitted since it is identical except that the application of the KL shifted chain rule (Theorem~\ref{thm:scr-kl}) is now replaced by the R\'enyi shifted composition rule (Theorem~\ref{thm:scr}), which leads to a replacement of the expectations with $\mathrm{ess\,sup}$, and is why $W_2$ is replaced by $W_{\infty}$.

\begin{lemma}[R\'enyi divergence bound]\label{lem:aux-renyi-simple}
	Denote $d_n \deq W_\infty(\hat\mu_n, \nu_n')$. Then,
	\begin{align*}
		\Ren_q(\hat{\mu}_N \mmid \nu_N)
		&\le
		\sum_{n=0}^{N-1} \bigl\lVert\Ren_q\bigl(Q_n(\tilde Y_n,\cdot) \bigm\Vert P(Y_n',\cdot)\bigr)\bigr\rVert_{L^\infty}
		\leq c\sum_{n=0}^{N-2} \eta_n^2 d_n^2 + c' d_{N-1}^2 + b^2\,.
	\end{align*}
\end{lemma}

\section{Proofs for Section~\ref{sec:multi}}

\subsection{Optimal shifts}\label{app:shifts-optimal}

Here we prove the analog of Lemma~\ref{lem:shift-opt} for $L<1$ (rather than $L=1$). Specifically, the following lemma solves in closed form the optimization problem
\begin{align}
    R_N(a,d_0,L)  \deq  \min_{\substack{\eta_0,\dots,\eta_{N} \in [0,1] \\ \text{s.t.}~\eta_N = 1 \\ d_{n+1} \leq L\,(1 - \eta_n) d_n + a, \; \forall n < N}} \sum_{n=0}^N \eta_n^2 d_n^2\,.
    \label{eq:shift-opt:L<1explicit}
\end{align}
Throughout this section, we assume for simplicity that $d_0 \geq a$; this suffices since the bounds increase in $d_0$, so these bounds hold for all $d_0$ by replacing $d_0$ with $\max(d_0,a)$. Exact closed-form solutions for $d_0 < a$ are also possible using the same techniques, but require more casework (cf., the $L=1$ analysis in \S\ref{sssec:shift_opt}). We remark that the derivations here for $L < 1$ extend more generally to $L=1$ (by taking appropriate limits using L'H\^{o}pital's theorem) and also to $L > 1$ (by assuming that $d_0$ is sufficiently large that the shifts $\eta_n$ described below are bounded by $1$).

\begin{lemma}[Closed-form expression for $R_N$]\label{lem:shiftopt:L<1explicit}
	For any $N \geq 0$, any $L \in (0,1)$, and any $d_0 \geq a \geq 0$:
	\begin{align}
		R_N(a,d_0,L) =
			\frac{(a\,(1+L+\dots + L^{N-1}) + d_0L^N)^2}{1+L^2 + \dots + L^{2N}} = \frac{1+L}{1-L}\, \frac{(a\,(1-L^N) + d_0L^N\,(1-L))^2}{1-L^{2(N+1)}}\,.
	\end{align}
\end{lemma}
\begin{proof}
    We prove by induction on $N$. The base case $N=0$ is trivial. For the induction, express    
    \begin{align*}
        R_{N}(a,d_0,L)
        &=
        \min_{\eta_0 \in [0,1]}\bigl\{\eta_0^2 d_0^2 + R_{N-1}(a,(1-\eta_0)Ld_0 + a, L)\bigr\}
        \\ &=
        \min_{\eta_0 \in [0,1]}\Bigl\{\eta_0^2 d_0^2 +
        \frac{1+L}{1-L} \, \frac{(a\,(1-L^{N-1}) + ((1-\eta_0) L d_0 + a)\,L^{N-1}\,(1-L))^2}{1-L^{2N}}\Bigr\}\,.
    \end{align*}
    Above, the first step splits the joint optimization over $\eta_0, \dots, \eta_{N}$ in $R_{N}$ into two layers of optimization: first over $\eta_0$, and then over $\eta_1, \dots, \eta_{N}$. The second step is by the inductive hypothesis. The remaining optimization over $\eta_0$ is now solvable in closed form by appealing to the following calculation. 
\end{proof}

\begin{obs}\label{obs:shift-opt-induct:L<1}
	For any $N \geq 0$ and any $d \geq a > 0$, the optimization problem
\begin{align*}
	\min_{\eta \in [0,1]}\Bigl\{\eta^2 d^2 +  \frac{1+L}{1-L} \, \frac{(a\,(1-L^{N-1}) + ((1-\eta) L d_0 + a)\,L^{N-1}\,(1-L))^2}{1-L^{2N}}\Bigr\}
\end{align*}
has unique optimal solution
\begin{align*}
	\eta =
	\frac{L^N\,(1+L)\, (a\,(1-L^N) + d_0 L^N\, (1-L))}{ d_0\, (1 - L^{2(N+1)}) }
\end{align*}
and optimal value
\begin{align*}
	\frac{1+L}{1-L} \, \frac{(a\,(1-L^N) + dL^N\,(1-L))^2}{1-L^{2(N+1)}}\,.
\end{align*}
\end{obs}
\begin{proof}
	This is a straightforward albeit tedious calculation using single-variable calculus. To find the optimal $\eta$, set the derivative equal to $0$ (this gives the global minimizer of the unconstrained optimization problem since the objective is convex in $\eta$) and note that the resulting unique value of $\eta$ lies in $[0,1]$ and thus is a minimizer of the original constrained problem. Plugging in this optimal $\eta$ yields the claimed optimal value. 
\end{proof}

\begin{remark}\label{rem:shiftopt:L<1explicit}
	The optimal shifts implicit in the proof of Lemma~\ref{lem:shiftopt:L<1explicit} are $\eta_N = 1$ and 
	\begin{align*}
		\eta_n =
			\frac{L^{N-n}\,(1+L)\, (a\,(1-L^{N-n}) + d_n L^{N-n}\, (1-L))}{ d_n\, (1 - L^{2(N-n+1)}) }
		 \qquad \forall n \in \{0,1, \dotsc, N-1\}\,,
	\end{align*}
	where $d_{n+1} = (1 - \eta_n)L d_n + a$. By substituting in these values and simplifying, one obtains expressions in terms of just the iteration $n$ and the initial data $a$, $d_0$, $L$, $N$:
	\begin{align*}
		\eta_n &= 
		 \frac{L^{N-n}\,(1+L)\, ( a\,(1-L^N) + d_0 L^N\,(1-L))}{d_0L^n\, (1-L^{2(N+1-n)}) + a\, (1-L^n)\, (1 - L^{N+2} - L^{N+2-n} + L^{2N+2-n}) ) /(1 - L)}
    \end{align*}
    for $n \in \{0,1,\dotsc,N-1\}$, and
    \begin{align*}
		d_n &= \frac{
			d_0L^n\,(1-L^{2(N+1-n)})
			+
			a\,(1-L^n)\, (1 - L^{N+2} - L^{N+2-n} + L^{2N+2-n}) / (1-L)
		}{1-L^{2(N+1)}}
	\end{align*}
    for $n\in \{1,\dotsc,N\}$.
\end{remark}

Of course, recall that the ``true'' shift optimization problem has a slightly different objective due to the cross-regularity arising from the final iteration. We now apply these shifts (optimal without the cross-regularity) to this.

\begin{lemma}[Final bound for shift-optimization for $L<1$]\label{lem:shift-opt-L<1}
	For $L<1$, the optimization problem in the right hand side of~\eqref{eq:shift-opt-disc:opt} has value at most 
	\[
	\Bigl( c + (c'-c)\, \frac{1-L^2}{1-L^{2N}} \Bigr)\, \Bigl( \frac{1+L}{1-L} \Bigr) \, \frac{( a\,(1-L^{N-1}) +  L^{N-1}\, (1-L)\, d_0) ^2 } {1-L^{2N}} 
	+ b^2\,.
	\]
\end{lemma}
\begin{proof}
	This optimization problem~\eqref{eq:shift-opt-disc:opt} has the same decision variables and constraints (and nearly the same objective) as the simplified problem~\eqref{eq:shift-opt:L<1explicit}. Thus, by plugging in the optimal solutions to the latter problem from Remark~\ref{rem:shiftopt:L<1explicit}, we obtain a feasible solution to the former problem with value 
	\begin{align*}
		c R_{N-1}(a, d_0, L) + (c' - c)\, d_{N-1}^2 + b^2\,.
	\end{align*}
	The value of $R_{N-1}(a,d_0, L) $ is computed in Lemma~\ref{lem:shiftopt:L<1explicit}, and the value of $d_{N-1}$ is computed in Remark~\ref{rem:shiftopt:L<1explicit}. Plugging these values into the above display and simplifying 
	 completes the proof.
\end{proof}

Note that in the limit $L \to 1$, Lemma~\ref{lem:shift-opt-L<1} exactly recovers the bound in Lemma~\ref{lem:shift-opt}.

\subsection{Proof of Theorem~\ref{thm:kl-general}}\label{app:general_framework}

The proof of Theorem~\ref{thm:kl-general} uses the auxiliary process constructed in \S\ref{sssec:aux_process}, and we use the notation therein, with the following slight modifications. We initialize $\hat X_0 \sim \mu$ and $Y_0' = Y_0 \sim \nu$, which does not change the argument.
For simplicity of notation, we define $a_0(x) \deq \mc E_{\rm strong}(x)$, $a_1(x) \deq \mc E_{\rm weak}(x) + \gamma \mc E_{\rm strong}(x)$, and $\bar a_i \deq \max_{n=0,1,\dotsc,N-1}{\norm{a_i}_{L^2(\mu\hat P^n)}}$ for $i \in \{0,1\}$.
We replace the distance recursion in Lemma~\ref{lem:aux-dist} with the following one.

\begin{lemma}[New distance recursion]\label{lem:new_aux_dist}
    Denote $d_n \deq W_2(\hat\mu_n, \nu_n')$.
    Then, for all $n < N-2$,
    \begin{align}\label{eq:new_aux_dist}
        d_{n+1}^2
        &\le L^2\,{(1-\eta_n)}^2\,d_n^2 + 2\bar a_1\,(1-\eta_n)\, d_n + \bar a_0^2\,.
    \end{align}
\end{lemma}
\begin{proof}
    By a coupling argument, the one-step distance recursion for standard local error analysis (cf.~equation~\eqref{eq:local_error_1step_W2}), and then the Cauchy--Schwarz inequality,
    \begin{align*}
        W_2^2(\hat\mu_{n+1},\nu_{n+1}')
        & \le \E W_2^2\bigl(\hat P(\hat X_n,\cdot), P(\tilde Y_n,\cdot)\bigr)
        \\ &\le \E\bigl[L^2\,\norm{\hat X_n - \tilde Y_n}^2 + 2a_1(\hat X_n)\,\norm{\hat X_n - \tilde Y_n} + {a_0(\hat X_n)}^2\bigr] \\
        &\le L^2\,{(1-\eta_n)}^2\,d_n^2 + 2\bar a_1\,(1-\eta_n)\,d_n + \bar a_0^2\,. \qedhere
    \end{align*}
\end{proof}

A simple adaptation of Lemma~\ref{lem:aux-kl}, recalling that $\tilde Y_{N-1} = \hat X_{N-1}$, yields
\begin{align*}
    \KL(\hat\mu_0\hat P^N \mmid \nu_0 P^N)
    &\le c\sum_{n=0}^{N-2} \eta_n^2 d_n^2 + c' d_{N-1}^2 + \bar b^2\,.
\end{align*}
To prove Theorem~\ref{thm:kl-general}, it remains to optimize over the choice of shifts, which we accomplish below.

\begin{lemma}[New shift optimization]\label{lem:shift-opt-new}
    For any $1/2 \le L \le 2$, there exist shifts $\eta_1,\dotsc,\eta_{N-2} \in [0,1]$ such that if ${\{d_n\}}_{n=0}^{N-1}$ satisfies the recursion~\eqref{eq:new_aux_dist}, then
    \begin{align*}
        c\sum_{n=0}^{N-2} \eta_n^2 d_n^2 + c' d_{N-1}^2
        &\lesssim (c + c')\,\Bigl[\frac{L^{-1}-1}{L^{-N}-1}\,d_0^2 +\bigl( (L-1)\,N \vee \log \bar N\bigr)\, \bar a_0^2 + \bar N\, \bar a_1^2\Bigr]\,,
    \end{align*}
    where $\bar N \deq N \wedge \frac{1}{{(1-L)}_+}$.
\end{lemma}

\begin{remark}[Optimality of Lemma~\ref{lem:shift-opt-new}]\label{rem:shift-opt-new}
    The proof of Lemma~\ref{lem:shift-opt-new} relies on an approximate solution to the simplified shift optimization problem
    \begin{align}
        \min_{\substack{\eta_0, \dots, \eta_{N-1} \in [0,1] \\ \text{\emph{s.t.}}\; \eta_{N-1} = 1 \\ d_{n+1}^2 \leq L^2(1-\eta_n)^2 d_n^2 + 2\bar{a}_1 (1 - \eta_n) d_n + \bar{a}_0^2}} \sum_{n=0}^{N-1} \eta_n^2 d_n^2\,.
        \label{eq:rem:shift-opt-new}
    \end{align}
    Namely, the proof exhibits feasible shifts $\eta_0, \dots, \eta_{N-1}$ that achieve, up to constants, the value $\tfrac{L^{-1} - 1}{L^{-N} - 1}\,d_0^2 +  \bar{a}_1^2 \bar{N}  + \bar{a}_0^2 \log \bar{N}$. This bound is optimal in all three terms. To see optimality of the first two terms, notice that in the case $\bar{a}_1 = L\bar{a}_0$,
    then~\eqref{eq:rem:shift-opt-new} is equivalent to~\eqref{eq:shift-opt:L<1explicit} which, by Lemma~\ref{lem:shiftopt:L<1explicit}, has value of order 
   $\tfrac{L^{-1} - 1}{L^{-N} - 1}\,d_0^2 +  \bar{a}_1^2 \bar{N} $ in the relevant regime $d_0 \geq a$ and $L = 1 - \delta$ for $\delta > 0$ small.
    To show optimality of the third term, consider next the case $\bar{a}_1 = 0$ and $L=1$ (this also gives a lower bound on the general value of~\eqref{eq:rem:shift-opt-new} since $\bar{a}_1 \geq 0$ and $L \leq 1$). Then one can explicitly solve the shift optimization problem~\eqref{eq:rem:shift-opt-new} to see that it has solution $\eta_{i} = \tfrac{1}{N-i}$ and value $\tfrac{d_0^2}{N} + \bar{a}_0^2 \sum_{i=1}^{N-1} \tfrac{1}{i} \asymp \tfrac{d_0^2}{N} + \bar{a}_0^2 \log N$. 
\end{remark}

\begin{proof}[Proof of Lemma~\ref{lem:shift-opt-new}]
    For brevity of notation, in this proof we write $a_0 \deq \bar a_0$ and $a_1 \deq \bar a_1$.
    We split the analysis into two phases, starting with the case $L \le 1$.
    
\textbf{Initial phase.}
We start by focusing on controlling the terms with $n \le n_\star$, where $n_\star \ge 0$ is the largest integer so that $L^{-(N-n_\star)} \ge 2$ (assuming such an integer even exists; if not, then this part of the proof can be skipped).
In this phase, we use the simple choice   
\begin{align}\label{eq:initial_shift}
    \eta_n = \frac{L^{-1} - 1}{L^{-(N-n)}-1}\,.
\end{align}
This choice is motivated by the continuous-time calculations in \S\ref{app:shifted_girsanov_rmk}.
By Young's inequality, the recursion~\eqref{eq:new_aux_dist} implies
\begin{align}\label{eq:shift_opt_aux_bd}
    d_{n+1}^2
    &\le L\,(1-\eta_n)\,d_n^2 + a_0^2 + \frac{(1-\eta_n)\,a_1^2}{L\,(1-L\,(1-\eta_n))}
    \,.
\end{align}
Note also that
\begin{align*}
    L\,(1-\eta_n)
    &= \frac{L^{-(N-(n+1))}-1}{L^{-(N-n)}-1}\,,
\end{align*}
so we have telescoping.

A calculation shows that in this initial phase,
\begin{align*}
    1-L\,(1-\eta_n)
    &= \frac{L^{-(N-n)} - L^{-(N-(n+1))}}{L^{-(N-n)}-1}
    = \frac{1 - L}{1 - L^{N-n}}
    \le 2\,(1-L)\,.
\end{align*}
Therefore, in this phase, the recursion~\eqref{eq:shift_opt_aux_bd} can be further simplified to
\begin{align*}
    d_{n+1}^2
    &\le L\,(1-\eta_n)\,d_n^2 + \frac{a^2}{1-L\,(1-\eta_n)}\,, \qquad a^2 \deq 2\,(1-L)\,a_0^2 + \frac{a_1^2}{L}\,.
\end{align*}
Unrolling the recursion, we have
\begin{align*}
    d_n^2
    &\le \frac{L^{-(N-n)}-1}{L^{-N}-1}\,d_0^2 + a^2 \sum_{k=1}^n \frac{L^{-(N-n)}-1}{L^{-(N-k)}-1}\, \frac{1}{1-L\,(1-\eta_{k-1})}\,.
\end{align*}
A further calculation shows
\begin{align*}
    \sum_{k=1}^n \frac{L^{-(N-n)}-1}{L^{-(N-k)}-1}\, \frac{1}{1-L\,(1-\eta_{k-1})}
    = \sum_{k=1}^n \frac{L^{-(N-n)} - 1}{L^{-1} - 1}\, \frac{L^{-1} - L^{N-k}}{L^{-(N-k)}-1}
    &\le \frac{2\,(L^{-(N-n)} - 1)}{1-L} \sum_{k=1}^n L^{N-k} \\
    &= \frac{2\,(1-L^n)}{{(1-L)}^2}\,.
\end{align*}
This means that the total sum up to this point is
\begin{align*}
    \sum_{n=0}^{n_\star} \eta_n^2 d_n^2
    &= \sum_{n=0}^{n_\star} \Bigl( \frac{L^{-1}-1}{L^{-(N-n)}-1}\Bigr)^2 \,\Bigl\{ \frac{L^{-(N-n)}-1}{L^{-N}-1}\,d_0^2 + \frac{2a^2\,(1-L^n)}{{(1-L)}^2}\Bigr\}\,.
\end{align*}
The first term is bounded by
\begin{align*}
    \frac{{(L^{-1} - 1)}^2\, d_0^2}{L^{-N}-1} \sum_{n=0}^{n_\star} \frac{1}{L^{-(N-n)}-1}
    &\le \frac{2\,{(L^{-1} - 1)}^2\, d_0^2}{L^{-N}-1} \sum_{n=0}^{n_\star} L^{N-n}
    = \frac{2\,(L^{-1}-1)\,d_0^2}{L^{-N}-1}\,(L^{N-n_\star-1}-L^N)\\
    &\le L^{-1}\,\frac{L^{-1}-1}{L^{-N}-1}\,d_0^2\,.
\end{align*}
The second term is bounded by
\begin{align*}
    2a^2 L^{-2} \sum_{n=0}^{n_\star} \frac{1-L^n}{{(L^{-(N-n)}-1)}^2}
    &\le 8a^2 L^{2\,(N-1)} \sum_{n=0}^{n_\star} L^{-2n}
    = \frac{8a^2 L^{-2}}{1-L^2} \,(L^{2\,(N-n_\star)} - L^{2\,(N+1)})
    \le \frac{8a^2}{1-L}\,.
\end{align*}
Hence, we have shown that
\begin{align*}
    \sum_{n=0}^{n_\star} \eta_n^2 d_n^2
    &\le L^{-1}\,\frac{L^{-1}-1}{L^{-N}-1}\,d_0^2 + \frac{2a^2}{1-L}
    \lesssim L^{-1}\,\frac{L^{-1}-1}{L^{-N}-1}\,d_0^2 + a_0^2 + \frac{a_1^2}{1-L}\,.
\end{align*}

\textbf{Terminal phase.}
Next, we consider the remaining terms with $n\ge n_\star$.
For this part of the proof, if $L^{-N} \ge 2$ does not hold, we simply set $n_\star = 0$. Write $d_\star \deq d_{n_\star}$.

In this phase, it turns out that the simple choice of shifts $\eta_n = 1/(N-n)$ suffices.
The distance recursion~\eqref{eq:shift_opt_aux_bd}, after dropping the factor of $L$ for simplicity, becomes
\begin{align}\label{eq:shift_terminal_recursion}
    d_{n+1}^2
    &\le \Bigl(\frac{N-n-1}{N-n}\Bigr)^2\,d_n^2 + a_0^2 + (N-n)\,a_1^2
\end{align}
which leads, after bounding sums by integrals, to
\begin{align*}
    d_n^2
    &\le \Bigl(\frac{N-n}{N-n_\star}\Bigr)^2\, d_\star^2 + a_0^2 \sum_{k=0}^{n-n_\star-1} \Bigl(\frac{N-n}{N-(n-k)}\Bigr)^2 + a_1^2 \sum_{k=0}^{n-n_\star-1}\Bigl(\frac{N-n}{N-(n-k)}\Bigr)^2\,(N-(n-k-1)) \\
    &\lesssim \Bigl(\frac{N-n}{N-n_\star}\Bigr)^2\, d_\star^2 + (N-n)\, a_0^2 + {(N-n)}^2 \log\bigl(\frac{N-n_\star}{N-n}\bigr)\,a_1^2\,.
\end{align*}
Hence,
\begin{align}
    \sum_{n=n_\star}^{N-2} \eta_n^2 d_n^2
    &\lesssim \frac{d_\star^2}{N-n_\star} + a_0^2 \sum_{n=n_\star}^{N-2} \frac{1}{N-n} + a_1^2 \sum_{n=n_\star}^{N-2} \log \frac{N-n_\star}{N-n} \nonumber \\
    &\lesssim \frac{d_\star^2}{N-n_\star} + \log(N-n_\star)\,a_0^2 + (N-n_\star)\,a_1^2\,.\label{eq:shift_terminal_sum}
\end{align}
We also note that
\begin{align}\label{eq:shift_terminal_final_dist}
    d_{N-1}^2
    &\lesssim \frac{d_\star^2}{{(N-n_\star)}^2} + a_0^2 + \log(N-n_\star)\,a_1^2\,.
\end{align}

\textbf{Combining the bounds.}
First, if $L^{-N} < 2$, the first case does not apply and we use the second case to obtain
\begin{align*}
    \sum_{n=0}^{N-2} \eta_n^2 d_n^2
    &\lesssim \frac{d_0^2}{N} + (\log N)\,a_0^2 + N\,a_1^2\,.
\end{align*}
Also, $d_{N-1}^2 \lesssim d_0^2/N^2 + a_0^2 + (\log N)\,a_1^2$, so by adding the last summand,
\begin{align}\label{eq:shift_terminal_kl}
    \Sigma \deq c\sum_{n=0}^{N-2} \eta_n^2 d_n^2 + c' d_{N-1}^2
    &\lesssim (c+c')\,\Bigl[ \frac{d_0^2}{N} + (\log N)\,a_0^2 + N\,a_1^2\Bigr]\,.
\end{align}
Also, $2 > L^{-N} = {(1+(L^{-1} - 1))}^N \ge 1 + N\,(L^{-1}-1)$ implies $N\,(L^{-1}-1) \le 1$, thus $L^{-N}-1 = {(1+(L^{-1}-1))}^N-1 \le \exp(N\,(L^{-1}-1)) -1 \le 2N\,(L^{-1}-1)$.
Therefore, we can also write the above inequality as
\begin{align}\label{eq:shift_terminal_kl_simplified}
    \Sigma
    &\lesssim (c+c')\,\Bigl[ \frac{L^{-1}-1}{L^{-N}-1}\, d_0^2 + (\log N)\,a_0^2 + N\,a_1^2\Bigr]\,.
\end{align}

On the other hand, suppose that $L^{-N} \ge 2$.
Then,
\begin{align}\label{eq:d_star}
    d_\star^2
    &\le \frac{L^{-(N-n_\star)}-1}{L^{-N}-1}\,d_0^2 + \frac{2a^2\,(1-L^{n_\star})}{{(1-L)}^2}
    \le \frac{2L^{-1}}{L^{-N}-1}\,d_0^2 + \frac{2a^2}{{(1-L)}^2}\,,
\end{align}
using the definition of $n_\star$.
Also, $2 \le L^{-(N-n_\star)} \le 2L^{-1}$ together with bounds on the logarithm yield $2L \le (1-L)\,(N-n_\star) \le 2L^{-1}$.
Hence,
\begin{align*}
    \sum_{n=0}^{N-2} \eta_n^2 d_n^2
    &\lesssim L^{-1}\,\frac{L^{-1}-1}{L^{-N}-1}\,d_0^2 + \frac{a^2}{1-L} + L^{-1}\,(1-L)\,\Bigl\{\frac{L^{-1}}{L^{-N}-1}\, d_0^2 + \frac{a^2}{{(1-L)}^2}\Bigr\} + \log\bigl(\frac{2L^{-1}}{1-L}\bigr)\,a_0^2+ \frac{L^{-1}\,a_1^2}{1-L} \\
    &\lesssim L^{-1}\,\Bigl(\frac{L^{-1}-1}{L^{-N}-1}\,d_0^2 + \log\bigl(\frac{1}{1-L}\bigr)\, a_0^2 + \frac{1}{1-L}\,a_1^2\Bigr)\,.
\end{align*}
Also,
\begin{align}\label{eq:final_dN}
    d_{N-1}^2
    &\lesssim L^{-2}\,{(1-L)}^2\,d_\star^2 + a_0^2 + a_1^2 \log \frac{2L^{-1}}{1-L}
    \lesssim \frac{L^{-3}\,{(1-L)}^2}{L^{-N}-1}\,d_0^2 + L^{-2}\,a_0^2 + L^{-2} \log\bigl(\frac{1}{1-L}\bigr)\, a_1^2\,,
\end{align}
so adding together the last summand yields a final bound of
\begin{align*}
    \Sigma
    &\lesssim (cL^{-1} + c' L^{-2})\,\Bigl[\frac{L^{-1}-1}{L^{-N}-1}\,d_0^2 + \log\bigl(\frac{1}{1-L}\bigr)\,a_0^2 + \frac{1}{1-L}\,a_1^2\Bigr]\,.
\end{align*}

If we let $\bar N \deq N \wedge \frac{1}{1-L}$, we can combine both cases into the following bound:
\begin{align*}
    \Sigma
    &\lesssim (cL^{-1}+c' L^{-2})\,\Bigl[ \frac{L^{-1}-1}{L^{-N}-1}\,d_0^2 + (\log\bar N)\,a_0^2 + \bar N\, a_1^2\Bigr]\,.
\end{align*}

\textbf{Non-contractive case.} %
Here, we consider the case $L > 1$.
Note that the recursion~\eqref{eq:shift_opt_aux_bd} still holds, provided that $L\,(1-\eta_n) < 1$.
As before, we split into two regimes.

In the distance recursion~\eqref{eq:shift_opt_aux_bd}, we choose $\eta_n = 1 - 1/L^2$ so that $L\,(1-\eta_n) = 1/L$ and
\begin{align*}
    d_{n+1}^2
    &\le \frac{1}{L}\,d_n^2 + a_0^2 + \frac{La_1^2}{L-1}
    = \frac{1}{L}\,d_n^2 + a^2\,, \qquad a^2 \deq a_0^2 + \frac{La_1^2}{L-1}\,.
\end{align*}
Unrolling,
\begin{align*}
    d_n^2
    &\le \frac{1}{L^n}\,d_0^2 + \frac{L}{L-1}\,a^2\,.
\end{align*}
Hence, since $\eta_n \le 2\,(L-1)/L$, for any $n_\star$ we obtain
\begin{align*}
    \sum_{n=0}^{n_\star-1} \eta_n^2 d_n^2
    &\le \frac{4\,(L-1)}{L}\,(d_0^2 + n_\star a^2)
    \lesssim \frac{L-1}{L}\,d_0^2 + \frac{L-1}{L}\,n_\star a_0^2 + n_\star a_1^2\,.
\end{align*}

In the second regime, we consider $n$ such that $N-n \le 2L/(L-1)$.
We take $\eta_n$ such that
\begin{align*}
    L\,(1-\eta_n) = \Bigl( \frac{N-n-1}{N-n}\Bigr)^2\,,
\end{align*}
i.e., $1-L\,(1-\eta_n) \ge 1/(N-n)$.
Also, $\eta_n \le (L-1)/L + 2/(N-n) \le 4/(N-n)$.
Substituting this into~\eqref{eq:shift_opt_aux_bd}, the recursion~\eqref{eq:shift_terminal_recursion} still holds.
Therefore,~\eqref{eq:shift_terminal_sum} and~\eqref{eq:shift_terminal_final_dist} still hold as well.

To obtain the final bound, we split into two cases.
First, suppose that $N \le 2L/(L-1)$.
Then, only the second regime is relevant, so that the final KL bound is given by~\eqref{eq:shift_terminal_kl}.
Moreover, since $(L^{-1}-1)/(L^{-N}-1) = 1/\sum_{k=0}^{N-1} L^{-k} \ge 1/N$, the bound can further be rewritten in the form~\eqref{eq:shift_terminal_kl_simplified}.

Otherwise, let $n_\star$ be the first non-negative integer such that $N-n_\star \le 2L/(L-1)$.
In this case, we have $N-n_\star \ge L/(L-1)$ and $d_\star^2 \le d_0^2 + \frac{L}{L-1}\, a^2$, so
\begin{align*}
    d_{N-1}^2
    &\lesssim \frac{{(L-1)}^2}{L^2}\,\bigl(d_0^2 + \frac{L}{L-1}\, a_0^2 + \frac{L^2}{{(L-1)}^2}\,a_1^2\bigr) + a_0^2 + \log\bigl(\frac{L}{L-1}\bigr)\,a_1^2 \\
    &\lesssim \frac{{(L-1)}^2}{L^2}\,d_0^2 + a_0^2 + \log\bigl(\frac{L}{L-1}\bigr)\,a_1^2\,.
\end{align*}
Hence, the final KL bound is
\begin{align*}
    \Sigma
    &\lesssim c\,\Bigl[ \frac{L-1}{L}\,d_0^2 + \frac{L-1}{L}\,Na_0^2 + Na_1^2 + \frac{L-1}{L}\,d_\star^2 + \log\bigl(\frac{L}{L-1}\bigr)\,a_0^2 + \frac{L}{L-1}\,a_1^2 \Bigr] \\
    &\qquad{} + c'\,\Bigl[ \frac{{(L-1)}^2}{L^2}\,d_0^2 + a_0^2 + \log\bigl(\frac{L}{L-1}\bigr)\,a_1^2\Bigr] \\
    &\lesssim (c+c')\,\Bigl[\frac{L-1}{L}\,d_0^2 + \frac{L-1}{L} \log\bigl(\frac{L}{L-1}\bigr)\,Na_0^2 + Na_1^2\Bigr]\,.
\end{align*}
In this case, $(L^{-1}-1)/(L^{-N}-1) \ge 1-L^{-1} = (L-1)/L$.

In both cases, we obtain
\begin{align*}
    \Sigma \lesssim (c+c')\,\Bigl[\frac{L^{-1}-1}{L^{-N}-1}\,d_0^2 + \bigl(\frac{(L-1)\,N}{L}\vee \log N\bigr)\, a_0^2 + Na_1^2\Bigr]\,.
\end{align*}
Combining all the cases proves the result.
\end{proof}

\begin{remark}[Tighter bound]\label{rmk:tighter_shift}
    A more refined bound can be read off from the proof; in particular, we can replace $c+c'$ in the final bound with $c + c'/(\log\bar N)$.
\end{remark}

\subsection{Mean-squared error framework}\label{app:mean_sq_err}

For completeness, in this section we prove Theorem~\ref{thm:local_error}.

\begin{proof}[Proof of Theorem~\ref{thm:local_error}]
    The assumptions imply
    \begin{align}
        W_2^2(\delta_x \hat P, \delta_y P)
        &\le L^2\,\norm{x-y}^2 + 2\,(\mc E_{\rm weak}(x) + \gamma \mc E_{\rm strong}(x))\,\norm{x-y} + {\mc E_{\rm strong}(x)}^2\,,
        \label{eq:local_error_1step_W2}
    \end{align}
    see, e.g.,~\cite[\S 5]{chewibook}.
    We now split into cases, depending on $L$.
    When $L < 1$, we apply Young's inequality on the cross term and bound it by $L\,(1-L)\,\norm{x-y}^2 + \frac{1}{L\,(1-L)}\,(\mc E_{\rm weak}(x) + \gamma \mc E_{\rm strong}(x))^2$, leading to the inequality
    \begin{align*}
        W_2^2(\delta_x\hat P,\delta_y P)
        &\le L \,\norm{x-y}^2 + \frac{1}{L\,(1-L)}\,(\mc E_{\rm weak}(x) + \gamma \mc E_{\rm strong}(x))^2 + {\mc E_{\rm strong}(x)}^2\,.
    \end{align*}
    Together with a coupling argument, this inequality can be iterated to obtain the desired bound.

    The other cases are similar. Namely, we bound the cross term for $L = 1$ by $\frac{1}{N}\,\norm{x-y}^2 + N\,(\mc E_{\rm weak}(x)+\gamma \mc E_{\rm strong}(x))^2$, and for $L > 1$ by $L^2\,(L-1) + \frac{1}{L^2\,(L-1)}\,(\mc E_{\rm strong}(x) +\gamma \mc E_{\rm strong}(x))^2$.
    The resulting inequalities are readily iterated and we omit the details.
\end{proof}

\section{Proofs for Section~\ref{sec:sampling}}

\subsection{Shifted Girsanov}\label{app:shifted_girsanov_rmk}

In this section, we show how to obtain a version of Theorem~\ref{thm:lmc} directly using the shifted Girsanov approach.
We assume that $-\beta I \preceq \alpha I \preceq \nabla^2 V \preceq \beta I$, and we follow the proof of Lemma~\ref{lem:cross_reg_lmc}.

We have the stochastic processes
\begin{align*}
    \D \hat X_t
    &= -\nabla V(\hat X_{t_-})\,\D t + \sqrt 2\,\D B_t'\,, & \hat X_0 &= x\,, \\
    \D Y_t'
    &= \{-\nabla V(Y_t') + \eta_t\,(\hat X_t - Y_t')\}\,\D t + \sqrt 2\,\D B_t'\,, & Y_0' &= y\,, \\
    &= -\nabla V(Y_t') \, \D t + \sqrt 2 \,\D B_t\,,
\end{align*}
where we introduce the notation $t_- \deq \lfloor t/h\rfloor \,h$.
We define the path measures $\mb P$, $\mb P'$ as before, so that
\begin{align*}
    \KL(\delta_x \hat P^N \mmid \delta_y P^N)
    \le \KL(\mb P' \mmid \mb P)
    = \frac{1}{4} \,\E^{\mb P'}\int_0^T\eta_t^2\,\norm{\hat X_t - Y_t'}^2\,\D t\,,
\end{align*}
where $T \deq Nh$.
We split the proof into two cases, based on whether or not $\alpha \ge 0$.
For simplicity, we use a simpler, suboptimal choice of shifts.

\paragraph*{Convex case.} Here, we assume that $\alpha \ge 0$.
We now have the inequality
\begin{align*}
    \frac{1}{2}\, \D(\norm{\hat X_t - Y_t'}^2)
    &\le \bigl[-(\alpha+\eta_t) \,\norm{\hat X_t - Y_t'}^2 + \langle \nabla V(\hat X_t) - \nabla V(\hat X_{t_-}), \hat X_t - Y_t'\rangle\bigr]\,\D t \\
    &\le \Bigl[- \frac{\alpha +\eta_t}{2}\,\norm{\hat X_t - Y_t'}^2 + \frac{\beta^2}{2\,(\alpha + \eta_t)}\,\norm{\hat X_t - \hat X_{t_-}}^2 \Bigr]\, \D t\,.
\end{align*}
Also,
\begin{align}\label{eq:shifted_girsanov_disc_error}
    \E^{\mb P'}[\norm{\hat X_t - \hat X_{t_-}}^2]
    &\lesssim h^2\,\E^{\mb P'}[\norm{\nabla V(\hat X_{t_-})}^2] + dh\,.
\end{align}
Let $\bar{\mf a}^2 \deq \max_{n=0,1,\dotsc,N-1}\{\beta^2 h^2\,\E^{\mb P'}[\norm{\nabla V(\hat X_{nh})}^2] + \beta^2 dh\}$ denote the maximum error.
Also, we take $\eta_t \deq \frac{2\alpha}{\exp(\alpha\,(T-t))-1}$ (in the case $\alpha = 0$, the expressions should be understood as their limiting values, e.g., $\eta_t = 2/(T-t)$).
Then, we obtain
\begin{align*}
    \partial_t\Bigl\{\exp\Bigl(\alpha t + \int_0^t \eta_s \, \D s\Bigr)\,\E^{\mb P'}[\norm{\hat X_t - Y_t'}^2]\Bigr\}
    &\lesssim \frac{\bar{\mf a}^2}{\alpha + \eta_t} \exp\Bigl(\alpha t + \int_0^t \eta_s \, \D s\Bigr) \\
    &= \frac{\bar{\mf a}^2\exp(\alpha t)\,\{\exp(\alpha\,(T-t)) - 1\}}{\alpha\,\{\exp(\alpha\,(T-t)) + 1\}}\, \Bigl( \frac{1-\exp(-\alpha T)}{1-\exp(-\alpha\,(T-t))}\Bigr)^2 \\
    &\le \frac{\bar{\mf a}^2 \exp(\alpha T)\,\{1-\exp(-\alpha T)\}^2}{\alpha\,\{\exp(\alpha\,(T-t)) - 1\}}\,.
\end{align*}
Integrating,
\begin{align*}
    &\exp\Bigl(\alpha t + \int_0^t \eta_s \,\D s\Bigr)\,\E^{\mb P'}[\norm{\hat X_t - Y_t'}^2] \\
    &\qquad \le \norm{x-y}^2 + O\biggl(\frac{\bar{\mf a}^2\exp(\alpha T)\,\{1-\exp(-\alpha T)\}^2}{\alpha^2} \log \frac{\exp(\alpha T) - 1}{\exp(\alpha T) - \exp(\alpha t)}\biggr)
\end{align*}
or
\begin{align}\label{eq:shifted_girsanov_cvx_dist}
    \begin{aligned}
        \E^{\mb P'}[\norm{\hat X_t - Y_t'}^2]
        &\le \exp(-\alpha t)\,\Bigl( \frac{1-\exp(-\alpha\,(T-t))}{1-\exp(-\alpha T)}\Bigr)^2\,\norm{x-y}^2 \\
        &\qquad{} + O\biggl( \frac{\bar{\mf a}^2 \exp(\alpha\,(T-t)) \,\{1-\exp(-\alpha\,(T-t))\}^2}{\alpha^2}\log \frac{\exp(\alpha T) - 1}{\exp(\alpha T) - \exp(\alpha t)} \biggr)\,.
    \end{aligned}
\end{align}
It yields the KL divergence bound
\begin{align*}
    \KL(\mb P' \mmid \mb P)
    &\lesssim \int_0^T \Bigl[\frac{\alpha^2\exp(-2\alpha T + \alpha t))\,\norm{x-y}^2}{\{1-\exp(-\alpha T)\}^2} + \bar{\mf a}^2 \exp(-\alpha \,(T-t))\log \frac{\exp(\alpha T) - 1}{\exp(\alpha T) - \exp(\alpha t)} \Bigr]\,\D t \\
    &\lesssim \frac{\alpha\,\norm{x-y}^2}{\exp(\alpha T) - 1} + \bar{\mf a}^2 \underbrace{\int_0^T \exp(-\alpha \,(T-t))\log \frac{\exp(\alpha T) - 1}{\exp(\alpha T) - \exp(\alpha t)} \,\D t}_{\eqqcolon \mc I_0^T}\,.
\end{align*}
We evaluate the last integral by splitting the region of integration.
First,
\begin{align*}
    \mc I_0^{T-1/\alpha}
    \deq \int_0^{(T-1/\alpha)\vee 0} \exp(-\alpha\,(T-t)) \log \frac{1-\exp(-\alpha T)}{1-\exp(-\alpha\,(T-t))} \, \D t
    &\lesssim \int_0^{(T-1/\alpha)\vee 0} \exp(-\alpha\,(T-t)) \, \D t \\
    &\lesssim T \wedge \frac{1}{\alpha}\,.
\end{align*}
For the remaining part, if $T \le 1/\alpha$,
\begin{align*}
    \mc I_{T-1/\alpha}^T
    \deq \int_{(T-1/\alpha) \vee 0}^T\exp(-\alpha\,(T-t)) \log \frac{1-\exp(-\alpha T)}{1-\exp(-\alpha\,(T-t))} \, \D t
    &\le \int_0^T \log \frac{T}{2\,(T-t)} \, \D t
    \lesssim T\,.
\end{align*}
On the other hand, if $T \ge 1/\alpha$, then
\begin{align*}
    \mc I_{T-1/\alpha}^T
    &\le \int_{T-1/\alpha}^T \log \frac{1}{2\alpha\,(T-t)} \, \D t
    \lesssim \frac{1}{\alpha}\,.
\end{align*}
Combining these estimates, we obtain the following bound:
\begin{align*}
    \boxed{\KL(\delta_x \hat P^N \mmid \delta_y P^N)
    \lesssim \frac{\alpha\,\norm{x-y}^2}{\exp(\alpha T) - 1} + \bigl( T \wedge \frac{1}{\alpha}\bigr)\,\bar{\mf a}^2}
\end{align*}

\paragraph*{Semi-convex case.}
We now consider the case $\alpha < 0$, and for simplicity, we set $\alpha = -\beta$.
In this case, assuming that $\eta_t > \beta$ (so that the use of Young's inequality is valid), we arrive as before at
\begin{align*}
    &\partial_t\Bigl\{\exp\Bigl(-\beta t + \int_0^t \eta_s \, \D s\Bigr)\,\E^{\mb P'}[\norm{\hat X_t - Y_t'}^2]\Bigr\} \\
    &\qquad \lesssim \frac{\bar{\mf a}^2\exp(-\beta t)\,\{1-\exp(-\beta\,(T-t))\}}{\beta\,\{1+\exp(-\beta\,(T-t))\}}\, \Bigl( \frac{\exp(\beta T)-1}{\exp(\beta\,(T-t))-1}\Bigr)^2
    \le \frac{\bar{\mf a}^2 \exp(-\beta T)\,\{\exp(\beta T)-1\}^2}{\beta\,\{\exp(\beta\,(T-t)) - 1\}}\,.
\end{align*}
Integrating,
\begin{align*}
    &\exp\Bigl(-\beta t + \int_0^t \eta_s \, \D s\Bigr)\,\E^{\mb P'}[\norm{\hat X_t - Y_t'}^2] \\
    &\qquad \lesssim \norm{x-y}^2 + \frac{\bar{\mf a}^2\exp(-\beta T)\,\{\exp(\beta T) - 1\}^2}{\beta^2} \log \frac{\exp(\beta T) - 1}{\exp(\beta T) - \exp(\beta t)}\,,
\end{align*}
which yields
\begin{align}\label{eq:shifted_girsanov_semicvx_dist}
    \begin{aligned}
        \E^{\mb P'}[\norm{\hat X_t - Y_t'}^2]
        &\lesssim \exp(\beta t)\,\Bigl(\frac{\exp(\beta\,(T-t)) - 1}{\exp(\beta T) - 1}\Bigr)^2\,\norm{x-y}^2 \\
        &\qquad{} + \frac{\bar{\mf a}^2\exp(-\beta\,(T-t)) \, \{\exp(\beta\,(T-t)) - 1\}^2}{\beta^2} \log \frac{\exp(\beta T) - 1}{\exp(\beta T) - \exp(\beta t)}\,.
    \end{aligned}
\end{align}
It yields the KL divergence bound
\begin{align*}
    \KL(\mb P' \mmid \mb P)
    &\lesssim \frac{\beta\,\norm{x-y}^2}{1-\exp(-\beta T)} + \bar{\mf a}^2 \int_0^T \exp(\beta\,(T-t)) \log \frac{\exp(\beta T) - 1}{\exp(\beta T) - \exp(\beta t)}\,.
\end{align*}
A similar case-by-case analysis of the last integral shows that it is of order $T$, hence:
\begin{align}\label{eq:lmc_shifted_girsanov_semicvx}
    \boxed{\KL(\delta_x \hat P^N \mmid \delta_y P^N) \lesssim \frac{\beta\,\norm{x-y}^2}{1-\exp(-\beta T)} + T\bar{\mf a}^2}
\end{align}
In particular, this establishes the semi-convex case of Lemma~\ref{lem:cross_reg_lmc}.

\subsection{Extension to R\'enyi divergences}\label{app:renyi_shifted_girsanov}

In this section, we extend the analysis of \S\ref{app:shifted_girsanov_rmk} to R\'enyi divergences of all orders $q > 1$.

First, we briefly recall the definition of the sub-Gaussian Orlicz norm.
We recall that a random variable is sub-Gaussian if and only if the sub-Gaussian Orlicz norm is finite. We refer the reader to standard textbooks such as~\cite{Rao91book} and~\cite{vershynin2018high} for further background on Orlicz norms in the context of analysis and probability, respectively.

\begin{defin}[Sub-Gaussian Orlicz norm]
	The sub-Gaussian Orlicz norm $\|\cdot\|_{\psi_2}$ is defined as
	\begin{align*}
		\|X\|_{\psi_2}  \deq  \inf \{\lambda > 0 \; : \; \E \exp(\|X\|^2/\lambda^2) \leq e \}\,.
	\end{align*}
\end{defin}

We follow the notation and setup of \S\ref{app:shifted_girsanov_rmk}. %
Here, by Girsanov's theorem and Cauchy{--}Schwarz,
\begin{align*}
    \Ren_q(\mb P' \mmid \mb P)
    &= \frac{1}{q-1} \log \E^{\mb P'} \exp\Bigl(\frac{q-1}{\sqrt 2} \int_0^T \eta_t\,\langle Y_t' - \hat X_t, \D B_t'\rangle + \frac{q-1}{4} \int_0^T \eta_t^2\,\norm{\hat X_t - Y_t'}^2\,\D t \Bigr) \\
    &\le \frac{1}{2\,(q-1)} \log \underbrace{\E^{\mb P'} \exp\Bigl(\sqrt 2\,(q-1) \int_0^T \eta_t\,\langle Y_t'-\hat X_t,\D B_t'\rangle - {(q-1)}^2 \int_0^T \eta_t^2\,\norm{\hat X_t - Y_t'}^2 \, \D t \Bigr)}_{\le 1} \\
    &\qquad{} + \frac{1}{2\,(q-1)} \log \E^{\mb P'} \exp\Bigl(\bigl({(q-1)}^2 + \frac{q-1}{2}\bigr) \int_0^T \eta_t^2 \,\norm{\hat X_t - Y_t'}^2 \, \D t\Bigr) \\
    &\le \frac{1}{2\,(q-1)} \log \E^{\mb P'} \exp\Bigl(q\,(q-1) \int_0^T \eta_t^2 \,\norm{\hat X_t - Y_t'}^2 \, \D t\Bigr)\,,
\end{align*}
where the inequality for the underlined term follows from It\^o's lemma.

We now substitute the inequality~\eqref{eq:shifted_girsanov_disc_error} with the pointwise bound
\begin{align*}
    \norm{\hat X_t - \hat X_{t_-}}^2
    &\lesssim h^2\,\norm{\nabla V(\hat X_{t_-})}^2 + \norm{B_t' - B_{t_-}'}^2\,.
\end{align*}
If we follow the calculations of \S\ref{app:shifted_girsanov_rmk}, now working pointwise rather than in expectation, we arrive at the following version of the bounds~\eqref{eq:shifted_girsanov_cvx_dist} and~\eqref{eq:shifted_girsanov_semicvx_dist}:
\begin{align*}
    \norm{\hat X_t - Y_t'}^2
    &\le A_t \,\norm{x-y}^2 + C_t\, \mf a^2\,,
\end{align*}
for some $A_t, B_t \ge 0$ (which depend on $\alpha$ and $T$), and
\begin{align*}
    \mf a^2 \lesssim \max_{n=0,1,\dotsc,N-1}\Bigl\{\beta^2 h^2\,\norm{\nabla V(\hat X_{nh})}^2 + \beta^2 \sup_{t\in [nh, (n+1)h]}{\norm{B_t' - B_{nh}'}^2}\Bigr\}\,.
\end{align*}
Therefore, we obtain
\begin{align}
    \Ren_q(\mb P' \mmid \mb P)
    &\le \frac{q\,\norm{x-y}^2}{2} \int_0^T \eta_t^2 A_t \, \D t + \frac{1}{2\,(q-1)} \log \E^{\mb P'} \exp\Bigl(q\,(q-1)\, \mf a^2 \int_0^T \eta_t^2 C_t \, \D t\Bigr) \label{eq:renyi_girsanov_splitting} \\
    &\le \frac{q\,\norm{x-y}^2}{2} \int_0^T \eta_t^2 A_t \, \D t + \frac{q\,\norm{\mf a}_{\psi_2}^2}{2} \int_0^T \eta_t^2 C_t \, \D t\,, \nonumber
\end{align}
provided that
\begin{align}\label{eq:renyi_shifted_girsanov_condition}
    q\,(q-1)\,\norm{\mf a}_{\psi_2}^2 \int_0^T \eta_t^2 C_t \, \D t \le 1\,.
\end{align}
Under this condition, from the calculations in \S\ref{app:shifted_girsanov_rmk}, we can read off the following bounds:
\begin{align*}
    \Ren_q(\delta_x \hat P^N \mmid \delta_y P^N)
    &\lesssim \frac{\alpha q\,\norm{x-y}^2}{\exp(\alpha T) - 1} + q\,\bigl(T \wedge \frac{1}{\alpha}\bigr)\,\norm{\mf a}_{\psi_2}^2
\end{align*}
in the case $\alpha > 0$, and
\begin{align*}
    \Ren_q(\delta_x \hat P^N \mmid \delta_y P^N)
    &\lesssim \frac{q\beta\,\norm{x-y}^2}{1-\exp(-\beta T)} + qT\,\norm{\mf a}_{\psi_2}^2
\end{align*}
in the case $\alpha = -\beta$, provided that the condition~\eqref{eq:renyi_shifted_girsanov_condition} is met.

However,~\eqref{eq:renyi_shifted_girsanov_condition} is often not sharp and can be improved with a more careful argument.
We conclude this section by doing so for the case $N = 1$, thereby establishing a R\'enyi version of the cross-regularity lemma (Lemma~\ref{lem:cross_reg_lmc}).

\begin{lemma}[R\'enyi cross-regularity for LMC]\label{lem:cross_reg_lmc_renyi}
    Let $P$, $\hat P$ denote the kernels corresponding to~\ref{eq:langevin} run for time $h$ and~\ref{eq:lmc} respectively, and let $q > 1$.
    Assume that $-\beta I \preceq \nabla^2 V \preceq \beta I$ and that $h \lesssim \frac{1}{\beta \sqrt{q\,(q-1)}}$ for a sufficiently small implied constant.
    Then, for all $x,y\in\R^d$,
    \begin{align*}
        \Ren_q(\delta_x \hat P \mmid \delta_y P)
        \lesssim \frac{q\,\norm{x-y}^2}{h} + \beta^2 h^3 q \,\norm{\nabla V(x)}^2 + \beta^2 dh^2 q\,.
    \end{align*}
\end{lemma}
\begin{proof}
    Here, $\mf a^2 \lesssim \beta^2 h^2 \,\norm{\nabla V(x)}^2 + \beta^2 \sup_{t\in [0,h]}{\norm{B_t'}^2}$.
    In~\eqref{eq:renyi_girsanov_splitting}, we can add and subtract any constant $\mf a_0^2 > 0$ before evaluating the expectation, leading to
    \begin{align*}
        \Ren_q(\mb P' \mmid \mb P)
        &\lesssim q\,\norm{x-y}^2 \int_0^h \eta_t^2 A_t \, \D t + q\,\mf a_0^2 \int_0^h \eta_t^2 C_t \, \D t \\
        &\qquad{} + \frac{1}{q-1} \log \E^{\mb P'} \exp\Bigl(q\,(q-1)\, (\mf a^2 - \mf a_0^2) \int_0^h \eta_t^2 C_t \, \D t\Bigr)\,.
    \end{align*}
    In particular, choosing $\mf a_0^2 \asymp \beta^2 h^2\,\norm{\nabla V(x)}^2$, we have $\mf a^2 - \mf a_0^2 \lesssim \beta^2 \sup_{t\in [0,h]}{\norm{B_t'}^2}$.
    Then,
    \begin{align*}
        \log \E^{\mb P'} \exp\Bigl(\beta^2 q\,(q-1)\,\sup_{t\in [0,h]}{\norm{B_t'}^2} \int_0^h \eta_t^2 C_t \, \D t\Bigr)
        &= d\log \E^{\mb P'} \exp\Bigl(\beta^2 q\,(q-1)\,\sup_{t\in [0,h]}{\abs{B_{t,1}^2}^2} \int_0^h \eta_t^2 C_t \, \D t\Bigr)
    \end{align*}
    where ${\{B_{t,1}'\}}_{t\ge 0}$ denotes the first coordinate of the Brownian motion.
    The desired R\'enyi bound now follows from standard estimates (e.g.,~\cite[Lemma 23]{chewi2021optimal}), provided that the condition $\beta^2 hq\,(q-1)\int_0^h \eta_t^2 C_t\,\D t \lesssim 1$ holds.
    It suffices to have $\beta h\sqrt{q\,(q-1)} \lesssim 1$.
\end{proof}

\subsection{Gradient estimates}\label{app:grad}

Here, we establish useful lemmas to control the gradient terms in the proofs. This is helpful throughout our applications to sampling since local error analysis bounds involve the squared gradient norm, and therefore one needs bounds on these quantities to obtain final iteration complexities. The first lemma, simpler, bounds the squared gradient norm at a measure $\mu$ by the squared gradient norm at the stationary measure $\pi$, which is $O(\beta d)$, plus either the Wasserstein or KL divergence between $\mu$ and $\pi$. The second lemma then shows how to combine this in a user-friendly way with the Wasserstein and KL divergence bounds produced by local error analysis. 

\begin{lemma}[Gradient bound]\label{lem:grad_bd}
    Suppose that $-\beta I \preceq \nabla^2 V \preceq \beta I$ and $\pi\propto\exp(-V)$.
    Then, for any probability measure $\mu$,
    \begin{align*}
        \E_\mu[\norm{\nabla V}^2]
        \lesssim \beta d + \min\{\beta^2\,W_2^2(\mu,\pi),\,\beta \KL(\mu \mmid \pi)\}\,.
    \end{align*}
\end{lemma}
\begin{proof}
    For the $W_2$ bound, we simply use $\E_\mu[\norm{\nabla V}^2] \lesssim \E_\pi[\norm{\nabla V}^2] + \beta^2\,W_2^2(\mu,\pi)$ together with the bound $\E_\pi[\norm{\nabla V}^2] \le \beta d$ (cf.~\cite[\S 4]{chewibook}, ``basic lemma'').

    For the KL bound, the Donsker{--}Varadhan variational principle, for any $\lambda > 0$,
    \begin{align*}
        \E_\mu[\norm{\nabla V}^2]
        &\le \frac{1}{\lambda}\,\{\KL(\mu \mmid \pi) + \log \E_\pi \exp(\lambda\,\norm{\nabla V}^2)\}\,.
    \end{align*}
    The second term in braces is controlled using the sub-Gaussian concentration of $\norm{\nabla V}$ (cf.~\cite{Neg22Thesis} or~\cite[Corollary 5.5]{scr2}): for $\lambda \asymp 1/\beta$, the second term is bounded by $O(d)$.
\end{proof}

\begin{lemma}[Recursive gradient control]\label{lem:recursive_grad}
    Suppose that $-\beta I \preceq \nabla^2 V \preceq \beta I$ and $\pi \propto \exp(-V)$.
    We write $G_n^2 \deq \max_{k < n} \E_{\hat\mu_0\hat P^k}[\norm{\nabla V}^2]$.
    \begin{enumerate}
        \item Suppose that the following bound holds for all iterations $n \le n_0$:
        \begin{align}\label{eq:w2_iteration}
            W_2^2(\hat\mu_0\hat P^n, \pi)
            &\leq \msf A^2 G_n^2 + \msf B^2\,.
        \end{align}
        If $\msf A\beta \lesssim 1$ for a sufficiently small implied constant, then for all $n \le n_0$,
        \begin{align*}
            W_2^2(\hat\mu_0\hat P^n, \pi)
            &\lesssim \msf A^2 \beta d + \msf B^2 \qquad\text{and}\qquad G_n^2 \lesssim \beta d + \msf B^2 \beta^2\,.
        \end{align*}
        \item Suppose that the following bound holds for all iterations $n \ge n_0$:
        \begin{align}\label{eq:kl_iteration}
            \KL(\hat\mu_0\hat P^n \mmid \pi)
            \le \msf C^2 G_n^2 + \msf D^2\,.
        \end{align}
        If $\msf C^2\beta \lesssim 1$ for a sufficiently small implied constant, then for all $n \ge n_0$,
        \begin{align*}
            \KL(\hat\mu_0\hat P^n \mmid \pi)
            &\lesssim \msf C^2 G_{n_0}^2 + \msf C^2 \beta d + \msf D^2 \qquad\text{and}\qquad G_n^2 \lesssim G_{n_0}^2 + \beta d + \msf D^2\beta\,.
        \end{align*}
    \end{enumerate}
\end{lemma}
\begin{proof}\mbox{}
    \begin{enumerate}
        \item By Lemma~\ref{lem:grad_bd},~\eqref{eq:w2_iteration} implies
        \begin{align*}
            \max_{k\le n} W_2^2(\hat\mu_0\hat P^k,\pi)
            &\lesssim \msf A^2 \beta^2 \max_{k<n} W_2^2(\hat\mu_0 \hat P^k, \pi) + \msf A^2 \beta d + \msf B^2\,.
        \end{align*}
        If we assume that $\msf A\beta \lesssim 1$ with a sufficiently small implied constant, then the first term on the right-hand side can be absorbed back into the left-hand side which yields the $W_2^2$ bound. The bound on $G_n^2$ then follows from Lemma~\ref{lem:grad_bd}.
        \item By Lemma~\ref{lem:grad_bd},~\eqref{eq:kl_iteration} implies
        \begin{align*}
            \max_{n_0\le k\le n} \KL(\hat\mu_0\hat P^k \mmid \pi)
            &\lesssim \msf C^2 \beta \max_{n_0\le k < n} \KL(\hat\mu_0 \hat P^k \mmid \pi) + \msf C^2 G_{n_0}^2 + \msf C^2 \beta d + \msf D^2\,.
        \end{align*}
        The remaining steps then mirror the proof of the first item. 
    \end{enumerate}
\end{proof}

    \section{Proofs for Section~\ref{sec:lmc_smooth}}\label{app:lmc_smooth}

\subsection{Proof of Lemma~\ref{lem:lmc_higher_local}}\label{app:lmc_higher_local}

\begin{proof}[Proof of Lemma~\ref{lem:lmc_higher_local}]
    We follow the proof of~\cite{LiZhaTao22LMCSqrtd}.
    By It\^o's formula,
    \begin{align*}
        &\norm{\E \hat X_h - \E X_h}
        = \Bigl\lVert \int_0^h \{\E \nabla V(X_t) -\nabla V(x)\}\, \D t\Bigr\rVert \\
        &\qquad = \Big\lVert\E \Bigl[\int_0^h \int_0^t \{-\nabla^2 V(X_s)\, \nabla V(X_s) + \nabla \Delta V(X_s)\} \, \D s \, \D t + \sqrt{2} \int_0^h \int_0^t \nabla^2 V(X_s) \, \D B_s\,\D t \Bigr]\Big\rVert \\
        &\qquad \leq \int_0^h \int_0^t \bigl\{\E \norm{\nabla^2 V(X_s)\, \nabla V(X_s)} + \E\norm{\nabla \Delta V(X_s)} \bigr\}\, \D s \, \D t \\
        &\qquad \leq (\beta + \zeta_1) \int_0^h \int_0^t \E \norm{\nabla V(X_s)}\, \D s \, \D t + \frac{\zeta_0 h^2}{2}\,.
    \end{align*}
    To handle the gradient term, we write
    \begin{align*}
        \E\norm{\nabla V(X_s)} &\leq \norm{\nabla V(x)} + \beta\,\E\norm{X_s - x} \\
        &\lesssim \norm{\nabla V(x)} + \beta \int_0^s \E \norm{\nabla V(X_r)} \, \D r + \beta\, \E \norm{B_s}\,.
    \end{align*}
    By Gr\"onwall's inequality, for $h \lesssim 1/\beta$, we obtain the following bound, which allows us to conclude:
    \begin{align*}
        \E\norm{\nabla V(X_s)}
        &\lesssim \norm{\nabla V(x)} + \beta\sqrt{dh} \,. \qedhere
    \end{align*}
\end{proof}

\subsection{Proof of Theorems~\ref{thm:lmc_smooth} and~\ref{thm:lmc_smooth_lsi}}\label{app:lmc_smooth_pf}

\begin{proof}[Proof of Theorem~\ref{thm:lmc_smooth}]
    We now apply Theorem~\ref{thm:kl-general} with $L = \exp(-\alpha h)$, $\gamma = \beta h$, $c = O(1/h)$ (by Lemma~\ref{lem:ld_properties}), $b(x) = O(\beta h^{3/2}\,\norm{\nabla V(x)} + \beta d^{1/2} h)$, $c' = O(1/h)$ (by Lemma~\ref{lem:cross_reg_lmc}), and
    \begin{align}
        \mc E_{\rm weak}(x)
        &\lesssim (\beta + \zeta_1)\, h^2\, \norm{\nabla V(x)} + (\beta+\zeta_1)\,\beta d^{1/2} h^{5/2} +\zeta_0 h^2\,,\label{eq:lmc_smooth_weak} \\
        \mc E_{\rm strong}(x)
        &\lesssim \beta h^2\,\norm{\nabla V(x)} + \beta d^{1/2} h^{3/2}\,,\label{eq:lmc_smooth_strong}
    \end{align}
    by Lemmas~\ref{lem:lmc_strong_err} and~\ref{lem:lmc_higher_local}.

    Following the proof of Theorem~\ref{thm:lmc}, we apply Theorems~\ref{thm:local_error} and~\ref{thm:kl-general} to obtain
    \begin{align*}
        W_2^2(\hat\mu_0\hat P^n, \pi)
        &\lesssim \exp(-\alpha nh)\,W_2^2(\hat\mu_0,\pi) + \bigl(n \wedge\frac{1}{\alpha h}\bigr)^2\,((\beta+\zeta_1)^2 \,h^4 G_n^2 + (\beta+\zeta_1)^2\,\beta^2 dh^5 + \zeta_0^2 h^4) \\
        &\qquad{} + \bigl(n \wedge \frac{1}{\alpha h}\bigr)\,\beta^2 dh^3\,,
    \end{align*}
    and
    \begin{align}
        \KL(\hat\mu_0\hat P^n \mmid \pi)
        &\lesssim \frac{\alpha W_2^2(\hat\mu_0,\pi)}{\exp(\alpha nh) -1} + \bigl(n \wedge \frac{1}{\alpha h}\bigr)\, ((\beta+\zeta_1)^2 \,h^3 G_n^2 + (\beta+\zeta_1)^2\,\beta^2 dh^4 + \zeta_0^2 h^3) \nonumber \\
        &\qquad{} + \log\bigl(n\wedge \frac{1}{\alpha h}\bigr)\,\beta^2 dh^2\,. \label{eq:lmc_smooth_kl_recursion}
    \end{align}

    In the case $\alpha > 0$, we invoke Lemma~\ref{lem:recursive_grad} with $n_0 = \infty$, $\msf A^2 = O((\kappa+\bar\kappa_1)^2\,h^2)$, $\msf B^2 = O(d/\alpha)$, provided that
    \begin{align*}
        h \lesssim \frac{1}{\beta}\,\Bigl[\frac{1}{\kappa+\bar\kappa_1}\wedge \frac{1}{(1+\zeta_1/\beta)^{1/3}\,(\kappa+\bar\kappa_1)^{1/3}} \wedge \frac{\kappa d^{1/2}}{\bar\kappa_0}\Bigr]\,.
    \end{align*}
    It yields $G_n^2\lesssim \beta \kappa d$ for all $n$, which we substitute into~\eqref{eq:lmc_smooth_kl_recursion} to prove the result in this case.

    In the case $\alpha = 0$, we invoke Lemma~\ref{lem:recursive_grad} with $N = \Theta(W^2/(\varepsilon^2 h))$, $n_0 = \Theta(1/(\beta h))$, $\msf A^2 = O((1+\zeta_1/\beta)^2\,h^2)$, $\msf B^2 = O(W^2)$, $\msf C^2 = O((\beta+\zeta_1)^2\,W^2 h^2/\varepsilon^2)$, $\msf D^2 = O(\beta W^2)$, provided that $h$ is sufficiently small---which is satisfied for our eventual choice of $h$, for sufficiently small $\varepsilon$---where we write $W \deq W_2(\hat\mu_0,\pi)$.
    It yields $G_N^2 \lesssim \beta d + \beta^2 W^2$, which is then substituted back into~\eqref{eq:lmc_smooth_kl_recursion} to complete the proof.
\end{proof}

\begin{proof}[Proof of Theorem~\ref{thm:lmc_smooth_lsi}]
    We follow the proof of Theorem~\ref{thm:lmc_lsi}.
    By the R\'enyi weak triangle inequality (Corollary~\ref{kl:weak-triangle}), Theorem~\ref{thm:langevin_renyi_conv}, and Theorem~\ref{thm:kl-general},
    \begin{align*}
        \KL(\hat\mu_0\hat P^n \mmid \pi)
        &\lesssim \exp(-\alpha nh)\,\widetilde O(d) + n\,((\beta+\zeta_1)^2\,h^3 G_n^2 + (\beta+\zeta_1)^2\,\beta^2 dh^4 + \zeta_0^2 h^3) \\
        &\qquad{} + (\beta nh \vee \log n)\,\beta^2 dh^2\,.
    \end{align*}
    The total iteration count is taken to be $N = \widetilde\Theta(1/(\alpha h))$.
    We invoke Lemma~\ref{lem:recursive_grad} with $n_0 = 0$, $\msf C^2 = \widetilde O((\beta+\zeta_1)\,(\kappa+\bar\kappa_1)\,h^2)$, $\msf D^2 = \widetilde O(d)$, provided
    \begin{align*}
        h\le \frac{1}{\beta}\,\widetilde O\Bigl(\frac{1}{\sqrt{(1+\zeta_1/\beta)\,(\kappa +\bar \kappa_1)}} \wedge \frac{\kappa d^{1/2}}{\bar\kappa_0} \Bigr)\,.
    \end{align*}
    It yields $G_N^2 \lesssim \beta d$, and substituting this into the bound on the KL divergence finishes the proof.
\end{proof}

    \section{Proofs for Section~\ref{sec:rmd}}

\subsection{Proof of Lemma~\ref{lem:cross_reg_rmlmc}}\label{app:cross_reg_rmlmc}

In this section, we prove the cross-regularity bound for~\ref{eq:rmlmc}.

\begin{proof}[Proof of Lemma~\ref{lem:cross_reg_rmlmc}]
At a high level, the proof consists of conditioning on the uniform random variable $u$ and splitting into cases.
For small $u$, there is enough noise added after time $uh$ so that we can apply the cross-regularity bound for LMC (Lemma~\ref{lem:cross_reg_lmc}).
For large $u$, we use a different argument based on change of variables.

Let $\hat P^u$ denote the kernel for~\ref{eq:rmlmc}, conditioned on the value of the uniform random variable $u$.
By the convexity of the KL divergence (Proposition~\ref{prop:kl_prop}), it suffices to bound $\KL(\delta_x \hat P^u \mmid \delta_y P)$ for each $u\in [0,1]$ separately.

\underline{\textbf{Argument for small $u$.}}
In this case, we condition on ${\{B_t\}}_{t\in [0,uh]}$ and apply the joint convexity of the KL divergence (Proposition~\ref{prop:kl_prop}), leading to
\begin{align*}
    \KL(\delta_x \hat P^u \mmid \delta_y P)
    \le \E \KL\bigl(\mc N(x-h\,\nabla V(\hat X_{uh}^+) + \sqrt 2\,B_{uh},\, 2h\,(1-u)\,I) \bigm\Vert \delta_{X_{uh}} P_{(1-u)\,h}\bigr)
\end{align*}
where $P_{(1-u)\,h}$ is the Langevin kernel run for time $(1-u)\,h$, ${\{X_t\}}_{t\in [0,h]}$ is the Langevin diffusion started at $y$, and the expectation is taken over ${\{B_t\}}_{t\in [0,uh]}$.
We now follow the proof of Lemma~\ref{lem:cross_reg_lmc}, except that we replace $x$ with $x-uh\,\nabla V(\hat X_{uh}^+) + \sqrt 2\, B_{uh}$, $y$ with $X_{uh}$, and we replace the drift of the discretized process by $-\nabla V(\hat X_{uh}^+)$.
By tracing through the computations, we see that
\begin{align*}
    \KL(\delta_x \hat P^u \mmid \delta_y P)
    &\lesssim \E\Bigl[\frac{\norm{x-uh\,\nabla V(\hat X_{uh}^+)+\sqrt 2\,B_{uh} - X_{uh}}^2}{(1-u)\,h}
    + \beta^2 h^3 \,\norm{\nabla V(\hat X_{uh}^+)}^2 + \beta^2 dh^2\Bigr]\,.
\end{align*}
It remains to control these terms.

First,
\begin{align*}
    \E[\norm{\hat X_{uh}^+ - x}^2]
    &= \E[\norm{-uh\,\nabla V(x) + \sqrt 2 \, B_{uh}}^2]
    \lesssim h^2 \,\norm{\nabla V(x)}^2 + dh\,.
\end{align*}
Therefore,
\begin{align*}
    \E[\norm{\nabla V(\hat X_{uh}^+)}^2]
    &\lesssim \norm{\nabla V(x)}^2 + \beta^2\,\E[\norm{\hat X_{uh}^+ - x}^2]
    \lesssim \norm{\nabla V(x)}^2 + \beta^2 dh\,.
\end{align*}
Next,
\begin{align*}
    &\E[\norm{x-uh\,\nabla V(\hat X_{uh}^+) + \sqrt 2\,B_{uh} - X_{uh}}^2] \\
    &\qquad \lesssim \E[\norm{x-uh\,\nabla V(\hat X_{uh}^+) + \sqrt 2\,B_{uh} - \hat X_{uh}^+}^2] + \E[\norm{\hat X_{uh}^+ - X_{uh}}^2]\,.
\end{align*}
The first term here equals
\begin{align*}
    \E[\norm{uh\,(\nabla V(x) - \nabla V(\hat X_{uh}^+))}^2]
    &\lesssim \beta^2 h^2\,\E[\norm{\hat X_{uh}^+ - x}^2]
    \lesssim \beta^2 h^4 \,\norm{\nabla V(x)}^2 + \beta^2 dh^3\,.
\end{align*}
The second term, by the second part of Lemma~\ref{lem:lmc_strong_err}, is bounded by
\begin{align*}
    \E[\norm{\hat X_{uh}^+ - X_{uh}}^2]
    &\lesssim \norm{x-y}^2 + \beta^2 h^4\,\norm{\nabla V(x)}^2 + \beta^2 dh^3\,,
\end{align*}
since $h \lesssim 1/\beta$.
Putting the bounds together,
\begin{align*}
    \KL(\delta_x \hat P^u \mmid \delta_y P)
    &\lesssim \frac{\norm{x-y}^2}{(1-u)\,h} + \beta^2 h^3 \,\norm{\nabla V(x)}^2 + \beta^2 dh^2\,.
\end{align*}

\underline{\textbf{Argument for large $u$.}}
For this part of the argument, it turns out that we only require relatively crude bounds.
Let $\hat P^{\msf{LMC}}$ denote the transition kernel for~\ref{eq:lmc}.
By the weak triangle inequality for R\'enyi divergences (Corollary~\ref{kl:weak-triangle}),
\begin{align*}
    \KL(\delta_x \hat P^u \mmid \delta_y P)
    &\le 2\KL(\delta_x \hat P^u \mmid \delta_x \hat P^{\msf{LMC}}) +\Ren_2(\delta_x \hat P^{\msf{LMC}} \mmid \delta_y P)\,.
\end{align*}
The second term is bounded via Lemma~\ref{lem:cross_reg_lmc_renyi}, provided $h\lesssim 1/\beta$.
We focus on the first term.

Define the measures
\begin{align*}
    \rho_0
    &\deq \law\bigl(x-h\,\nabla V(x-uh\,\nabla V(x)+\sqrt{2}\,B_{uh}) + \sqrt 2\,B_{uh}\bigr)\,, \\
    \rho_1
    &\deq \law\bigl(x-h\,\nabla V(x) + \sqrt 2\,B_{uh}\bigr)\,,
\end{align*}
and note that
\begin{align*}
    \delta_x \hat P^u = \rho_0 * \mc N(0,\, 2\,(1-u)\,h\,I)\,, \qquad \delta_x \hat P^{\msf{LMC}} = \rho_1 * \mc N(0,\,2\,(1-u)\,h\,I)\,.
\end{align*}
By the data-processing inequality (Proposition~\ref{prop:kl_prop}),
\begin{align*}
    \KL(\delta_x \hat P^u \mmid \delta_x \hat P^{\msf{LMC}})
    &\le \KL(\rho_0 \mmid \rho_1)\,.
\end{align*}
Moreover, if $\gamma_{uh} \deq \mc N(0, uhI)$, then $\rho_0 = {(F_0)}_\# \gamma_{uh}$ and $\rho_1 = {(F_1)}_\# \gamma_{uh}$ for appropriate deterministic maps $F_0, F_1 : \R^d\to\R^d$.
The change of variables formula yields $\gamma/(F_\# \gamma \circ F) = \det \nabla F$ for any smooth density $\gamma$ and any diffeomorphism $F : \R^d\to\R^d$.
Hence,
\begin{align*}
    \KL(\rho_0 \mmid \rho_1)
    &= \int \bigl(\log \frac{\D\rho_0}{\D\rho_1}\bigr)\,\D \rho_0
    = \int \Bigl( \log \frac{\rho_0 \circ F_0}{\rho_1\circ F_0}\Bigr)\,\D \gamma_{uh} \\
    &= \int \log\Bigl(\frac{\gamma_{uh}}{\gamma_{uh} \circ F_1^{-1} \circ F_0}\, \frac{(\det \nabla F_1) \circ F_1^{-1} \circ F_0}{\det \nabla F_0} \Bigr)\,\D \gamma_{uh}\,.
\end{align*}
Since
\begin{align*}
    \nabla F_0(B_{uh})
    &= \sqrt 2\,\bigl(I - h\,\nabla^2 V(x-uh\,\nabla V(x)+\sqrt 2 \, B_{uh})\bigr)\,, \qquad \nabla F_1(B_{uh}) = \sqrt 2\,I\,,
\end{align*}
and $F_1^{-1}(Z) = (Z - x + h\,\nabla V(x))/\sqrt 2$,
\begin{align*}
    \KL(\rho_0 \mmid \rho_1)
    &= \E\Bigl[ \frac{\norm{(Z_{uh} - x + h\,\nabla V(x))/\sqrt 2}^2 - \norm{B_{uh}}^2}{2uh} \\
    &\qquad\qquad{} - \log \det\bigl(I - h\,\nabla^2 V(x - uh\,\nabla V(x) + \sqrt 2\,B_{uh})\bigr)\Bigr]\,,
\end{align*}
where $Z_{uh} \deq x-h\,\nabla V(x-uh\,\nabla V(x) + \sqrt 2\,B_{uh})+\sqrt 2\,B_{uh}$.
For the second term, we use the assumption $h\lesssim 1/\beta$ to argue that $\abs{\log \det(I-h\,\nabla^2 V(z))} \lesssim \beta dh$ for any $z\in\R^d$.
Then, since we only seek a crude bound, we can assert
\begin{align*}
    \KL(\rho_0 \mmid \rho_1)
    &\lesssim \frac{1}{uh}\,\bigl(h^2 \,\E[\norm{\nabla V(x-uh\,\nabla V(x) + \sqrt 2\,B_{uh}) - \nabla V(x)}^2] + dhu\bigr) + \beta dh \\
    &\lesssim \frac{\beta^2 h}{u}\,\E[\norm{-uh\,\nabla V(x) + \sqrt 2\,B_{uh}}^2] + d + \beta dh
    \lesssim \beta^2 h^3 u\,\norm{\nabla V(x)}^2 + \beta^2 dh^2 + d \\
    &\lesssim \beta^2 h^3\,\norm{\nabla V(x)}^2 + d\,.
\end{align*}
Putting the bounds together,
\begin{align*}
    \KL(\delta_x \hat P^u \mmid \delta_y P)
    &\lesssim \frac{\norm{x-y}^2}{h} + \beta^2 h^3 \,\norm{\nabla V(x)}^2 + d\,.
\end{align*}

\textbf{\underline{Completing the proof.}}
Let $\delta \in (0,1)$ be a small parameter to be chosen later. After invoking the joint convexity of the KL divergence, we apply our first argument for $u \le 1-\delta$ and our second argument for $u \ge 1-\delta$, yielding
\begin{align*}
    \KL(\delta_x \hat P \mmid \delta_y P)
    &\le \int_0^1 \KL(\delta_x \hat P^u \mmid \delta_y P) \, \D u \\
    &\lesssim \int_0^{1-\delta} \Bigl( \frac{\norm{x-y}^2}{(1-u)\,h} + \beta^2 h^3 \,\norm{\nabla V(x)}^2 + \beta^2 dh^2 \Bigr) \, \D u
    + \delta \,\Bigl( \frac{\norm{x-y}^2}{h} + \beta^2 h^3 \,\norm{\nabla V(x)}^2 + d\Bigr) \\
    &\lesssim \frac{\norm{x-y}^2}{h}\,\bigl(1+\log(1/\delta)\bigr) + \beta^2 h^3\,\norm{\nabla V(x)}^2 + \beta^2 dh^2 + \delta d\,.
\end{align*}
We conclude the proof by setting $\delta = \beta^2 h^2$.
\end{proof}

\subsection{Proof of Theorems~\ref{thm:rmlmc} and~\ref{thm:rmlmc_lsi}}\label{app:rmlmc}

\begin{proof}[Proof of Theorem~\ref{thm:rmlmc}]
    We apply Theorem~\ref{thm:kl-general} with the following parameters: $L = \exp(-\alpha h)$, $\gamma = \beta h$, $c = O(1/h)$ (Lemma~\ref{lem:ld_properties}), $c' = O(\log(1/(\beta h))/h)$, $b(x) = \beta h^{3/2}\,\norm{\nabla V(x)} + \beta d^{1/2} h$ (Lemma~\ref{lem:cross_reg_rmlmc}), and $\mc E_{\rm weak}$, $\mc E_{\rm strong}$ given in Lemma~\ref{lem:rmlmc_local}.
    
    Following the proof of Theorem~\ref{thm:lmc}, Theorems~\ref{thm:local_error} and~\ref{thm:kl-general} yield
    \begin{align*}
        W_2^2(\hat\mu_0\hat P^n, \pi)
        &\lesssim \exp(-\alpha nh)\,W_2^2(\hat\mu_0,\pi) + \bigl(n \wedge \frac{1}{\alpha h}\bigr)^2\,(\beta^4 h^6 G_n^2 + \beta^4 dh^5) \\
        &\qquad{} + \bigl(n \wedge \frac{1}{\alpha h}\bigr)\,(\beta^2 h^4 G_n^2 + \beta^2 dh^3)
    \end{align*}
    and
    \begin{align}\label{eq:rmd_kl_recursion}
        \KL(\hat\mu_0\hat P^n \mmid \pi)
        &\lesssim_{\log} \frac{\alpha W_2^2(\hat\mu_0,\pi)}{\exp(\alpha nh)-1} + \bigl(n \wedge \frac{1}{\alpha h}\bigr)\,(\beta^4 h^5 G_n^2 + \beta^4 dh^4) + \beta^2 h^3 G_n^2 + \beta^2 dh^2\,,
    \end{align}
    where $\lesssim_{\log}$ is used to suppress logarithms for simplicity.

    In the case $\alpha > 0$, we invoke Lemma~\ref{lem:recursive_grad} with $n_0=\infty$, $\msf A^2 = \widetilde O(\beta^2 h^3/\alpha)$, $\msf B^2 = O(d/\alpha)$, provided $h\le \widetilde O(1/(\beta \kappa))$.
    This yields $G_n^2 \lesssim \beta \kappa d$ for all $n$, and substitution into~\eqref{eq:rmd_kl_recursion} finishes the proof.

    In the case $\alpha = 0$, we invoke Lemma~\ref{lem:recursive_grad} with total iteration count $N = \widetilde \Theta(W^2/(\varepsilon^2 h))$, $n_0 = \Theta(1/(\beta h))$, $\msf A^2 = \widetilde O(\beta h^3)$, $\msf B^2 = \widetilde O(W^2)$, $\msf C^2 = \widetilde O(\beta^4 h^4 W^2/\varepsilon^2 + \beta^2 h^3)$, $\msf D^2 = \widetilde O(\beta W^2)$, provided
    \begin{align*}
        h \le \frac{1}{\beta}\, \widetilde O\Bigl(1 \wedge \frac{\beta^{1/2} W}{d^{1/2}} \wedge \frac{\varepsilon^{1/2}}{\beta^{1/4} W^{1/2}} \wedge \frac{\varepsilon^{2/3}}{d^{1/3}}\Bigr)\,.
    \end{align*}
    It yields $G_N^2 = \widetilde O(\beta d + \beta^2 W^2)$, and substituting this into~\eqref{eq:rmd_kl_recursion} finishes the proof.
\end{proof}

\begin{proof}[Proof of Theorem~\ref{thm:rmlmc_lsi}]
    We follow the proof of Theorem~\ref{thm:lmc_lsi}.
    Namely, by the R\'enyi weak triangle inequality (Corollary~\ref{kl:weak-triangle}), Theorem~\ref{thm:langevin_renyi_conv}, and Theorem~\ref{thm:kl-general},
    \begin{align*}
        \KL(\hat\mu_0 \hat P^n \mmid \pi)
        &\lesssim_{\log} d\exp(-\alpha n h) + n\,(\beta^4 h^5 G_n^2 + \beta^4 dh^4) + \beta nh\,(\beta^2 h^3 G_n^2 + \beta^2 dh^2)
    \end{align*}
    The total iteration count is taken to be $N = \widetilde\Theta(1/(\alpha h))$.
    We invoke Lemma~\ref{lem:recursive_grad} with $n_0 = 0$, $\msf C^2 = \widetilde O(\beta^3 \kappa h^4 + \beta^2 \kappa h^3)$, $\msf D^2 = \widetilde O(d)$, provided $h \le \widetilde O(1/(\beta \kappa^{1/2}))$.
    It yields $G_N^2 = \widetilde O(\beta d)$, and substituting this into the KL divergence bound completes the proof.
\end{proof}

\begin{remark}\label{rmk:rmd_weak_cvx}
    In the usual analysis of the randomized midpoint discretization, the strong error dominates.
    However, in the proof of Theorem~\ref{thm:rmlmc} in the weakly convex case, the weak error dominates and leads to the final $\widetilde O(\beta^{4/3} d^{1/3} W^{8/3}/\varepsilon^{10/3})$ rate.
    This suggests that if the weak error can be further reduced, e.g., by using more Picard iterations as suggested in~\cite{shen2019randomized}, it may be possible to further improve the rate to $\widetilde O(\beta d^{1/2} W^2/\varepsilon^3)$, although we do not pursue this here.
\end{remark}

    \small
    \addcontentsline{toc}{section}{References}
    \printbibliography{}
    
\end{document}